\documentclass[12pt]{article}
\usepackage{authblk} % Adds affiliation

\usepackage[T1]{fontenc}
\usepackage[utf8]{inputenc}
\usepackage[english]{babel}
\usepackage{enumerate,amsmath,amsthm}
\usepackage{amssymb}
\usepackage{bookman}
\usepackage{enumitem,etoolbox,hyperref}
\usepackage{tabu}
\usepackage{multirow}
\usepackage{graphicx} %% Obrázky
\usepackage{wrapfig} %% Obrázky v textu
\usepackage[font=small,labelfont=bf]{caption} % smaller font in figure labels
\graphicspath{ {Pictures/} } %% Path to pictures directory
\usepackage[section]{placeins} % FloatBarrier
\allowdisplaybreaks % Allows breaking equation in align in multiple pages
\usepackage[normalem]{ulem} % striked letters

\newcommand{\ru}[1]{\rule{0pt}{#1 em}}%changing 
\newcommand{\tupsingle}{\ru{3}} % change number in ru 

\input{mycommands.sty}

\usepackage{comment}

\newlength\titlebox 

\theoremstyle{definition}
\newtheorem{definition}{Definition}%[section]

\theoremstyle{plain}
\newtheorem{theorem}{Theorem}%[section]
\newtheorem{lemma}[theorem]{Lemma}

\newtheorem{proposition}[theorem]{Proposition}
\newtheorem{corollary}{Corollary}[theorem]

\theoremstyle{remark}
\newtheorem{remark}{Remark}
\newtheorem{example}{Example}

\input{amsbibstyle.sty}
\title{Mean distance in polyhedra}
\author[ ]{Dominik Beck}
\affil[ ]{\small Faculty of Mathematics and Physics, Charles University, Prague}
\affil[ ]{\textit {\href{mailto:beckd@karlin.mff.cuni.cz}{beckd@karlin.mff.cuni.cz}}}
\date{September 23, 2023}

\begin{document}
\maketitle
\begin{abstract}
Given any polyhedron from which we select two random points uniformly and independently, we show that all the moments of the distance between those points can be always written in terms of elementary functions. As an illustration, the mean distance is found in the exact form for all Platonic solids.
\end{abstract}

\tableofcontents

\section{Introduction} 
Let $K$ be a polyhedron\footnote{In this paper, a polyhedron is a three-dimensional polytope, not necessarily convex in general.}, from which we select two random points $X$ and $Y$ uniformly and independently. Let $L=\|X-Y\|$ be the distance between them and $L_{KK}^{(p)} = \expe{L^p |\, X \in K, Y \in K}$ its $p$-th statistical moment. Even-power moments are trivial to compute. The value $L_{KK}^{(p)}$ has been known in the exact form only for $K$ being a ball (trivial) or (for $p=1$) a unit cube \cite{robbinsCuLi}, known as the so called Robbins constant
\begin{equation}
\mathbb{E}\left[L\right]=\frac{4}{105} + \frac{17\sqrt{2}}{105}-\frac{2\sqrt{3}}{35}-\frac{\pi}{15}+\frac{1}{5}\arccoth\sqrt{2}+\frac{4}{5}\arccoth\sqrt{3} \approx 0.66170718.
\end{equation}
Recently, Bonnet, Gusakova, Th\"{a}le and Zaporozhets \cite{bonnet2021sharp} found a sharp optimal bound on the \emph{normalised mean distance} $\Gamma_{KK} = L_{KK}/V_1(K)$ in convex and compact $K$, where $V_1(K) = 2\oint_{\partial B_1(0)} \|\operatorname{proj}_{\uvect{r}} K\| \ddd\uvect{r}$ is the first intrinsic volume of $K$. A special case of their result in three dimensions gives $\frac{5}{28}<\Gamma_{KK}<\frac{1}{3}$.
\begin{comment}
The first intrinsic volume $V_1(K)$ can be expressed for convex $K$ as $V_1(K) = 2W(K)$, where $W(K)=\oint_{B_1(0)} |\operatorname{proj}_{\uvect{r}} K| \ddd\uvect{r}$ is the \emph{mean width} of $K$. For polyhedra, the value of $V_1(K)$ can be trivially expressed using external angles. Let $l_i$ be the length of the edge $E_i$ of $K$ and $\theta_i$ its corresponding external angle ($\theta_i = \pi - \delta_i$, where $\delta_i$ is the \emph{dihedral angle}), then $V_1(K) = \frac{1}{2\pi} \sum_i l_i \theta_i  $ (XXX: reference)
\end{comment}

As stated in \cite{bonnet2021sharp}, although the first intrinsic volume is easy to express in any polyhedron, number of examples for which an exact formula for $L_{KK}$ is available is rather limited. We will show that this might not be the case and indeed one can find $L_{KK}$ (and all natural moments $L_{KK}^{(p)}$) in an exact form easily for any $K$ being a polyhedron. The main result of our own investigation is thus the following theorem:

\begin{theorem}\label{mainthm}
For any given polyhedron, the mean distance between two of its inner points selected at random can be always expressed in terms of elementary functions of the location of its vertices and sides. The same holds for all other natural moments.
\end{theorem}
\begin{remark}
By elementary functions, we mean closed field of functions containing radicals, exponential, trigonometric and hyperbolic functions and their inverses.
\end{remark}

The theorem is solely based on the Crofton Reduction Technique (CRT), see \cite{dunbar1997average,ruben1973more}, which under certain conditions enables us to express $L_{KK}^{(p)}$ as some linear combination of $L_{AB}^{(p)} = \expe{L^p |\, X \in A, Y \in B}$ over domains $A$ and $B$ with smaller dimension than that of $K$. In fact very recently, using different methods, Ciccariello \cite{ciccariello2020chord} showed that the so called chord-length distribution, which is related with the distribution of $L$, can also be expressed in terms of elementary functions in any polyhedron $K$.

\begin{comment}
% XXX use only if the section on polygons is included
Also recently, Uve B\"{a}sel \cite{basel2021moments} expressed $L_{KK}^{(p)}$ for $K$ being a regular polygon in $\mathbb{R}^2$. We will also discuss how we can re-derive this result using Crofton Reduction Technique.
\end{comment}

\section{Preliminaries}
\subsection{Crofton Reduction Technique}

\begin{definition}
A polytope $A \subset \mathbb{R}^d$ of dimension $\dim A = a \in \{0,1,2,\ldots,d\}$ and $a-$volume $\volA$ is defined as a connected and finite union of $a$-dimensional simplices. We say a polytope is \textbf{flat} if $\dim \mathcal{A}(A) = \dim A$, where $\mathcal{A}(A)$ stands for the affine hull of $A$. Note that any polytope with $a=d$ is flat automatically.
\end{definition}

\begin{definition}
We denote $\mathcal{P}_a(\mathbb{R}^d)$ the set of flat polytopes of dimension $a$ in $\mathbb{R}^d$ and denote $\mathcal{P}(\mathbb{R}^d) = \bigcup_{0\leq a \leq d} \mathcal{P}_a(\mathbb{R}^d)$ the set of all flat polytopes in $\mathbb{R}^d$. Finally, we denote $\mathcal{P}_+(\mathbb{R}^d) = \mathcal{P}(\mathbb{R}^d)\setminus\mathcal{P}_0(\mathbb{R}^d)$ (flat polytopes excluding points).
\end{definition}

\begin{definition}
Let $A,B \in \mathcal{P}(\mathbb{R}^d)$ and $P: \mathbb{R}^d \times \mathbb{R}^d \to \mathbb{R}$, we denote $P_{AB} = \Ex \left[ P(X,Y) \, |\, X\in A,\right.$ $\left. Y \in B, \,\text{uniform and independent}\right]$. Whenever it is unambiguous, we write $P_{ab}$ where $a=\dim A$ and $b = \dim B$ instead of $P_{AB}$. If there is still ambiguity, we can add additional letters after as superscripts to distinguish between various mean values $P_{ab}$.
\end{definition}

\begin{proposition}
For any $A \in \mathcal{P}_a(\mathbb{R}^d)$ with $a> 0$, there exist \textbf{convex} $\partial_i A \in \mathcal{P}_{a-1}(\mathbb{R})$ (\textbf{sides} of $A$) such that $\partial A = \bigcup_i \partial_i A$ with pairwise intersection of $\partial_i A$ having $(a\!-\!1)$-volume equal to zero.
\end{proposition}
\begin{comment}
%Rataj: simpl. con. satisfied automatically due to convexity
\begin{remark}
Without loss of generality, we will always assume $\partial_i A$ are \textbf{simply connected} (if they are not, we split those to a finite union of simply connected ones).
\end{remark}
\end{comment}
\begin{remark}
The sides of three dimensional polytopes (polyhedra) are called \textbf{faces}.%(for those we also assume simple connectedness).
\end{remark}
\begin{definition}
Let $A \in \mathcal{P}_+(\mathbb{R}^d)$. Let $\uvect{n}_i$ be the outer normal unit vector of $\partial_i A$ in $\mathcal{A}(A)$, then we define a \textbf{signed distance} $h_C(\partial_i A)$ from a given point $C \in \mathcal{A}(A)$ to $\partial_i A$ as a scalar product $\langle v_i,\uvect{n}_i \rangle$, where $v_i = x_i - C$ and $x_i \in \partial_i A$ arbitrary. Note that if $A$ is convex, the signed distance coincides with the \textbf{support function} $h(A - C,\uvect{n}_i)$ defined for any convex domain $B$ as $h(B,\uvect{n}_i) = \sup_{b \in B} \langle b,\uvect{n}_i \rangle$.
\end{definition}

\begin{remark}\label{RemGeo}
The signed distance has another geometric interpretation. Put $C = O$ (the origin) and $r = 1 + \varepsilon$ (with $\varepsilon$ small). Denote $\int_{(B,A)} = \int_{B/A}-\int_{A/B}$, by linearity $\int_{B} = \int_{A} + \int_{(B,A)}$. Hence
\begin{equation}
    \vol{r A} = \int_{rA} \dd x = \int_{A} \dd x + \int_{(rA,A)} \dd x = \volA + \varepsilon \sum_i \vol{\partial_i A} h_O(\partial_i A) + O(\varepsilon^2),
\end{equation}
or in other words, $\dd \vol{rA}/\dd r |_{r=1} = \sum_i \vol{\partial_i A} h_O(\partial_i A)$ for $A$ arbitrary (possibly non-convex).
\end{remark}

\begin{definition}
\label{def:weights}
Let $A \in \mathcal{P}_+(\mathbb{R}^d)$ with $a = \dim A$. Even though $\partial A \notin \mathcal{P}(\mathbb{R}^d)$, we still define $P_{\partial A \, B}$ for a given point (called a scaling point) $C \in \mathcal{A}(A)$ as a weighted mean via the relation
\begin{equation}
    P_{\partial A\, B} = \sum_i w_i P_{\partial_i A\, B} 
\end{equation}
with weights $w_i$ equal to
\begin{equation}
    w_i = \frac{\vol{\partial_i A}}{a \volA} h_C(\partial_i A).
\end{equation}
\end{definition}

\begin{definition}
Let $P: \mathbb{R}^d \times \mathbb{R}^d \to \mathbb{R}$. We say $P$ is \textbf{homogeneous} of order $p \in \mathbb{R}$, if there exists $\tilde{P}: \mathbb{R}^d \to \mathbb{R}$ such that $P(x,y) = \tilde{P}(x-y)$ for all $x,y \in \mathbb{R}^d$ and $\tilde{P}(r v) = r^p \tilde{P}(v)$ for all $v \in \mathbb{R}^d$ and all $r > 0$. We write $p = \dim P$ (although $p$ might not be an integer). We say $P$ is \textbf{even} if $P(x,y)=P(y,x)$.
\end{definition}
\begin{remark}
Note that if $P$ is even, then $P_{AB}=P_{BA}$ for any domains $A$ and $B$.
\end{remark}
\begin{example}
If $P = L^p$, or more precisely $P(x,y) = L^p(x,y) = \|x-y\|^p$, then $P$ is even and homogeneous of $\dim P = p$ and with $\tilde{P}(x)=\|x\|^p$. Whenever $P = L^p$, we will use $P_{AB}$ and $L_{AB}^{(p)}$ interchangeably throughout this paper.
\end{example}
\begin{lemma}[Crofton Reduction Technique] Let $P: \mathbb{R}^d\times \mathbb{R}^d \to \mathbb{R}$ be homogeneous of order $p$ and $A,B \in \mathcal{P}(\mathbb{R}^d)$, then for any $C \in \mathcal{A}(A)\cap \mathcal{A}(B)$ holds
\begin{equation}
p P_{AB} = a (P_{\partial A\, B} - P_{AB}) + b (P_{A\,\partial B} - P_{AB}).
\end{equation}
\begin{figure}[h]
    \centering
     \includegraphics[height=0.18\textwidth]{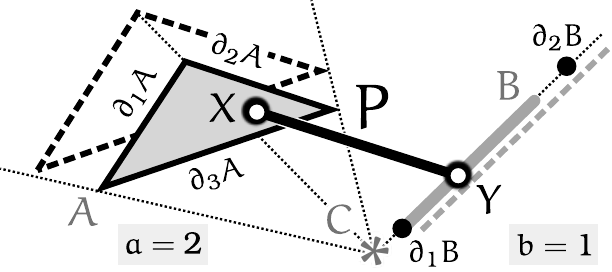}
    \caption{Crofton Reduction Technique}
    \label{fig:Crofton}
\end{figure}
\end{lemma}
\begin{proof}
The formula is a special case of the extension of the Crofton theorem by Ruben and Reed \cite{ruben1973more}, although it is fairly simple to derive directly. Let $r = 1 + \varepsilon$ and put $C=O$ (the origin) without loss of generality. The key is to express $P_{rA,rB}$ in two different ways:
\begin{itemize}
    \item By definition,
        \begin{align*}
            P_{rA,rB} & = \expe{P(X,Y) \, |\, X\in rA, Y \in rB, \,\text{uniform}} = \expe{P(rX',rY') \, |\, X'\in A, Y' \in B, \,\text{uniform}} \\
            & = r^p \,\expe{P(X',Y') \, |\, X'\in A, Y' \in B, \,\text{uniform}} = r^p P_{AB} = P_{AB} + \varepsilon p P_{AB} + O(\varepsilon^2).
        \end{align*}
    \item On the other hand,
        \begin{align*}
            & \vol{rA} \vol{rB} P_{rA,rB} = \vol{rA} \vol{rB}\expe{P(X,Y) \, |\, X\in rA, Y \in rB, \,\text{uniform}}\\
            &= \int_{rA} \int_{rB} P(x,y) \ddd x \dd y = \int_{A} \int_{B} P(x,y) \ddd x \dd y + \int_{(rA,A)} \int_{B} P(x,y) \ddd x \dd y  \\
            & + \int_{A} \int_{(rB,B)} P(x,y) \ddd x \dd y + \int_{(rA,A)} \int_{(rB,B)} P(x,y) \ddd x \dd y  \\
            & = \volA \volB P_{AB} + \varepsilon \volB \sum_i \vol{\partial_iA} h_O(\partial_iA) P_{\partial_i\, A B} \\
            & + \varepsilon\volA \sum_j \vol{\partial_jB} h_O(\partial_jB) P_{A \, \partial_jB} +  O(\varepsilon^2).
        \end{align*}
\end{itemize}
Comparing the $\varepsilon$ terms of both expressions and using Remark \ref{RemGeo}, we get the statement of the lemma. If either of $\dim A$ or $\dim B$ is zero, the lemma holds too. 
\end{proof}

\subsection{Auxiliary integrals}
Apart from rotations and reflections, any integral encountered in our paper has the following form ($h>0$)
\begin{equation}
I^{(p)}_f(h,\zeta,\gamma)=\int_{D(\zeta,\gamma)} f(x,y) \left(h^2 + x^2 + y^2\right)^{p/2} \ddd x \dd y,
\end{equation}
where $D(\zeta,\gamma)$ is the \emph{fundamental triangle domain} with vertices $[0,0]$, $[\zeta,0]$, $[\zeta,\zeta\tan\gamma]$ ($\zeta > 0$, $0<\gamma<\pi/2$) and $f(x,y)$ is a polynomial in $x$ and $y$ of degree at most two (quadratic in $x$ and $y$). We can write $f(x,y) = a_{00} + a_{10}x + a_{01}y + a_{20} x^2+a_{11} x y + a_{02} y^2$. Based on $x$ and $y$ terms, we have the following
\begin{equation}
\begin{split}
I^{(p)}_f(h,\zeta,\gamma)& =
a_{00} I^{(p)}_{00}(h,\zeta,\gamma) + a_{10} I^{(p)}_{10}(h,\zeta,\gamma) + a_{01} I^{(p)}_{01}(h,\zeta,\gamma)\\
& + a_{20} I^{(p)}_{20}(h,\zeta,\gamma) + a_{11} I^{(p)}_{11}(h,\zeta,\gamma) + a_{02} I^{(p)}_{02}(h,\zeta,\gamma),
\end{split}
\end{equation}
where
\begin{equation}
I^{(p)}_{ij}(h,\zeta,\gamma) =\int_{D(\zeta,\gamma)} x^i y^j \left(h^2 + x^2 + y^2\right)^{p/2} \ddd x \dd y.
\end{equation}
The parameters of those integrals are not optimal. We only need to consider the case $h=1$. To see this, denote 
\begin{equation}
I^{(p)}_{ij}(q,\gamma) = I^{(p)}_{ij}(1,q,\gamma) =\int_{D(q,\gamma)} x^i y^j \left(1 + x^2 + y^2\right)^{p/2} \ddd x \dd y.
\end{equation}
By scaling $x \to h x$, $y \to h y$, we can write
\begin{equation}
I^{(p)}_{ij}(h,\zeta,\gamma) =h^{2+p+i+j} I^{(p)}_{ij}(\zeta/h,\gamma).
\end{equation}
Thus, with $q = \zeta/h$,
\begin{equation}
\begin{split}
I^{(p)}_f(h,\zeta,\gamma)& = h^{2+p} \big{[}
a_{00} I^{(p)}_{00}(q,\gamma) + a_{10}h I^{(p)}_{10}(q,\gamma) + a_{01}h I^{(p)}_{01}(q,\gamma)\\
& + a_{20}h^2 I^{(p)}_{20}(q,\gamma) + a_{11}h^2 I^{(p)}_{11}(q,\gamma) + a_{02}h^2 I^{(p)}_{02}(q,\gamma) \big{]}.
\end{split}
\end{equation}
Selected values of the auxiliary integrals $I^{(p)}_{ij}(q,\gamma)$ and the methods how we can derive them are found in the appendix.

\section{Irreducible configurations of the reduction technique}
Repeated use of the latter Lemma we refer to as the Crofton Reduction Technique. To find the moments of $L$, we put $P = L^p$. In the first step, we choose $A=K$ and $B=K$. Since the affine hulls of both $A$ and $B$ fill the whole space $\mathbb{R}^d$, any point in $\mathbb{R}^d$ can be selected for $C$. We then employ the reduction technique to express $P_{AB}$ in $P_{A'B'}$ where $A'$ and $B'$ have smaller dimensions then $A$ and $B$. The pairs of various $A'$ and $B'$ we encounter we call \textbf{configurations}. The process is repeated until the affine hull intersection of $A'$ and $B'$ is empty. In that case, we have reached an \textbf{irreducible} configuration. From now on in the rest of our paper if not explicitely stated, $d=3$. The following configurations are irreducible in $\mathbb{R}^3$:
\begin{itemize}
    \item $A$ is a polygon and $B$ is a point
    \item $A$ and $B$ are two skew line segments
    \item $A$ and $B$ are two parallel polygons or one polygon and one line segment parallel to it
\end{itemize}
\subsection{Polygon and a point}
In the first case, $A$ is a polygon and $B$ a point. Denote $\operatorname{proj}_A(\cdot)$ a perpendicular projection onto $\mathcal{A}(A)$. Next, denote $h$ the distance between $B$ and $\mathcal{A}(A)$. With $k = x - \operatorname{proj}_A B$ where $x \in A$, we have that
\begin{equation}
L_{AB}^{(p)} = \frac{1}{\volA} \int_A (h^2 + k^2)^{p/2} \ddd k
\end{equation}
is expressible in terms of elementary functions. To see this, write and $\partial_i A, i=1,\ldots,s$ for the sides of the polygon $A$, oriented such that the path through the vertices of $A$ is counterclockwise. Then, by inclusion/exclusion, and switching to polar coordinates
\begin{equation}
L^{(p)}_{AB} = \frac{1}{\volA} \sum_{i=1}^s\int_{T_i} (h^2+k^2)^{p/2}\ddd k = \frac{1}{\volA} \sum_{i=1}^s\int_{\alpha_i}^{\beta_i}\int_{0}^{h_i/\cos\varphi} (h^2+r^2)^{p/2}  r\ddd r \dd\varphi
\end{equation}
where $T_i$ is a signed triangle whose one vertex is the point $\operatorname{proj}B$ and the other two vertices are the consecutive endpoints of $\partial_i A$. Rescaling the vector $k$ by $h$, we can rewrite each integral in the sum in a standard way
\begin{equation}
\int_{T_i} (h^2+k^2)^{p/2}\ddd k = h^{2+p}\left(I^{(p)}_{00}\left(h_i/h,\beta_i\right) -I^{(p)}_{00}\left(h_i/h,\alpha_i\right)\right)
\end{equation}
where $\alpha_i$ and $\beta_i$ are their respective polar angles (in counterclockwise order) and $h_i$ is the perpendicular distance from $\operatorname{proj}B$ to $\partial_i A$. The polar angles are defined such that the closest point on the line $\mathcal{A}(\partial_i A)$ from $\operatorname{proj}B$ has its value equal to zero, increasing in the clockwise direction (see Figure \ref{fig:PolyPoint}). The integral is positive if the angle of consecutive vertices of the polygon increased and negative if it decreased.  
\begin{figure}[h]
    \centering
 \includegraphics[height=0.3\textwidth]{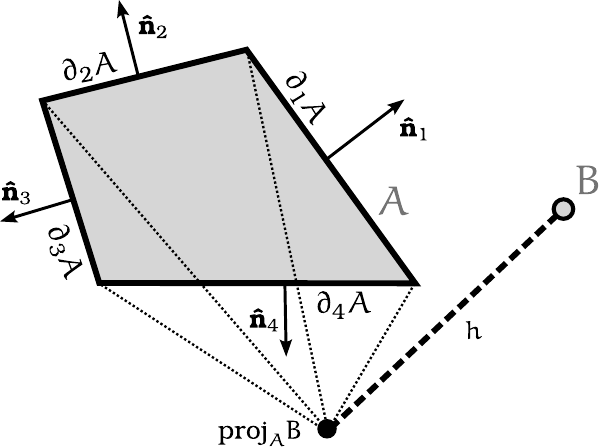}
    \caption{Point-polygon triangle decomposition}
    \label{fig:PolyPoint}
\end{figure}

\noindent
Summing all contributions, we finally get our \textbf{point-polygon formula}
\begin{equation}\label{LABpolypo}
L^{(p)}_{AB} = \frac{h^{2+p}}{\volA} \sum_{i=1}^s \left(I^{(p)}_{00}\left(h_i/h,\beta_i\right) -I^{(p)}_{00}\left(h_i/h,\alpha_i\right)\right).
\end{equation}

\subsection{Two skewed line segments}
The second case is in fact equivalent with the first. If $A$ and $B$ are two skew line segments, write $A-B = \{ u - v\, |\, u\in A, v \in B\}$ (which is a parallelogram). Then, by shifting, we get for any homogeneous $P$, denoting $O$ as the origin
\begin{equation}\label{ShiftEq}
P_{AB} = P_{O,A-B}.
\end{equation}
So we can always reduce this problem to the polygon and a point problem treated before.

\subsection{Overlap formula}
From now on, in case of no ambiguity, we often write simply $\operatorname{proj}$ instead of $\operatorname{proj}_A$ for the perpendicular projection operator onto $\mathcal{A}(A)$.

\begin{proposition}
Let $A,B \in \mathcal{P}_+(\mathbb{R}^3)$, $a=2$, $b \in \{1,2\}$, such that $\mathcal{A}(A)$ and $\mathcal{A}(B)$ are parallel with perpendicular separation vector $s$ having length $h = \|s\|$. Let $P(x,y)$ be homogeneous and let $k$ be a vector lying in the projection plane $\mathcal{A}(A)$, then
\begin{equation}
P_{AB} = \frac{1}{\volA \volB} \int_A \int_B \tilde{P}(x-y) \ddd x \dd y = \frac{1}{\volA \volB} \int_{\mathbb{R}^2} \tilde{P}(s+k) \vol{A\cap (\operatorname{proj}B + k)} \ddd k.
\end{equation}
Especially, for $P = L^p$, we get $L_{AB}^{(p)} = \frac{1}{\volA \volB} \int_{\mathbb{R}^2} (h^2 + k^2)^{p/2} \vol{A\cap (\operatorname{proj}B + k)} \ddd k$.
\end{proposition}
\begin{figure}[h]
    \centering
     \includegraphics[height=0.20\textwidth]{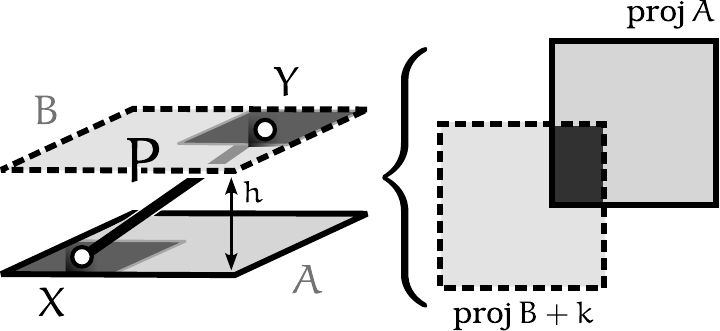}
    \caption{Overlap formula}
    \label{fig:Overlap}
\end{figure}
\begin{remark}
Since $\vol{A\cap (\operatorname{proj}B + k)}$ is a piece-wise polynomial function of degree at most two on polygonal domains, the double integral is expressible in terms of elementary functions for any integer $p > -3$.
\end{remark}

\begin{proof}
Let $A,B \subset \mathbb{R}^d$ be compact domains with dimensions $a$ and $b$, respectively, and $P$ be even homogeneous function $\mathbb{R}^d \to \mathbb{R}$ of order $p>-3$. Let $A_B(z) = A \cap (B+z)$, $c=\max_{z \in \mathbb{R}^d} \dim A_B(z)$ and $C=\{z \in \mathbb{R}^{d}\, |\, \dim A_B(z) = c\}$. Then, by substitution $y=x+z$ and by Fubini's theorem, \begin{equation}
P_{AB} = \frac{1}{\volA \volB}\int_{A}\int_{B} \tilde{P}(x-y) \ddd x \dd y = \frac{1}{\volA \volB}\int_{C} \tilde{P}(z) \operatorname{vol}A_B(z) \ddd z.
\end{equation}
When $A,B$ are parallel in $d=3$, the proposition follows.
\end{proof}

\begin{definition}
An \emph{overlap diagram} of $A$ (face) and $B$ (parallel face or edge) consists of partitions of $\mathbb{R}^2$ into open subdomains $D_i$ where $\vol{A\cap (\operatorname{proj}B + k)}$ can be expressed as a single polynomial function in $k$ of degree at most two. Since $A$ and $B$ are polygons (or a polygon and a polyline, respectively), these subdomains $D_i$ are also polygonal (polylinial, respectively). When there is no ambiguity, we denote those subdomains $D_i$ by numbers corresponding to the number of sides of the polygon $A\cap (\operatorname{proj}B + k)$ of intersection in case $B$ is a face, or the number of line segments of the polyline $A\cap (\operatorname{proj}B + k)$ of intersection when $B$ is an edge, respectively
\end{definition}

\begin{remark}
For brevity, we often write $\vol{A \cap \operatorname{proj}B + k}$ instead of $\vol{A \cap (\operatorname{proj}B + k)}$.
\end{remark}

\section{General and special polyhedra}
\subsection{General polyhedra}
\begin{theorem}\label{ThmBasic}
Let $K \in \mathcal{P}(\mathbb{R}^3)$, $E_j, F_k, j\in\{ 1,\ldots,e\}, k\in\{ 1,\ldots,f\}$, denote the edges and faces of $K$, respectively, and let $P: \mathbb{R}^d\times \mathbb{R}^d \to \mathbb{R}$ be even and homogeneous of order $p>-3$. Then
\begin{equation}
\label{eq:thmgen}
    P_{KK} = \frac{2}{(6+p)(5+p)} \bigg{(}\sum_{k<k'} P_{F_k F_{k'}} w_{F_k F_{k'}} + \sum_j P_{K E_j} w_{K E_j}\bigg{)},
\end{equation}
with weights $w_{AB}$ (independent on $P$ and $p$) given as follows: We fix $C$ any point in $\mathbb{R}^3$, $C_k$ any point on $\mathcal{A}(F_k)$ and $D_j$ any point on $\mathcal{A}(E_j)$. Denote $F_{k(j)}, F_{k'(j)}$ the two faces on which lies the edge $E_{j}$, then
\begin{align}
    w_{F_k F_{k'}} & = \frac{\vol{F_k}\vol{F_{k'}}}{\volKsq}(h_C( F_k) h_{C_k}(F_{k'}) + h_C(F_{k'}) h_{C_{k'}}( F_k)), \\
    w_{K E_j} & = \frac{\volK \vol{E_j}}{\volKsq} \left(h_C(F_{k(j)}) h_{C_{k(j)}}(D_j)+h_C(F_{k'(j)}) h_{C_{k'(j)}}(D_j)\right).
\end{align}
\end{theorem}
\begin{proof}
Use the Crofton Reduction Technique twice.
\end{proof}
\begin{remark}
Note that the weights are not unique as they depend on the position of scaling points.
\end{remark}
\begin{remark}
Note that if $P = L^p$ and for any polyhedron $K$, all terms $P_{AB}$ in Equation \eqref{eq:thmgen} are either further reducible or $A$ and $B$ are parallel. In both cases, we can express $L^{(p)}_{AB}$ in terms of auxiliary integrals. Theorem \ref{mainthm} follows.
\end{remark}

\subsection{Nonparallel polyhedra}
For polyhedra which have some special properties, we are able to further reduce Theorem \ref{ThmBasic} above.
\begin{definition}
Let $\mathcal{P}^*(\mathbb{R}^3)$ denote the set of all polyhedra having the property that affine hulls of any of its three faces of meet at a single point. We call them \textbf{nonparallel polyhedra}. Also, we denote $\mathcal{P}_{\mathrm{convex}}^*(\mathbb{R}^3)$ a subset of those which are convex.
\end{definition}

\begin{theorem}\label{ThmNonpar}
Let $K \in \mathcal{P}^*(\mathbb{R}^3)$ and $V_i, E_j, F_k, i\in\{ 1,\ldots,v\}, j\in\{ 1,\ldots,e\}, k\in\{ 1,\ldots,f\}$, denote the vertices, edges and faces of $K$, respectively, and $P$ be even and homogeneous of order $p>-3$. Then
\begin{equation}
    P_{KK} = \frac{12}{(6+p)(5+p)(4+p)(3+p)} \bigg{(}\sum_{\substack{ik \\ V_i \notin \mathcal{A}(F_k)}} P_{V_i F_k} w_{V_i F_k} + \sum_{\substack{j<j' \\ \mathcal{A}(E_j)\cap \mathcal{A}(E_{j'}) = \emptyset}} P_{E_j E_{j'}} w_{E_j E_{j'}}\bigg{)}
\end{equation}
for some weights $w_{AB}$ which are independent on $P$ and $p$.
\end{theorem}
\begin{proof}
Since no pair of faces nor edges are parallel, we can further reduce $P_{F_k F_{k'}}$ and $P_{K E_j}$ from Theorem \ref{ThmBasic} twice. The weights are easily computable by choosing appropriate scaling points. Note that again the weights are not unique and depend on the selection of those scaling points. For example, let $C \in \mathcal{A}(F_k) \cap \mathcal{A}(F_{k'})$, $k<k'$. Then by CRT, we get
\begin{equation}
P_{F_k F_{k'}} = \frac{2}{4+p} \left( P_{\partial F_k F_{k'}} + P_{F_k \partial F_{k'}} \right).
\end{equation}
Note that both $P_{\partial F_k F_{k'}}$ and $P_{F_k \partial F_{k'}}$ are expressible as some linear combination of $P_{E_i F_k}$ with $\mathcal{A}(E_i) \cap \mathcal{A}(F_k)$. Finally, we can reduce even this term. Let $C' \in \mathcal{A}(E_i) \cap \mathcal{A}(F_k)$, then 
\begin{equation}
P{E_i F_k} = \frac{1}{2+p}\left( P_{\partial E_i F_k} + 2 P_{E_i \partial F_k} \right),
\end{equation}
which in turn is expressible as a linear combination of $P_{V_i F_k}$ and $P_{E_i E_{i'}}$ with $V_i \notin \mathcal{A}(F_k)$ and $\mathcal{A}(E_i) \cap \mathcal{A}(E_{i'})=\emptyset$. The reduction of terms $P_{KE_j}$ is similar.
\end{proof}

\subsection{Nonparallel convex polyhedra}
In the case of convex nonparallel polyhedra, we can find very simple relations for weights $w_{AB}$. First, we start with a known formula (a simple generalisation of \cite[Eq. 34]{kingman1969random} in $d=3$)
\begin{lemma}
\label{lem:king}
Let $K$ be a convex and compact set in $\mathbb{R}^3$ and $P$ even homogeneous of order $p>-3$, then
\begin{equation}
P_{KK} = \frac{1/\volKsq}{(4+p)(3+p)} \oint_{B_1(0)} \int_{\Sigma_\perp(\uvect{r})} \tilde{P}(\uvect{r})\, \Lambda_Q^{4+p}(\uvect{r}) \ddd Q \dd\uvect{r},
\end{equation}
where the integration in carried over all directions $\uvect{r}$ on the unit sphere $B_1(0)$ with surface measure $\dd \uvect{r}$ having $\oint \dd \uvect{r} = 4\pi$ and over all points $Q$ on plane $\Sigma_\perp(\uvect{r})$ passing through the origin and being perpendicular to $\uvect{r}$. Finally, $\Lambda_Q(\uvect{r})$ denotes the length of the intersection of $K$ and the line passing through $Q$ in the direction of unit vector $\uvect{r}$.
\end{lemma}
\begin{corollary}
By Fubini's theorem,
\begin{equation}
    \lim_{p \rightarrow -3^+} (3+p)P_{KK} = \frac{1}{\volK} \oint_{B_1(0)} \tilde{P}(\uvect{r})\, \dd\uvect{r}
\end{equation}
\end{corollary}
\begin{remark}
    Similar formulae are available in higher dimensions as well.
\end{remark}

\begin{theorem}
\label{th:nonpcon}
Let $K \in \mathcal{P}_{\mathrm{convex}}^*(\mathbb{R}^3)$ and $V_i, E_j, F_k, P, w_{AB}$ be defined exactly as in Theorem \ref{ThmNonpar}. Denote $h_{ik}$ the distance between $V_i$ and $\mathcal{A}(F_k)$, similarly denote $h_{jj'}$ the distance between $O$ and $\mathcal{A}(E_j-E_{j'})$ and $\theta_{jj'}$ the angle between $E_j$ and $E_{j'}$ (on the same plane under perpendicular projection). Then
\begin{equation}
\begin{split}
    P_{KK} & = \frac{12/\volK}{(6+p)(5+p)(4+p)(3+p)} \bigg{(}\sum_{\substack{ik \\ V_i \notin \mathcal{A}(F_k)}} P_{V_i F_k} n_{V_i F_k} \vol{F_k}h_{ik} \,\,\\
    &+ \!\!\!\! \sum_{\substack{j<j' \\ \mathcal{A}(E_j)\cap \mathcal{A}(E_{j'}) = \emptyset}} \! P_{E_j E_{j'}} n_{E_j E_{j'}} \vol{E_j}\vol{E_{j'}}h_{jj'}\sin\theta_{jj'}\bigg{)},
\end{split}
\end{equation}
with weights $n_{AB}$ satisfying the following projection relation: Choose a direction $\uvect{n}$ and project $K$ onto a plane perpendicular to it. Then the weights corresponding to vertex-face pairs which overlap and to pairs of edges which cross add up to one. Symbolically,
\begin{equation}\label{EqWeightsN}
    1 = \sum_{\substack{ik \\ V_i \notin \mathcal{A}(F_k)}} n_{V_i F_k}1_{\uvect{n} \in V_iF_k} + \sum_{\substack{j<j' \\ \mathcal{A}(E_j)\cap \mathcal{A}(E_{j'}) = \emptyset}} n_{E_j E_{j'}} 1_{\uvect{n}\in E_j E_{j'}},
\end{equation}
where $1_{\uvect{n}\in AB}$ = 1 if there are points $x\in A, y\in B$ such that $x-y$ is parallel with $\uvect{n}$, otherwise $1_{\uvect{n}\in AB} = 0$. On top of that, the extreme case where one of the points $x,y$ lies on the boundary of $A$ or $B$ leaves the value $1_{\uvect{n}\in AB}$ undefined.
\end{theorem}
\begin{proof}
The key observation is that the weights are independent of the choice of the function $P$ as long it is even and homogeneous. Let $\varepsilon>0$ be small and $\uvect{n}$ be a fixed unit vector, $\Omega_\varepsilon = \pi \varepsilon^2 + O(\varepsilon^4)$ then denotes a solid angle with apex half angle equal to $\varepsilon$. We define $R^{(p)}(\varepsilon,\uvect{n},x,y) = \| x-y\|^p$ if the angle between $\uvect{r}$ and $x-y$ is smaller than $\varepsilon$ and zero otherwise. Alternatively, denote $C(\varepsilon,V,\uvect{n})$ a double-cone region whose vertex is $V$, apex angle $2\varepsilon$ and the axis has direction $\uvect{n}$. Then for any domains $A$ and $B$,
\begin{equation}
R^{(p)}_{AB}(\varepsilon,\uvect{n}) = \int_{A}\int_{B} R^{(p)}(\varepsilon,\uvect{n},x,y) \ddd y \ddd x = \int_{A} \int_{B \cap C(\varepsilon,x,\uvect{n})}\| x-y\|^p \ddd y \ddd x.
\end{equation}
Note that $R$ is even and homogeneous in $x,y$ of order $p$. Hence, by Lemma \ref{lem:king},
\begin{equation}\label{EqLimR}
    \lim_{p \rightarrow -3^+} (3+p)R^{(p)}_{KK}(\varepsilon,\uvect{n}) = \frac{1}{\volK} \oint_{B_1(0)} \tilde{R}^{(-3)}(\varepsilon,\uvect{n},\uvect{r})\, \dd\uvect{r} = \frac{2\Omega_\varepsilon}{\volK} + O(\varepsilon^4).
\end{equation}
On the other hand, via Theorem \ref{ThmNonpar},
\begin{equation}
    \lim_{p \rightarrow -3^+} (3+p)R^{(p)}_{KK}(\varepsilon,\uvect{n}) = 2 \, \bigg{(}\sum_{\substack{ik \\ V_i \notin \mathcal{A}(F_k)}} R^{(-3)}_{V_i F_k}(\varepsilon,\uvect{n}) w_{V_i F_k} + \sum_{\substack{j<j' \\ \mathcal{A}(E_j)\cap \mathcal{A}(E_{j'}) = \emptyset}} R^{(-3)}_{E_j E_{j'}}(\varepsilon,\uvect{n}) w_{E_j E_{j'}}\bigg{)}.
\end{equation}
We are able to express $R^{(-3)}_{V_i F_k}(\varepsilon,\uvect{n})$ and $R^{(-3)}_{E_j E_{j'}}(\varepsilon,\uvect{n})$ in the following way:
\begin{equation}
R^{(-3)}_{V_i F_k}(\varepsilon,\uvect{n}) = \frac{\Omega_\varepsilon 1_{\uvect{n} \in V_iF_k}}{\vol{F_k}h_{ik}} + O(\varepsilon^4), \qquad
R^{(-3)}_{E_j E_{j'}}(\varepsilon,\uvect{n}) = \frac{\Omega_\varepsilon 1_{\uvect{n}\in E_j E_{j'}}}{\vol{E_j}\vol{E_{j'}}h_{jj'}\sin\theta_{jj'}} + O(\varepsilon^4).
\end{equation}
We will prove only the first equality as the other one is get simply by shifting (edge-edge configuration is equivalent to vertex-face configuration by means of Equation \eqref{ShiftEq}). Let $V_i \notin \mathcal{A}(F_k)$ for some (polygonal) face $F_k$ and vertex $V_i$. We denote by $r$ the distance between $V_i$ and the point of intersection of $\mathcal{A}(F_k)$ and the line passing through the vertex $V_i$ in the direction of $\uvect{n}$. Note that the perpendicular distance $h_{ik}$ between $V_i$ and $\mathcal{A}(F_k)$ is independent on the direction of $\uvect{n}$. Since $\varepsilon$ is small, we can write
\begin{equation}
R^{(p)}_{V_i F_k}(\varepsilon,\uvect{n}) = \frac{1}{\vol{F_k}} \int_{F_k} R^{(p)}(\varepsilon,\uvect{n},x,V_i) \ddd x = \frac{r^p \vol{F_k \cap C(\varepsilon,V_i,\uvect{n})}}{\vol{F_k}} + O(\varepsilon^4)
\end{equation}
Assuming $\uvect{n} \in V_iF_k$, the point of intersection lies in the interior of $F_k$. Hence, for sufficiently small $\varepsilon$, we get that $V_i \cap C(\varepsilon,V_i,\uvect{n})$ is an ellipse with area
\begin{equation}
\vol{V_i \cap C(\varepsilon,V_i,\uvect{n})} = 1_{\uvect{n} \in V_iF_k}\frac{\Omega_\varepsilon r^3}{h_{ik}} + O(\varepsilon^4)
\end{equation}
Hence
\begin{equation}
R^{(p)}_{V_i F_k}(\varepsilon,\uvect{n}) = \frac{1}{\vol{F_k}} \int_{F_k} R^{(p)}(\varepsilon,\uvect{n},x,V_i) \ddd x = \frac{r^{3+p} \Omega_\varepsilon 1_{\uvect{n} \in V_iF_k}}{\vol{F_k} h_{ik}} + O(\varepsilon^4)
\end{equation}
when $p = -3$, the dependency on $r$ vanishes. Finally, comparing this relation with \eqref{EqLimR}, we get the equation for weights
\begin{equation}\label{EqWeights}
    \frac{1}{\volK} = \sum_{\substack{ik \\ V_i \notin \mathcal{A}(F_k)}} \frac{w_{V_i F_k} 1_{\uvect{n} \in V_iF_k}}{\vol{F_k}h_{ik}} + \sum_{\substack{j<j' \\ \mathcal{A}(E_j)\cap \mathcal{A}(E_{j'}) = \emptyset}} \frac{w_{E_j E_{j'}} 1_{\uvect{n}\in E_j E_{j'}}}{\vol{E_j}\vol{E_{j'}}h_{jj'}\sin\theta_{jj'}}
\end{equation}
valid for any $\uvect{n}$ for which all the values $1_{\uvect{n} \in AB}$ are well defined. Lastly, defining auxiliary weight $n_{AB}$ via
\begin{equation}
w_{V_iF_k} = \frac{\vol{F_k} h_{ik} n_{V_i F_k}}{\volK}, \qquad w_{E_jE_{j'}} = \frac{\vol{E_j}\vol{E_{j'}} h_{jj'} n_{E_j E_{j'}} \sin \theta_{jj'}}{\volK},
\end{equation}
we get
\begin{equation}
    1 = \sum_{\substack{ik \\ V_i \notin \mathcal{A}(F_k)}} n_{V_i F_k}1_{\uvect{n} \in V_iF_k} + \sum_{\substack{j<j' \\ \mathcal{A}(E_j)\cap \mathcal{A}(E_{j'}) = \emptyset}} n_{E_j E_{j'}} 1_{\uvect{n}\in E_j E_{j'}}.
\end{equation}
This constrain alone enables us to determine admissible weights for any convex nonparallel polyhedron via set of linear equations got by varying the direction of $\uvect{n}$.
\end{proof}

\subsection{Tetrahedron}
As an example, we express the random distance moments in the case of a tetrahedron. There are two possible ways how a planar projection of a tetrahedron could look like (almost surely) with respect to the number of intersecting pairs of edges and vertices/faces in the projection (see Figure \ref{fig:TetraOri}).

\begin{figure}[h]
    \centering
 \includegraphics[width=0.3\textwidth]{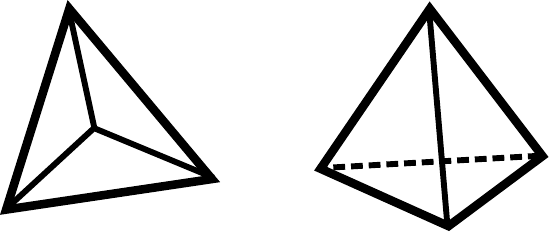}
    \caption{Tetrahedron projection orientations}
    \label{fig:TetraOri}
\end{figure}

In the first case, one vertex covers one face. There are no other vertex/face nor edge/edge coverings. Similarly, in the second case, one edge is covered by another edge. There are again no other coverings. Thus, in order to satisfy Equation \eqref{EqWeightsN}, we can simply choose $n_{V_iF_k} = n_{E_jE_{j'}} = 1$ for each vertex $V_i$, face $F_k$ and edges $E_j, E_{j'}$. Hence, by Theorem \ref{th:nonpcon},
\begin{equation}
\begin{split}
    P_{KK} & = \frac{12/\volK}{(6+p)(5+p)(4+p)(3+p)} \bigg{(}\sum_{\substack{ik \\ V_i \notin \mathcal{A}(F_k)}} P_{V_i F_k} \vol{F_k}h_{ik} \,\,\\
    &+ \!\!\!\! \sum_{\substack{j<j' \\ \mathcal{A}(E_j)\cap \mathcal{A}(E_{j'}) = \emptyset}} \! P_{E_j E_{j'}} \vol{E_j}\vol{E_{j'}}h_{jj'}\sin\theta_{jj'}\bigg{)}.
\end{split}
\end{equation}

%XXX continue here

\section{Mean distance in regular polyhedra}
To apply our general method, we shall derive the mean distance in all five regular polyhedra (also known as Platonic solids). Among those solids, only the tetrahedron is nonparallel convex, so Theorem \ref{th:nonpcon} applies here. Hence, we used this theorem to find the mean distance in a general (possibly irregular) tetrahedron. In the rest of our paper, we calculate the mean distance in all other Platonic solids (including the regular tetrahedron again). Since they are an example of parallel polyhedra, we cannot use Theorem \ref{th:nonpcon} due to presence of irreducible configurations of type face-face and edge-face. However, we can still calculate the mean distance. The calculation relies the Overlap formula as well as the symmetries of those regular polyhedra which drastically reduce the number of configurations needed to be considered. Throughout this section, we denote $\nu$ the area of (any) face of $K$ and $l$ the length (any) of its edge. These values makes sense because $K$ is a regular polyhedron. Furthermore, $\phi = \frac{1+\sqrt{5}}{2}$ is the Golden ratio.

\subsection{Regular tetrahedron}
Let us have $P$ even homogeneous of order $p$ dependent on two random points picked from $K$ a regular tetrahedron given by vertices $V_1[1,0,0]$, $V_2[0,1,0]$, $V_3[0,0,1]$, $V_4[1,1,1]$, edges connecting them $E_{12}$, $E_{13}$, $E_{14}$, $E_{23}$, $E_{24}$, $E_{34}$ ($E_{ij} = \overline{V_iV_j},$ where $i\neq j$) and with opposite faces $F_1$, $F_2$, $F_3$, $F_4$. Note that $\volK = 1/3$, so if we want to express the mean of $P$ in a tetrahedron of unit volume, we must multiply all our results by $3^{p/3}$. We put $P = L^p$. For the definition of various mean values $P_{ab} = L_{ab}^{(p)}$, see Figure \ref{fig:LinePickingTetrahedron}. We also included the position of the scaling point $C$ in cases reduction is possible. The arrows indicate which configurations reduce to which.
\begin{figure}[h]
    \centering
     \includegraphics[width=0.85\textwidth]{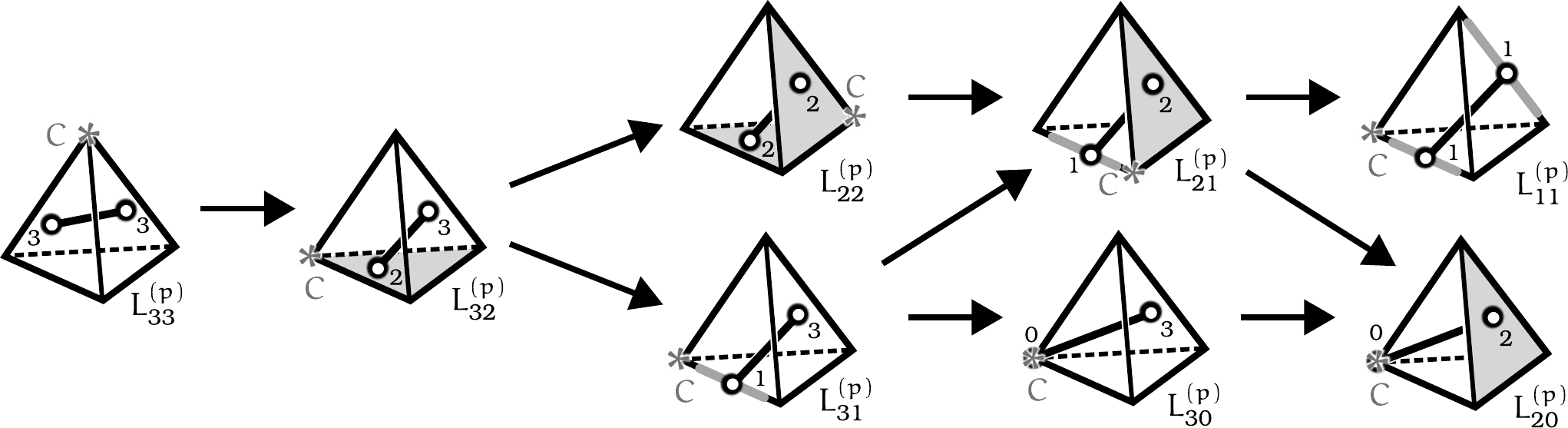}
    \caption{All different $L_{ab}^{(p)}$ configurations encountered for $K$ being a regular tetrahedron}
    \label{fig:LinePickingTetrahedron}
\end{figure}

\noindent
Based on CRT, let us write our reduction system of equations:
\begin{align*}
    p P_{33} & = 6(P_{32}-P_{33})\\
    p P_{32} & = 3(P_{22}-P_{32})+2(P_{31}-P_{32})\\
    p P_{22} & = 4(P_{21}-P_{22})\\
    p P_{31} & = 3(P_{21}-P_{31})+1(P_{30}-P_{31})\\
    p P_{21} & = 2(P_{11}-P_{21})+1(P_{20}-P_{21})\\
    p P_{30} & = 3(P_{20}-P_{30}),
\end{align*}
where $P_{33} = P_{KK}$ and by symmetry, we can put $P_{32} = P_{KF_1}, P_{31} = P_{KE_{12}}, P_{30} = P_{KV_1}, P_{22} = P_{F_1F_2}, P_{21} = P_{F_4V_{14}}, P_{20} = P_{F_4V_4}, P_{11} = P_{E_{12}E_{34}}$. This linear system has a solution
\begin{equation}\label{P33Tetr}
    P_{33} = \frac{72 (3 P_{11}+2 P_{20})}{(6+p)(5+p)(4+p)(3+p)}.
\end{equation}
To demonstrate our technique for irreducible configurations, we derive the value of $L_{33}$. That means, we choose $P = L^p$ with $p=1$.
\phantomsection
\subsubsection{\texorpdfstring{L\textsubscript{20}}{L20}}
%\addcontentsline{toc}{subsubsection}{L20}
By \eqref{LABpolypo}, by symmetry and using $\vol{F_4} = \sqrt{3}/2, h_1= 1/\sqrt{6} , h=2/\sqrt{3}$,
\begin{equation}
P_{20} = L^{(p)}_{F_4V_4} =  \frac{6h^{2+p}}{\vol{F_4}} I^{(p)}_{00}\left(\frac{\sqrt{2}}{4},\frac{\pi}{3}\right).
\end{equation}
Using the recursion relations,
\begin{equation}\label{J1sq2o4}
I^{(1)}_{00}\left(\frac{\sqrt{2}}{4},\frac{\pi}{3}\right) = \frac{1}{16 \sqrt{2}}-\frac{\pi}{9}+\frac13\arcsin\sqrt{\frac{2}{3}}+\frac{25}{96 \sqrt{2}} \arcsinh\frac{1}{\sqrt{3}},
\end{equation}
so, further using $\arcsin \sqrt{2/3} = \arctan\sqrt{2}$ and $\arcsinh (1/\sqrt{3}) = \tfrac12\ln 3$,
\begin{equation}
L_{20} = \frac{\sqrt{2}}{3}-\frac{32 \pi }{27}+\frac{32}{9} \arctan\sqrt{2}+\frac{25 \ln 3}{18 \sqrt{2}}.  
\end{equation}

\phantomsection
\subsubsection{\texorpdfstring{L\textsubscript{11}}{L11}}
%\addcontentsline{toc}{subsubsection}{L11}
By shifting \eqref{ShiftEq}, we get $L_{11} = L_{AB}$, where $B$ is the origin and $A$ is a parallelogram with vertices $[1, 0, -1], [0, 1, -1]$, $[-1, 0, -1]$, $[0, -1, -1]$. Therefore, by the point-polygon formula \eqref{LABpolypo} with $h=1$ and $\volA = 2$,
\begin{equation}
L_{11} = \frac{8 h^3}{\volA} I^{(1)}_{00}\left(\frac{\sqrt{2}}{2},\frac{\pi}{4}\right),
\end{equation}
where by recurrences,
\begin{equation}
I^{(1)}_{00}\left(\frac{\sqrt{2}}{2},\frac{\pi}{4}\right) = \frac{1}{6 \sqrt{2}}-\frac{\pi}{12}+\frac{1}{3}\arcsin\frac{1}{\sqrt{3}}+\frac{7 }{12 \sqrt{2}}\arcsinh\frac{1}{\sqrt{3}}.
\end{equation}
Hence, writing $
\arcsin\left(1/\sqrt{3}\right) =\frac{\pi }{2}-\arctan\sqrt{2}$ and $\arcsinh (1/\sqrt{3}) = \tfrac12\ln 3$,
\begin{equation}
    L_{11} = \frac{\sqrt{2}}{3}+\frac{\pi }{3}-\frac{4}{3} \arctan\sqrt{2}+\frac{7 \ln 3}{6 \sqrt{2}}.
\end{equation}

\phantomsection
\subsubsection{\texorpdfstring{L\textsubscript{33}}{L33}}
%\addcontentsline{toc}{subsubsection}{L33}
Substituting $L_{20}$ and $L_{11}$ into \eqref{P33Tetr} with $P = L^p$ and $p=1$, we get, finally
\begin{equation*}
    L_{33} = \frac{3}{35} (3 L_{11}+2L_{20}) = \frac{\sqrt{2}}{7}-\frac{37 \pi }{315}+\frac{4}{15} \arctan \sqrt{2}+\frac{113 \ln 3}{210 \sqrt{2}}.
\end{equation*}
Or, re-scaling to the unit volume tetrahedron, 
\begin{equation}
    L_{33}\big{|}_{\volK = 1} =\sqrt[3]{3} \left(\frac{\sqrt{2}}{7}-\frac{37 \pi }{315}+\frac{4}{15} \arctan \sqrt{2}+\frac{113 \ln 3}{210 \sqrt{2}}\right) \approx 0.72946242,
\end{equation}
which is an \textbf{exact} expression of an approximation given by Weisstein \cite{weissteinTeLi}. Similarly, we would proceed in the case of the second moment:
\begin{equation}
    L_{33}^{(2)}\big{|}_{\volK = 1}=\frac{9}{10\sqrt[3]{3}}.
\end{equation}
Alternatively, we can express the result as the normalised mean distance $\Gamma_{KK}$. Since $V_1(K) = 3\sqrt{2} \arccos\left(-\frac13\right)/\pi$ (see Table \ref{tab:IntrinPlaton} with $a=\sqrt{2}$), we have
\begin{equation}
    \Gamma_{KK} = \frac{L_{33}}{V_1(K)}=\frac{\pi}{3\sqrt{2} \arccos\left(-\frac13\right)} \left(\frac{\sqrt{2}}{7}-\frac{37 \pi }{315}+\frac{4}{15} \arctan \sqrt{2}+\frac{113 \ln 3}{210 \sqrt{2}}\right) \approx 0.19601928.
\end{equation}
Of course, using the reduction technique, we could get other moments (replacing $I_{ij}^{(1)}$ by $I_{ij}^{(p)}$ integrals), and even for a general edge-length tetrahedron.

\subsection{Cube}
We present a re-derivation of the Robbins constant for $K$ being a cube via our method. Here, we demonstrate the Crofton Reduction Technique including the overlap formula. A standard way how to choose its vertices is $[0,0,0]$, $[1,0,0]$, $[0,1,0]$, $[0,0,1]$, $[0,1,1]$, $[1,0,1]$, $[1,1,0]$, $[1,1,1]$. Under this choice, the edge length $l=1$, face area $\nu =1$ and the volume $\volK = 1$. We put $P = L^p$. For the definition of various mean values $P_{ab} = L_{ab}^{(p)}$, see Figure \ref{fig:LinePickingCube}. Note that in $L_{21r}$ configuration, we let $B$ to be four edges (boundary of an opposite face) rather than just one edge.

\begin{comment}
\newcommand{\cubeheight}{0.11\textwidth}
\renewcommand{\arraystretch}{1.5}
   \begin{table}[h]
\centering
\setlength{\tabcolsep}{6pt} % reduces collumn width
\setlength\jot{1pt} % math expressions smaller space between rows
\begin{tabular}{ c c c c c }
 \tupsingle
\includegraphics[height=\cubeheight]{C_L33.pdf}  &
\includegraphics[height=\cubeheight]{C_L32.pdf} &
\includegraphics[height=\cubeheight]{C_L22e.pdf} &
\includegraphics[height=\cubeheight]{C_L22r.pdf} &
\includegraphics[height=\cubeheight]{C_L31.pdf} \\  \ru{7}
\includegraphics[height=0.12\textwidth]{C_L21v.pdf}  &
\includegraphics[height=\cubeheight]{C_L21r.pdf} &
\includegraphics[height=\cubeheight]{C_L30.pdf} &
\includegraphics[height=\cubeheight]{C_L20.pdf} &
\includegraphics[height=\cubeheight]{C_L11.pdf}
\end{tabular} 
\caption{All different $L_{ab}^{(p)}$ configurations encountered for $K$ being a cube}
\label{CubeConfigs}
\end{table}
\end{comment}

\begin{figure}[h]
    \centering
     \includegraphics[width=0.90\textwidth]{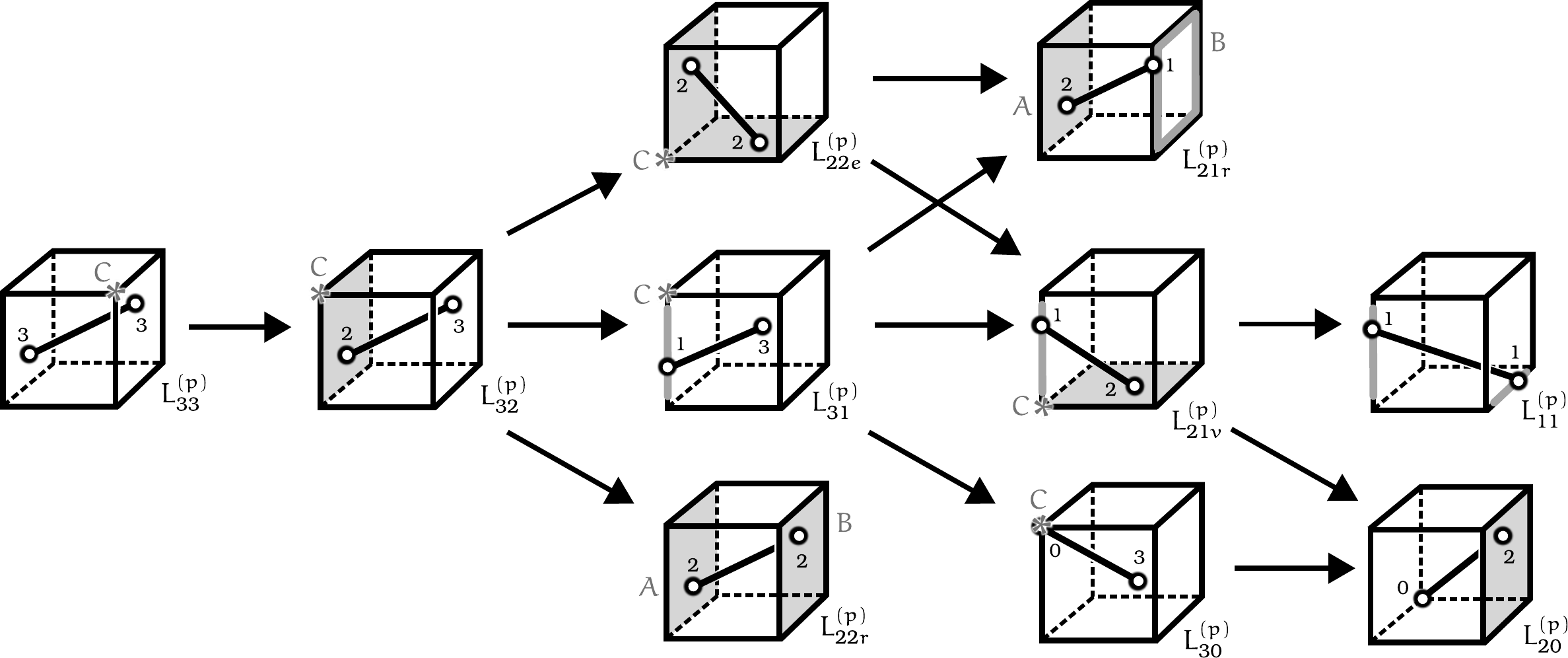}
    \caption{All different $L_{ab}^{(p)}$ configurations encountered for $K$ being a cube}
    \label{fig:LinePickingCube}
\end{figure}

\noindent
Performing the reduction, we get the set of equations, where
\begin{align*}
p P_{33}& =6 (P_{32}-P_{33}),\\
p P_{32} & = 3(P_{22} - P_{32})+2(P_{31}-P_{32}),\\
p P_{22e} & = 4 (P_{21}'-P_{22e}),\\
p P_{31} & = 3 (P_{21}-P_{31}) + 1(P_{30}-P_{31}),\\
p P_{21v} & = 2 (P_{11}-P_{21v}) + 1(P_{20}-P_{21v}),\\
p P_{30} & = 3 (P_{20}-P_{30})
\end{align*}
with
\begin{align*}
    P_{22} & = \tfrac23 P_{22e}+\tfrac13 P_{22r},\\
    P_{21} & = \tfrac13 P_{21v}+\tfrac23 P_{21r},\\
    P_{21}' & = \tfrac12 P_{21v}+\tfrac12 P_{21r}.
\end{align*}
Solving the system, we get
\begin{equation}
P_{33} = \frac{72 (P_{11}+P_{20})}{(3+p) (4+p) (5+p) (6+p)}+\frac{48 P_{21r}}{(4+p) (5+p) (6+p)}+\frac{6 P_{22r}}{(5+p) (6+p)}.
\end{equation}
When $p=1$, we get for the mean distance
\begin{equation}\label{Eq:Cube}
L_{33} = \frac{1}{35} (3 L_{11}+3 L_{20}+8 L_{21r}+5 L_{22r}).
\end{equation}

\phantomsection
\subsubsection{\texorpdfstring{L\textsubscript{20}}{L20}}
%\addcontentsline{toc}{subsubsection}{L20}
Without loss of generality, we can write $L_{20} = L_{AB}$, where $A$ is the cube's upper face defined as a square with vertices $[0,0,1]$, $[1,0,1]$, $[1,1,1]$, $[0,1,1]$ and $B$ is the origin $[0,0,0]$. Domains $A$ and $B$ are separated by distance $h=1$. The face $A$ is having area $\volA = 1$. By \eqref{LABpolypo} and by symmetry,
\begin{equation}
L_{20} = \frac{2}{\volA} I^{(1)}_{00}\left(1,\frac{\pi}{4}\right).
\end{equation}
Using recurrence relations (see Table \ref{tab:auxil} in Appendix),
\begin{equation}
L_{20} = \frac{1}{\sqrt{3}}-\frac{\pi }{18}+\frac{4}{3} \arccoth\sqrt{3}.
\end{equation}

\phantomsection
\subsubsection{\texorpdfstring{L\textsubscript{11}}{L11}}
%\addcontentsline{toc}{subsubsection}{L11}
The value $L_{11}$ can be defined as a mean distance between egde $E_1 = \overline{[0,0,0][0,1,0]}$ and edge $E_2 = \overline{[0,1,1][1,1,1]}$. Shifting $E_1$ by vector $-E_2$ (See shifting relation \eqref{ShiftEq}), we can rewrite this as $L_{11} = L_{AB}$, where again $A$ is the upper face of the cube and $B$ is the origin. Hence $L_{11} = L_{20}$.

\phantomsection
\subsubsection{\texorpdfstring{L\textsubscript{21r}}{L21r}}
%\addcontentsline{toc}{subsubsection}{L22}
Since the reduction technique cannot be applied on $AB$ being parallel, we use the overlap formula with $A$ being one face of the cube. In case of $L_{21r}$, the other domain $B$ is an opposite edge. By symmetry, we can add to this edge also three other edges opposite to $A$ (see Figure \ref{fig:LinePickingCube}). Hence, $B$ is a boundary of the face opposite to $A$ with length $\volB = 4$. Let $k=(x,y)$ then $\operatorname{proj}A = \operatorname{proj}B$ is a square with vertices $[\frac12,\frac12]$, $[-\frac12,\frac12]$, $[-\frac12,-\frac12]$, $[\frac12,-\frac12]$. In order to $\operatorname{proj} A$ and $\operatorname{proj} B+k$ have nonzero intersection, $k$ must be confined in the region $-1<x<1 \land -1<y<1$. By symmetry, we can chose $k$ to lie in the fundamental triangle domain $D(1,\pi/4)$ (we then multiply the values by $8$).

\begin{figure}[h]
    \centering     \includegraphics[height=0.22\textwidth]{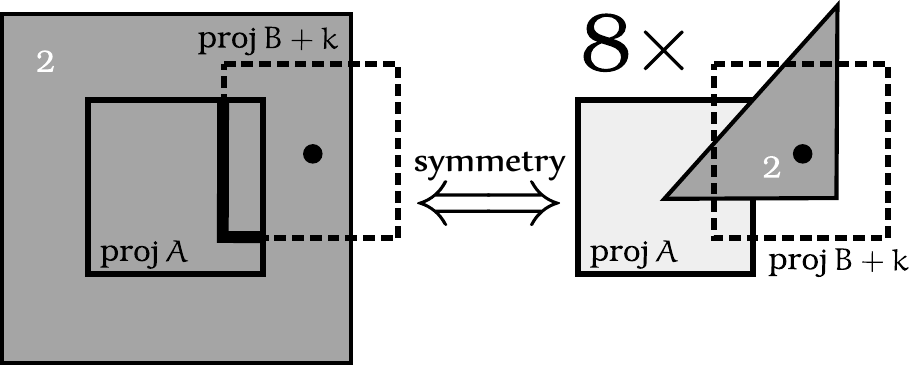}
    \caption{Overlap of the opposite face and edges of a cube}
    \label{fig:C_over_wire}
\end{figure}

\noindent
Setting up the integral,
\begin{equation}
L_{21r} = \frac{8}{\volA \volB} \int_D \sqrt{h^2 + k^2}\vol{A \cap \operatorname{proj}B + k} \ddd k,
\end{equation}
where $h = 1, \volA = 1 ,\volB = 4$ and $D = D(1,\pi/4)$ is a domain in Figure \ref{fig:C_over_wire} on the right (labeled with the number $2$). In this domain, we can write for the length of the polyline of intersection
\begin{equation}
\vol{A \cap \operatorname{proj} B + k} = 2-x-y,
\end{equation}
which gives us in terms of our auxiliary integrals
\begin{equation}
L_{21r} = \frac{8}{\volA \volB}\Bigg{[}
2I^{(1)}_{00}\left(1,\frac{\pi}{3}\right)
- I^{(1)}_{10}\left(1,\frac{\pi}{4}\right)
-I^{(1)}_{01}\left(1,\frac{\pi}{4}\right)\Bigg{]}
\end{equation}
Via recursions (see Table \ref{tab:auxil} in Appendix), we get
\begin{equation}
    L_{21r} = \frac{7}{6 \sqrt{2}}-\frac{1}{\sqrt{3}}-\frac{\pi }{9}+\frac{1}{4} \arccoth\sqrt{2}+\frac{5}{3} \arccoth\sqrt{3}.
\end{equation}

\phantomsection
\subsubsection{\texorpdfstring{L\textsubscript{22r}}{L22r}}
%\addcontentsline{toc}{subsubsection}{L22}
Again, we use the overlap formula for $AB$ being opposite faces. By symmetry, we again integrate $\vol{A \cap \operatorname{proj} B + k}$ over one 
eighth of all positions of $k$ (see Figure \ref{fig:C_over_full}).

\begin{figure}[h]
    \centering
     \includegraphics[height=0.22\textwidth]{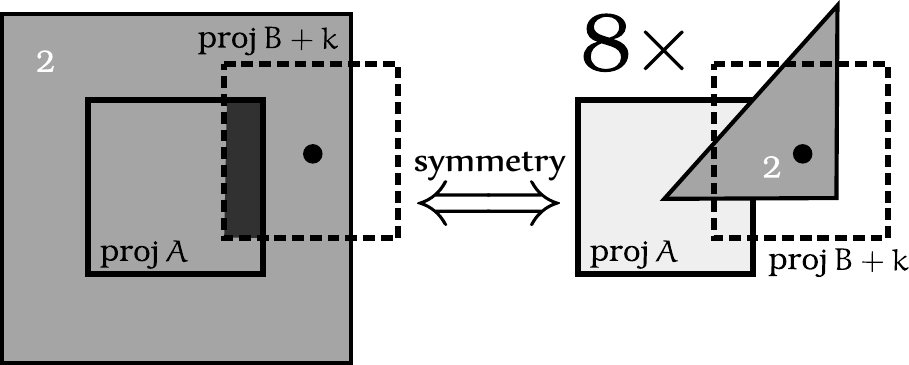}
    \caption{Overlap of the opposite faces of a cube}
    \label{fig:C_over_full}
\end{figure}

\noindent
Setting up the integral,
\begin{equation}
L_{22r} = \frac{8}{\volA \volB} \int_D \sqrt{h^2 + k^2}\vol{A \cap \operatorname{proj}B + k} \ddd k,
\end{equation}
where $h = 1, \volA = \volB = 1$ and $D = D()$ is a fundamental triangle domain (labeled $2$ in Figure \ref{fig:O_over_full} on the right). In this domain, we have for the polygon of intersection
\begin{equation}
\vol{A \cap \operatorname{proj} B + k} = (1-x)(1-y),
\end{equation}
and therefore
\begin{equation}
L_{22r} = \frac{8}{\volA \volB}\Bigg{[}
I^{(1)}_{00}\left(1,\frac{\pi}{4}\right) - I^{(1)}_{10}\left(1,\frac{\pi}{4}\right) - I^{(1)}_{01}\left(1,\frac{\pi}{4}\right) + I^{(1)}_{11}\left(1,\frac{\pi}{4}\right)
\Bigg{]}.
\end{equation}
Going through all recursions, we get, after simplifications
\begin{equation}
L_{22r} = \frac{4}{15}+\frac{\sqrt{2}}{5}-\frac{4}{5 \sqrt{3}}-\frac{2 \pi }{9}+\arccoth\left(\sqrt{2}\right)+\frac{4}{3} \arccoth\sqrt{3}.
\end{equation}

\phantomsection
\subsubsection{\texorpdfstring{L\textsubscript{33}}{L33}}
%\addcontentsline{toc}{subsubsection}{L33}
Putting everything together by using \eqref{Eq:Octa}, we finally arrive at the Robin's constant
\begin{equation}
    L_{33} = \frac{4}{105}+\frac{17 \sqrt{2}}{105}-\frac{2 \sqrt{3}}{35}-\frac{\pi }{15}+\frac{1}{5} \arccoth\sqrt{2}+\frac{4}{5} \arccoth\sqrt{3}  \approx 0.66170718.
\end{equation}

\subsection{Regular octahedron}
A standard way how to select vertices of an regular octahedron the vertices is $\left[\pm 1,0,0 \right]$, $\left[\pm 1,0,0 \right]$, $\left[\pm 1,0,0 \right]$. Under this choice, the edge length is $l=\sqrt{2}$, the area of each face is $\nu = \sqrt{3}/2$ and the volume of $K$ is $\volK = 4/3$. Again, we put $P = L^p$. For the definition of various mean values $P_{ab} = L_{ab}^{(p)}$, see Figure \ref{fig:LinePickingOctahedron}. We also included the position of the scaling point $C$ in cases when the reduction is possible.

\begin{comment}
\newcommand{\octaheight}{0.13\textwidth}
\renewcommand{\arraystretch}{1.5}
   \begin{table}[h]
\centering
\setlength{\tabcolsep}{2pt} % reduces collumn width
\setlength\jot{1pt} % math expressions smaller space between rows
\begin{tabular}{ c c c c c c }
% \hline
 \tupsingle
\includegraphics[height=\octaheight]{O_L33.pdf}  &
\includegraphics[height=\octaheight]{O_L32.pdf} &
\includegraphics[height=\octaheight]{O_L22e.pdf} &
\includegraphics[height=\octaheight]{O_L22r.pdf} &
\includegraphics[height=\octaheight]{O_L22v.pdf} &
\includegraphics[height=\octaheight]{O_L31.pdf} \\ %\cline{5-6}
%\hline
 \ru{7}
\includegraphics[height=\octaheight]{O_L21v.pdf}  &
\includegraphics[height=\octaheight]{O_L21r.pdf} &
\includegraphics[height=\octaheight]{O_L30.pdf} &
\includegraphics[height=\octaheight]{O_L20.pdf} &
\includegraphics[height=\octaheight]{O_L11.pdf}%\multicolumn{2}{|c|}{\includegraphics[height=\octaheight]{O_L11.pdf}
%\includegraphics[height=\octaheight]{O_L11tran.pdf}
%} \\ \cline{5-6}
%\hline
\end{tabular} 
\caption{All different $L_{ab}^{(p)}$ configurations encountered for $K$ being regular octahedron}
\label{OctaConfigs}
\end{table}    
\end{comment}

\begin{figure}[h]
    \centering
     \includegraphics[width=0.95\textwidth]{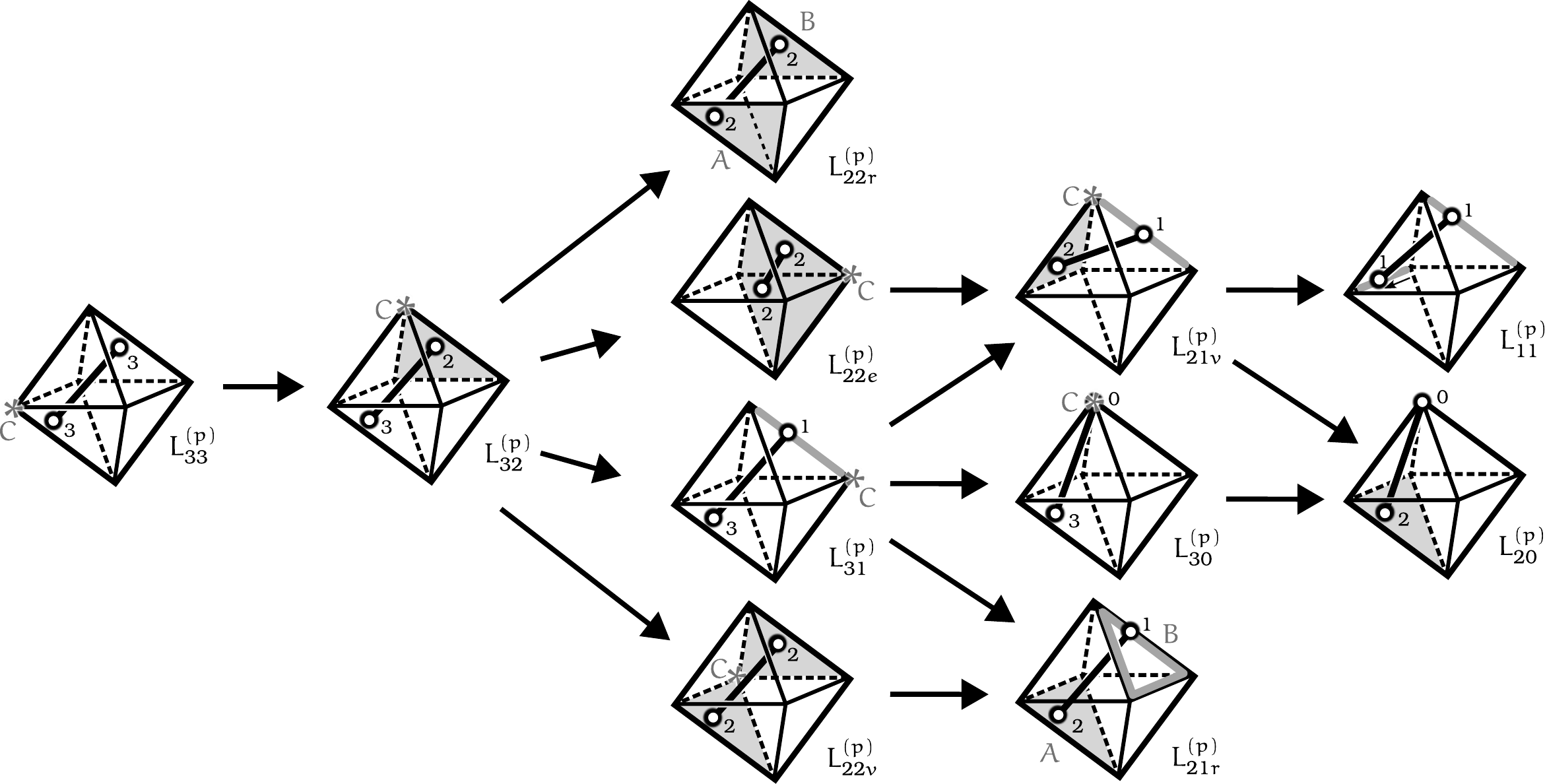}
    \caption{All different $L_{ab}^{(p)}$ configurations encountered for $K$ being a regular octahedron}
    \label{fig:LinePickingOctahedron}
\end{figure}

\noindent
Performing the reduction, we get the set of equations, where
\begin{align*}
p P_{33}& =6 (P_{32}-P_{33}),\\
p P_{32} & = 3(P_{22} - P_{32})+2(P_{31}-P_{32}),\\
p P_{22e} & = 4 (P_{21v}-P_{22e}),\\
p P_{22v} & = 4 (P_{21r}-P_{22v}),\\
p P_{31} & = 3 (P_{21}-P_{31}) + 1(P_{30}-P_{31}),\\
p P_{21v} & = 2 (P_{11}-P_{21v}) + 1(P_{20}-P_{21v}),\\
p P_{30} & = 3 (P_{20}-P_{30})
\end{align*}
with
\begin{align*}
    P_{22} & = \tfrac14 P_{22e}+\tfrac{1}{4}P_{22r}+\tfrac{1}{2} P_{22v},\\
    P_{21} & = \tfrac12 P_{21v}+\tfrac12 P_{21r}.
\end{align*}
Solving the system, we get
\begin{equation}
P_{33} = \frac{72(P_{20}+P_{11})}{(3+p)(4+p)(5+p)(6+p)} + \frac{54 P_{21r}}{(4+p)(5+p)(6+p)} + \frac{9 P_{22r}}{2(5+p)(6+p)}.
\end{equation}
When $p=1$, we get for the mean length
\begin{equation}\label{Eq:Octa}
L_{33} = \frac{3}{140} (4L_{20}+4 L_{11}+12L_{21r}+5L_{22r}).
\end{equation}

\phantomsection
\subsubsection{\texorpdfstring{L\textsubscript{20}}{L20}}
%\addcontentsline{toc}{subsubsection}{L20}
More precisely, we can write $L_{20} = L_{AB}$, where face $A$ has vertices $[-1,0,0]$, $[0,-1,0]$, $[0,0,-1]$ and $B = [0,0,1]$ (see Figure \ref{fig:LinePickingOctahedron}). Vertex $B$ is separated from $\mathcal{A}(A)$ by distance $h=2/\sqrt{3}$. By \eqref{LABpolypo} and by symmetry,
\begin{equation}
L_{20} = \frac{2h^3}{\volA} \left(I^{(1)}_{00}\left(\frac{\sqrt{2}}{2},\frac{\pi}{3}\right) -I^{(1)}_{00}\left(\frac{\sqrt{2}}{4},\frac{\pi}{3}\right)\right),
\end{equation}
where $\volA = \nu = \sqrt{3}/2$ is the area of $A$. Using recurrence relations on $I^{{(1)}}_{00}(\cdot,\cdot)$,
\begin{equation}\label{J1sq2o2p3}
I^{(1)}_{00}\left(\frac{\sqrt{2}}{2},\frac{\pi}{3}\right) = \frac{1}{4}-\frac{\pi}{36}+\frac{7 \arcsinh 1}{12 \sqrt{2}} = \frac{1}{4}-\frac{\pi}{36}+\frac{7 \arccoth\sqrt{2}}{12 \sqrt{2}}.
\end{equation}
The other $I^{(1)}_{00}$ integral is already given by \eqref{J1sq2o4} (we just write $\arcsin\sqrt{2/3} = \pi/2 - \arccot \sqrt{2}$), hence
\begin{equation}
L_{20} = \frac{8}{9}-\frac{\sqrt{2}}{9}-\frac{8 \pi
   }{27}-\frac{25 \ln 3}{54
   \sqrt{2}}+\frac{32}{27} \arccot\sqrt{2}+\frac{28}{27}
   \sqrt{2} \arccoth\sqrt{2}
\end{equation}

\phantomsection
\subsubsection{\texorpdfstring{L\textsubscript{11}}{L11}}
%\addcontentsline{toc}{subsubsection}{L11}
By shifting \eqref{ShiftEq}, we get $L_{11} = L_{AB}$, where $B$ is the origin and $A$ is a parallelogram with vertices $[1, 0, 1]$, $[0, 1, 1]$, $[0, 2, 0]$, $[1, 1, 0]$. Therefore, by the point-polygon formula \eqref{LABpolypo} with $h = 2/\sqrt{3}$ and $\volA = \sqrt{3}$,
\begin{equation}
L_{11} = \frac{h^3}{\volA} \left(4I^{(1)}_{00}\left(\frac{\sqrt{2}}{4},\frac{\pi}{3}\right) +2I^{(1)}_{00}\left(\frac{\sqrt{2}}{2},\frac{\pi}{3}\right)\right),\end{equation}
Hence, since the $I$'s are already given by \eqref{J1sq2o4} and \eqref{J1sq2o2p3}, we get, simplifying,
\begin{equation}
L_{11} = \frac{4}{9}+\frac{\sqrt{2}}{9}+\frac{4 \pi
   }{27}+\frac{25 \ln 3}{54
   \sqrt{2}}-\frac{32}{27} \arccot\sqrt{2}+\frac{14}{27}
   \sqrt{2} \arccoth\sqrt{2}
\end{equation}

\phantomsection
\subsubsection{\texorpdfstring{L\textsubscript{21r}}{L21r}}
%\addcontentsline{toc}{subsubsection}{L22}
Since the reduction technique cannot be applied on $AB$ being parallel, we use the overlap formula with $A$ being one face of the octahedron. By symmetry, we can choose $B$ as all three opposite edges to $A$ instead of just one, the mean value stays the same (see Figure \ref{fig:LinePickingOctahedron}). This choice makes the overlap formula simpler. To compute $\vol{A \cap \operatorname{proj} B + k}$, we slide the projection $B$ across $A$. To get $L_{21r}$, we then integrate over the length of their intersection with respect to all vectors $k$. By symmetry, we can integrate over just one sixth of all sliding domains (see Figure \ref{fig:O_over_wire} for our overlap diagram, in which white numbers represent the number of line segments in the $AB$ projection intersection with respect to position of the shift vector $k$ -- black dot).
\begin{figure}[h]
    \centering     \includegraphics[height=0.24\textwidth]{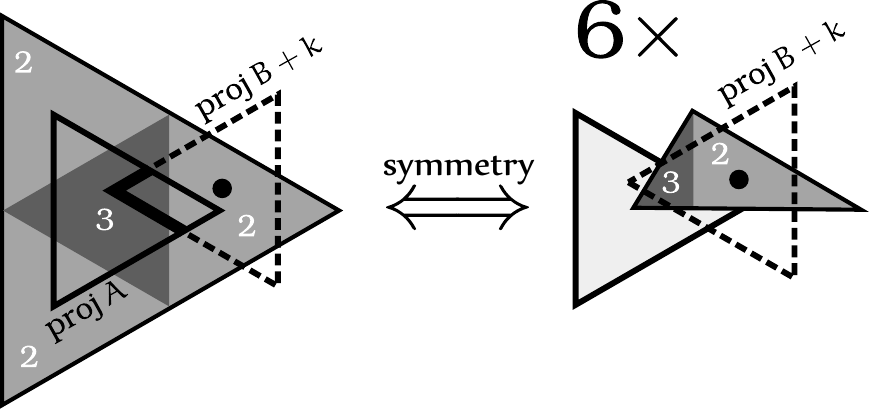}
    \caption{Overlap of the opposite face and edges of an octahedron}
    \label{fig:O_over_wire}
\end{figure}

\noindent
Hence, setting up the integral,
\begin{equation}
L_{21r} = \frac{6}{\volA \volB} \int_D \sqrt{h^2 + k^2}\vol{A \cap \operatorname{proj}B + k} \ddd k,
\end{equation}
where $h = 2/\sqrt{3}, \volA = \sqrt{3}/2 ,\volB = 3\sqrt{2}$ and $D$ is a domain in Figure \ref{fig:O_over_wire} on the right consisted of two subdomains $D_j$ where $j \in \{2,3\}$ denotes the number of line segments of the intersection $A \cap (\operatorname{proj} B+k)$, which is a polyline. We have $D = D_2 \sqcup D_3$. Let $k = (x,y)$ with the origin coinciding with the centroid of $\operatorname{proj}A$ triangle with vertices $[0,\frac{\sqrt{6}}{3}]$, $[-\frac{\sqrt{6}}{2},\frac{\sqrt{2}}{2}]$, $[-\frac{\sqrt{6}}{2},-\frac{\sqrt{2}}{2}]$ (see Figure \ref{fig:O_over_wire}). Let us denote $v_j = \vol{A \cap \operatorname{proj}B + k}$ for all $k \in D_j$, then we have for the subdomains:
\begin{itemize}
    \item $D_3$ is a triangle with vertices $[0,0]$, $[\frac{\sqrt{6}}{6},0]$, $[\frac{\sqrt{6}}{6},\frac{\sqrt{2}}{2}]$ in which $v_3 = \sqrt{2}$
    \item $D_2$ is a triangle with vertices $[\frac{\sqrt{6}}{6},0]$, $[\frac{2\sqrt{6}}{3},0]$, $[\frac{\sqrt{6}}{6},\frac{\sqrt{2}}{2}]$ in which $v_2 = \frac{4 \sqrt{2}}{3}-\frac{2 x}{\sqrt{3}}$
\end{itemize}
Note that in general, $\vol{A \cap \operatorname{proj}B + k}$ is \textbf{linear} in $(x,y)$ in the subdomains. By inclusion/exclusion, we can write our integral as
\begin{equation}
\begin{split}
& L_{21r}  = \frac{6}{\volA \volB}\left[\iint_{D_3\cup D_2} v_2\sqrt{h^2+x^2+y^2} \ddd x \dd y + \iint_{D_3} \left(v_3-v_2\right)\sqrt{h^2+x^2+y^2} \ddd x \dd y\right]\\
&\!=\! \frac{6}{\volA \volB}\!\left[\iint_{\! D_3\cup D_2}\!\! \left(\!\frac{4 \sqrt{2}}{3}\!-\!\frac{2 x}{\sqrt{3}}\!\right)\!\!\sqrt{h^2\!+\!x^2\!+\!y^2} \dd x \dd y \!+\!\!\! \iint_{\! D_3} \!\!\left(\!\frac{2 x}{\sqrt{3}}\!-\!\frac{\sqrt{2}}{3}\!\right)\!\!\sqrt{h^2\!+\!x^2\!+\!y^2} \dd x \dd y\right].
\end{split}
\end{equation}
Note that the second integral over domain $3$ is already in the form of an integral over standard fundamental triangle domain since $D_3 = D(\frac{\sqrt{6}}{6},\frac{\pi}{3})$. The first integral over domain $3\cup 2$ can be written in such manner after rotation and reflection. To obtain the correct transformation, we let $\varphi'$ to start (be zero) for the half-line connecting the origin with point $[\frac{1}{\sqrt{6}},\frac{\sqrt{2}}{2}]$, increasing in the clockwise direction. That is $\varphi = \pi/3 - \varphi'$ and thus $x = r \cos(\frac{\pi}{3}-\varphi')$ and $y = r \sin(\frac{\pi}{3}-\varphi')$. Expanding out the trigonometric functions and writing $x' = r \cos \varphi'$ and $y' = r \cos\varphi'$, we get
\begin{equation}
\begin{split}
    x & = r \cos\frac{\pi}{3} \cos\varphi' + r \sin\frac{\pi}{3} \sin\varphi' = x' \cos\frac{\pi}{3} + y' \sin\frac{\pi}{3} = \frac{1}{2}x' + \frac{\sqrt{3}}{2} y',\\
    y & = r \sin\frac{\pi}{3}\cos \varphi' - r \cos\frac{\pi}{3} \sin\varphi' = x' \sin\frac{\pi}{3} - y' \cos\frac{\pi}{3} = \frac{\sqrt{3}}{2}x' - \frac{1}{2} y'
\end{split}
\end{equation}
and so
\begin{equation}
   v_2 = \frac{4 \sqrt{2}}{3}-\frac{2 x}{\sqrt{3}} = \frac{4 \sqrt{2}}{3}-\frac{x'}{\sqrt{3}}-y'.
\end{equation}
Our integration domain $D_3 \sqcup D_2$ in $(x',y')$ is simply $D(\frac{\sqrt{6}}{3},\frac{\pi}{3})$. Note that $x^2 + y^2$ is invariant with respect to this transformation so $x^2+y^2 = x'^2+y'^2$. By scaling with $h$, we can write $L_{21r}$ in terms of the auxiliary integrals as
\begin{equation}
\label{eq:octaL21r}
\begin{split}
L_{21r} = \frac{6 h^3}{\volA \volB}\Bigg{[} &
\frac{4 \sqrt{2}}{3}I^{(1)}_{00}\left(\frac{\sqrt{6}}{3h},\frac{\pi}{3}\right)
-\frac{h}{\sqrt{3}}I^{(1)}_{10}\left(\frac{\sqrt{6}}{3h},\frac{\pi}{3}\right)
-hI^{(1)}_{01}\left(\frac{\sqrt{6}}{3h},\frac{\pi}{3}\right) \\
& + \frac{2h}{\sqrt{3}}I^{(1)}_{10}\left(\frac{\sqrt{6}}{6h},\frac{\pi}{3}\right)
- \frac{\sqrt{2}}{3}I^{(1)}_{00}\left(\frac{\sqrt{6}}{6h},\frac{\pi}{3}\right)\Bigg{]}
\end{split}
\end{equation}
with $\frac{\sqrt{6}}{3h} = \frac{\sqrt{2}}{2}$ and $\frac{\sqrt{6}}{6h} = \frac{\sqrt{2}}{4}$. Via recursions (see Table \ref{tab:auxil} in Appendix), we get
\begin{equation}
    L_{21r} = -\frac{10}{27}+\frac{47}{54 \sqrt{2}}-\frac{16 \pi
   }{81}+\frac{143 \ln 3}{648
   \sqrt{2}}+\frac{32}{81} \arccot\sqrt{2}+\frac{85}{81}
   \sqrt{2} \arccoth\sqrt{2}
\end{equation}

\phantomsection
\subsubsection{\texorpdfstring{L\textsubscript{22r}}{L22r}}
%\addcontentsline{toc}{subsubsection}{L22}
Again, we use the overlap formula for $A$ and $B$ being opposite faces of $K$. By symmetry, we again integrate $\vol{A \cap \operatorname{proj} B + k}$ over one sixth of all positions of vector $k$ (see Figure \ref{fig:O_over_full}, in which white numbers represent the number of sides of a polygon of intersection of $AB$ projections with respect to position of the shift vector $k$ -- black dot).

\begin{figure}[h]
    \centering
     \includegraphics[height=0.24\textwidth]{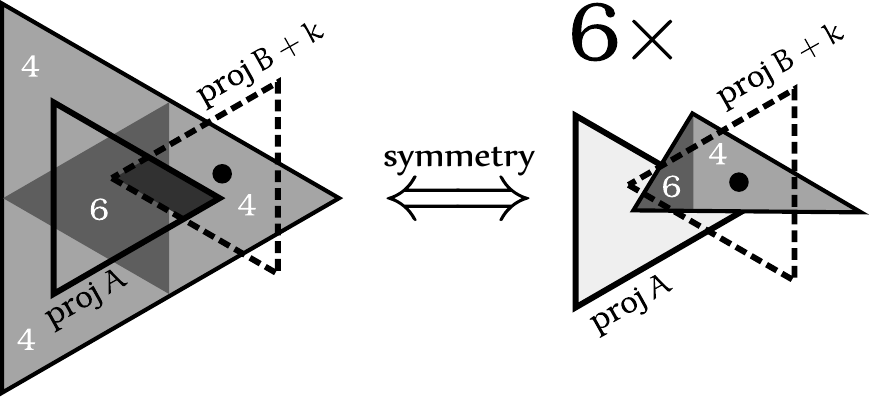}
    \caption{Overlap of the opposite faces of an octahedron}
    \label{fig:O_over_full}
\end{figure}

\noindent
Setting up the integral,
\begin{equation}
L_{22r} = \frac{6}{\volA \volB} \int_D \sqrt{h^2 + k^2}\vol{A \cap \operatorname{proj}B + k} \ddd k,
\end{equation}
where $h = 2/\sqrt{3}, \volA = \volB = \nu = \sqrt{3}/2$ and $D$ is a domain in Figure \ref{fig:O_over_full} on the right consisted of two subdomains labeled $6$ and $4$ according to the number of sides of the intersection (which is a polygon). That is, $D=D_6 \sqcup D_4$. Let $k = (x,y)$ and denote $v_j = \vol{A \cap \operatorname{proj}B + k}$ for those $k$ which lie in $\in D_j$, then the subdomain
\begin{itemize}
    \item $D_6$ is again a triangle with vertices $[0,0]$, $[\frac{\sqrt{6}}{6},0]$, $[\frac{\sqrt{6}}{6},\frac{\sqrt{2}}{2}]$ in which $v_6 = \frac{1}{\sqrt{3}} - \frac{\sqrt{3}}{2}x^2 - \frac{\sqrt{3}}{2}y^2$
    \item $D_4$ is a triangle with vertices $[\frac{\sqrt{6}}{6},0]$, $[\frac{2\sqrt{6}}{3},0]$, $[\frac{\sqrt{6}}{6},\frac{\sqrt{2}}{2}]$ in which $v_4 =\frac{4}{3 \sqrt{3}}-\frac{2 \sqrt{2} x}{3}+\frac{x^2}{2 \sqrt{3}}-\frac{\sqrt{3} y^2}{2}$
\end{itemize}
Domains $6,4$ coincide with $3,2$ in $L_{21r}$ case, that is $D_6 = D_3$ and $D_4=D_2$. Note that in general, $\vol{A \cap \operatorname{proj}B + k}$ is \textbf{quadratic} in $(x,y)$ in the subdomains. By inclusion/exclusion, we can write the integral as
\begin{equation}
\begin{split}
L_{22r} = \frac{6}{\volA \volB}\Big{[} & \iint_{D_6\cup D_4} v_4\sqrt{h^2+x^2+y^2} \ddd x \dd y + \iint_{D_6} \left(v_6-v_4\right)\sqrt{h^2+x^2+y^2} \ddd x \dd y\Bigg{]}\\
= \frac{6}{\volA \volB}\Big{[} & \iint_{D_6\cup D_4} \left(\frac{4}{3 \sqrt{3}}-\frac{2 \sqrt{2} x}{3}+\frac{x^2}{2 \sqrt{3}}-\frac{\sqrt{3} y^2}{2}\right)\sqrt{h^2+x^2+y^2} \ddd x \dd y \\
& + \iint_{D_6} \left(-\frac{1}{3 \sqrt{3}}+\frac{2 \sqrt{2} x}{3}-\frac{2 x^2}{\sqrt{3}}\right)\sqrt{h^2+x^2+y^2} \ddd x \dd y\Big{]}.
\end{split}
\end{equation}
Again, the integral over domain $6$ is in a standard form. The other integral must be first transformed using $x = \frac{1}{2}x' + \frac{\sqrt{3}}{2} y'$, $y = \frac{\sqrt{3}}{2}x' - \frac{1}{2} y'$, which gives
\begin{equation}
v_4 = \frac{4}{3 \sqrt{3}}-\frac{2 \sqrt{2} x}{3}+\frac{x^2}{2 \sqrt{3}}-\frac{\sqrt{3} y^2}{2} = \frac{4}{3 \sqrt{3}}-\frac{\sqrt{2} x'}{3}-\sqrt{\frac{2}{3}} y'-\frac{x^{\prime \, 2}}{\sqrt{3}}+x' y'
\end{equation}
and therefore
\begin{equation}
\label{eq:octaL22r}
\begin{split}
L_{22r} = \frac{6 h^3}{\volA \volB}\Bigg{[} & 
\frac{4}{3 \sqrt{3}}I^{(1)}_{00}\left(\frac{\sqrt{2}}{2},\frac{\pi}{3}\right)
-\frac{\sqrt{2} h}{3}I^{(1)}_{10}\left(\frac{\sqrt{2}}{2},\frac{\pi}{3}\right)
-\sqrt{\frac{2}{3}} h I^{(1)}_{01}\left(\frac{\sqrt{2}}{2},\frac{\pi}{3}\right)\\
&-\frac{h^2}{\sqrt{3}}I^{(1)}_{20}\left(\frac{\sqrt{2}}{2},\frac{\pi}{3}\right)
+h^2I^{(1)}_{11}\left(\frac{\sqrt{2}}{2},\frac{\pi}{3}\right)\\
&-\frac{1}{3 \sqrt{3}}I^{(1)}_{00}\left(\frac{\sqrt{2}}{4},\frac{\pi}{3}\right)
+\frac{2 \sqrt{2} h}{3}I^{(1)}_{10}\left(\frac{\sqrt{2}}{4},\frac{\pi}{3}\right)
-\frac{2 h^2}{\sqrt{3}}I^{(1)}_{20}\left(\frac{\sqrt{2}}{4},\frac{\pi}{3}\right)
\Bigg{]}.
\end{split}
\end{equation}
Going through all recursions, we get, after simplifications
\begin{equation}
L_{22r} = \frac{8}{45}+\frac{\sqrt{2}}{9}-\frac{32 \pi
   }{135}+\frac{293 \ln 3}{270
   \sqrt{2}}-\frac{64}{135} \arccot\sqrt{2}+\frac{124}{135}
   \sqrt{2} \arccoth\sqrt{2}
\end{equation}

\phantomsection
\subsubsection{\texorpdfstring{L\textsubscript{33}}{L33}}
%\addcontentsline{toc}{subsubsection}{L33}
Putting everything together by using \eqref{Eq:Octa}, we finally arrive at
\begin{equation}
    L_{33} = \frac{4}{105}+\frac{13 \sqrt{2}}{105}-\frac{4 \pi }{45}+\frac{109 \ln 3}{630 \sqrt{2}}+\frac{16\arccot\sqrt{2}}{315}+\frac{158\arccoth\sqrt{2}}{315}\sqrt{2}.
\end{equation}
Rescaling, we get our mean distance in a regular octahedron having unit volume
\begin{equation}
    L_{33}\big{|}_{\volK = 1} = \sqrt[3]{\frac{3}{4}}\left(\frac{4}{105}\!+\!\frac{13 \sqrt{2}}{105}\!-\!\frac{4 \pi }{45}\!+\!\frac{109 \ln 3}{630 \sqrt{2}}\!+\!\frac{16\arccot\sqrt{2}}{315}\!+\!\frac{158\arccoth\sqrt{2}}{315}\sqrt{2} \right) \!\approx 0.65853073.
\end{equation}

\subsection{Regular icosahedron}
Regular icosahedron shares many features with regular octahedron. We have already seen that the Crofton Reduction Technique itself is very powerful to reduce the the mean distance until two domains from which we select two points have empty affine hull. As a consequence, the only remaining terms in the icosahedron expansion are the parallel edge-face and parallel face-face configurations. Note that these two parallel configurations have the same overlap diagram as the octahedron has.

Let $\phi = (1+\sqrt{5})/2$ be the Golden ratio. A standard selection of vertices is $[\pm \phi,\pm 1,0]$ and all of their cyclic permutations. That way, our edges have length $l=2$. The volume is equal to $\volK = 10(3+\sqrt{5})/3$ and the face area $\nu = \sqrt{3}$. Again, we put $P = L^p$. For the definition of various mean values $P_{ab} = L_{ab}^{(p)}$, see Figure \ref{fig:LinePickingIcosahedron}.
\begin{figure}[h]
    \centering
     \includegraphics[width=1.0\textwidth]{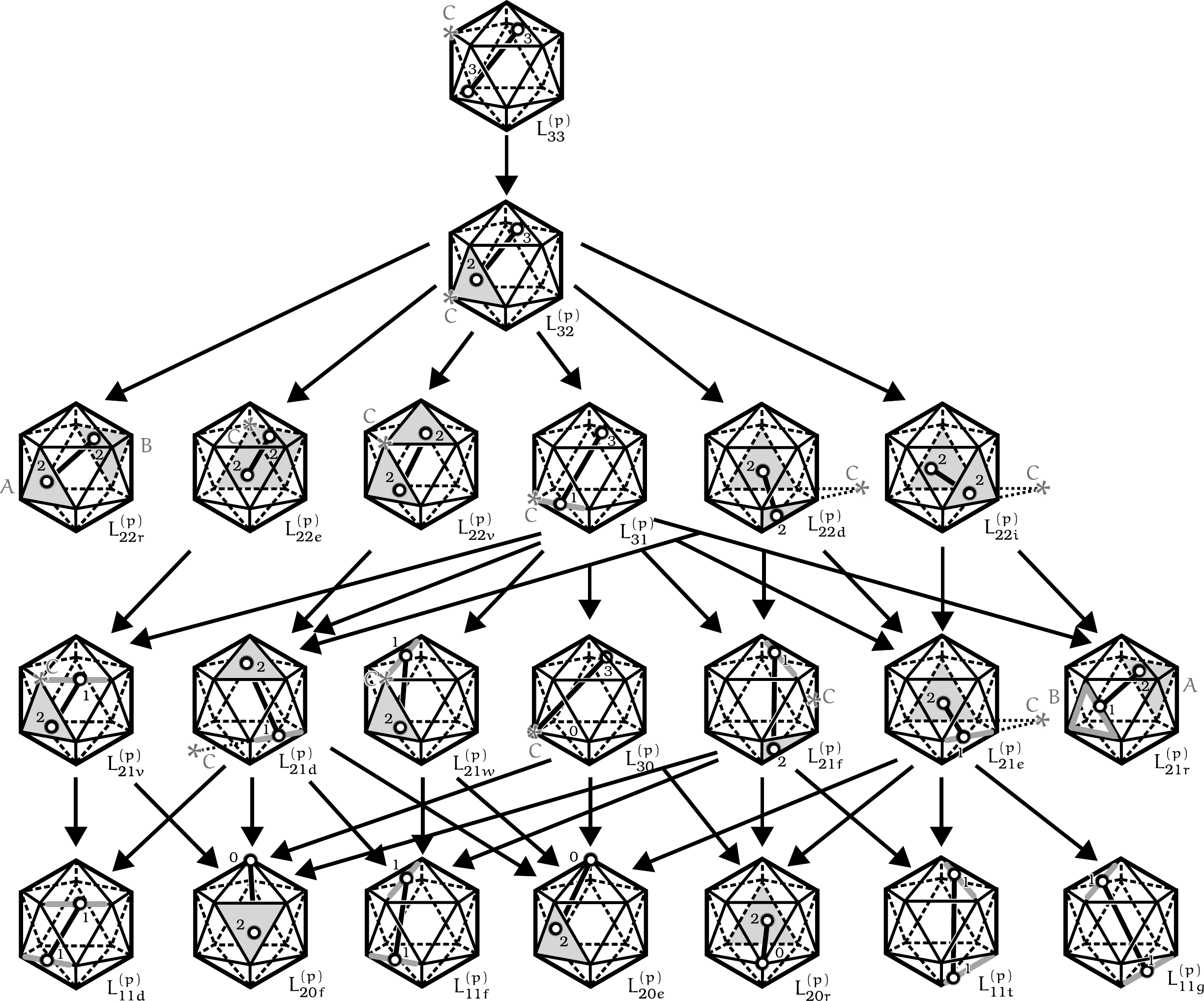}
    \caption{All different $L_{ab}^{(p)}$ configurations encountered for $K$ being a regular icosahedron}
    \label{fig:LinePickingIcosahedron}
\end{figure}

Performing the reduction, we get the set of equations:
\begin{align*}
p P_{33}& =6 (P_{32}-P_{33}),\\
p P_{32} & = 3(P_{22} - P_{32})+2(P_{31}-P_{32}),\\
p P_{31} & = 3 (P_{21}-P_{31}) + 1(P_{30}-P_{31}),\\
p P_{22e} & = 4 (P_{21v}-P_{22e}),\\
p P_{22v} & = 4 (P_{21d}-P_{22v}),\\
p P_{22d} & = 4 (P_{21}'-P_{22d}),\\
p P_{22i} & = 4 (P_{21}''-P_{22i}),\\
p P_{30} & = 3 (P_{20}-P_{30}),\\
p P_{21v} & = 2(P_{11d}-P_{21v}) + 1 (P_{20u}-P_{21v}),\\
p P_{21w} & = 2(P_{11f}-P_{21w}) + 1 (P_{20l}-P_{21w}),\\
p P_{21f} & = 2(P_{11}-P_{21f}) + 1 (P_{20}'-P_{21f}),\\
p P_{21d} & = 2(P_{11}'-P_{21d}) + 1 (P_{20}''-P_{21d}),\\
p P_{21e} & = 2(P_{11}''-P_{21e}) + 1 (P_{20}'''-P_{21e}),\\
\end{align*}
with
\begin{align*}
    P_{22} & = \frac{2
   P_{22d}}{5}+\frac{P_{22r}}{10}+\frac{P_{22v}}{5}+\frac{P_{22e}}{10 \phi ^2}+\frac{\phi ^2 P_{22i}}{10},\\
    P_{21} & = \frac{\phi  P_{21e}}{5}+\frac{P_{21f}}{2 \phi \sqrt{5}}+\frac{P_{21r}}{5}+\frac{P_{21d}}{5}+\frac{P_{21v}}{5 \phi ^2}+\frac{P_{21w}}{10 \phi },\\
    P_{21}' &= \frac{P_{21f}}{2}-\frac{\phi  P_{21d}}{2}+\frac{\phi ^2 P_{21e}}{2},\\
    P_{21}'' & =\phi ^2 P_{21r}-\phi  P_{21e},\\
    P_{20} & = \frac{P_{20f}}{2 \phi ^2}+\frac{P_{20e}}{2 \phi
   }+\frac{P_{20r}}{2},\\
    P_{20}' & =\phi  P_{20r}-\frac{P_{20f}}{\phi },\\
    P_{20}'' & =\phi ^2 P_{20e}-\phi  P_{20f},\\
    P_{20}''' & =\phi ^2 P_{20r}-\phi  P_{20e},\\
    P_{11} & =2 P_{11t}-P_{11f},\\
    P_{11}' & =\phi ^2 P_{11f}-\phi  P_{11d},\\
    P_{11}'' & =\phi ^2 P_{11g}-\phi  P_{11t}.
\end{align*}
Solving the system, we get, after simplifications,
\begin{equation}
\begin{split}
P_{33} & = \frac{18 \left(12 \phi ^2 P_{11d}-12 \phi^4 P_{11f}+4 \phi^8 P_{11g}+12 \phi ^2 P_{11t}-6 \phi ^4 P_{20e}+4 \phi ^8 P_{20r}+6
   P_{20f}\right)}{5 \phi^4 (3+p) (4+p) (5+p)(6+p)}\\
& +\frac{108 \phi^2 P_{21r}}{5 (4+p) (5+p) (6+p)}+\frac{9 P_{22r}}{5 (5+p) (6+p)}.
\end{split}
\end{equation}
When $p=1$, we get for the mean distance
\begin{equation}\label{Eq:Icosa}
\begin{split}
L_{33} & = \frac{3 L_{22r}}{70}-\frac{9 L_{11f}}{175}-\frac{9 L_{20e}}{350}+\frac{9 L_{20f}}{350 \phi ^4}+\frac{9 L_{11d}}{175 \phi^2}+\frac{9 L_{11t}}{175 \phi ^2}+\frac{18\phi^2 L_{21r}}{175}+\frac{3\phi ^4 L_{11g}}{175}+\frac{3\phi ^4 L_{20r}}{175}.
\end{split}
\end{equation}

\phantomsection
\subsubsection{\texorpdfstring{L\textsubscript{11d}}{L11d}}
%\addcontentsline{toc}{subsubsection}{L11d}
Let $A'= \overline{[1,0, \phi][-1,0, \phi]}$ and $B'=\overline{[0,\phi,1][\phi,1,0]}$ be edges of $K$, then $L_{11d} = L_{A'B'}$. By shifting, $L_{A'B'}=L_{OA}$, where $O=[0,0,0]$ is the origin and $A=A'-B'$ is a polygon with vertices $[1,-\phi ,\frac{1}{\phi }], [-\frac{1}{\phi },-1,\phi], [-\phi ^2,-1,\phi], [-1,-\phi ,\frac{1}{\phi }]$ (a parallelogram) having area $\volA = \sqrt{10-2 \sqrt{5}}$. Projecting $O$ onto $\mathcal{A}(A)$, we obtain $\operatorname{proj}_A O =[0,-1-\frac{1}{\sqrt{5}},\frac{2}{\sqrt{5}}]$ and separation $h=\sqrt{2+\frac{2}{\sqrt{5}}}$. Point-Polygon formula yields
\begin{equation}
\begin{split}
L_{11d} = \frac{2 h^3}{\volA} \Bigg{(} & 2I^{(1)}_{00}\left(\frac{1}{2 \phi ^2},\frac{2 \pi }{5}\right)-I^{(1)}_{00}\left(\frac{1}{2},\frac{\pi}{5}\right)+I^{(1)}_{00}\left(\frac{1}{2},\frac{2 \pi }{5}\right)\Bigg{)} \approx 2.0431430525135.
\end{split}
\end{equation}
Explicitly, after series of simplifications on $I^{(1)}_{00}(\cdot,\cdot)$ by recursion formulae, we obtain
\begin{equation}
\begin{split}
    L_{11d} & = \frac{5}{6}+\frac{1}{2 \sqrt{5}}+\frac{1}{15} (2 \pi ) \left(3+\sqrt{5}\right)-\frac{8}{15} \left(3+\sqrt{5}\right) \arccot\phi\\
    &-\frac{8}{15} \left(3+\sqrt{5}\right) \arccot\left(\phi ^2\right)+\frac{1}{60} \left(31-3 \sqrt{5}\right) \ln 3+\frac{13}{120}\left(3+\sqrt{5}\right) \ln 5.
\end{split}
\end{equation}

\phantomsection
\subsubsection{\texorpdfstring{L\textsubscript{11g}}{L11g}}
%\addcontentsline{toc}{subsubsection}{L11g}
Let $A'$ be the same edge as in $L_{11d}$ and $B'=\overline{[1,0,\phi][-1,0,\phi]}$, then $L_{11g} = L_{A'B'}$. By shifting, $L_{A'B'}=L_{OA}$, where $O=[0,0,0]$ and $A=A'-B'$ is a polygon with vertices $[0,0,2 \phi], [1,-\phi ,\phi ^2], [-1,-\phi ,\phi ^2], [-2,0,2 \phi]$ having area $\volA = 2\sqrt{3}$. Projecting $O$ onto $\mathcal{A}(A)$, we obtain $\operatorname{proj}_A O =[0,-\frac{2 \phi }{3},\frac{2 \phi ^3}{3}]$ and separation $h=\sqrt{\frac{14}{3}+2 \sqrt{5}}$. Point-Polygon formula yields
\begin{equation}
\begin{split}
L_{11g} = \frac{2 h^3}{\volA} \Bigg{(} & 2 I^{(1)}_{00}\left(\frac{1}{2 \phi ^2},\frac{\pi }{3}\right)+I^{(1)}_{00}\left(\frac{1}{\phi ^2},\frac{\pi }{3}\right)\Bigg{)} \approx 3.1806727116118.
\end{split}
\end{equation}
Explicitly, after series of simplifications,
\begin{equation}
\begin{split}
    L_{11g} & = \frac{1}{9}+\frac{\sqrt{5}}{9}+\frac{2}{9} \sqrt{2 \left(5+\sqrt{5}\right)}+\frac{4}{45} \left(9+4 \sqrt{5}\right) \pi -\frac{16}{27} \left(9+4
   \sqrt{5}\right) \arccot \phi\\
   &+\left(\frac{43}{27}+\frac{2 \sqrt{5}}{3}\right) \arccoth\phi+\frac{2}{27} \left(23+9 \sqrt{5}\right)
   \arccsch\phi-\left(\frac{43}{108}+\frac{\sqrt{5}}{6}\right) \ln 5.
\end{split}
\end{equation}

\phantomsection
\subsubsection{\texorpdfstring{L\textsubscript{11f}}{L11f}}
%\addcontentsline{toc}{subsubsection}{L11f}
Let $A'$ be the same edge as in $L_{11d}$ and $B'=\overline{[\phi,1,0][\phi,-1,0]}$, then $L_{11f} = L_{A'B'}$. By shifting, $L_{A'B'}=L_{OA}$, where $O=[0,0,0]$ and $A=A'-B'$ is a polygon with vertices $[-\frac{1}{\phi },-1,\phi], [-\frac{1}{\phi },1,\phi], [-\phi ^2,1,\phi], [-\phi ^2,-1,\phi]$ having area $\volA = 4$. Projecting $O$ onto $\mathcal{A}(A)$, we obtain $\operatorname{proj}_A O =[0,0,\phi]$ and separation $h=\phi$. Point-Polygon formula yields
\begin{equation}
\begin{split}
L_{11f} = \frac{2 h^3}{\volA} \Bigg{(} & I^{(1)}_{00}\left(\frac{1}{\phi },\arctan\left(\phi ^2\right)\right)+I^{(1)}_{00}\left(\phi ,\arctan\left(\frac{1}{\phi^2}\right)\right)-I^{(1)}_{00}\left(\frac{1}{\phi ^2},\arctan \phi \right)\\
&-I^{(1)}_{00}\left(\frac{1}{\phi },\arctan\frac{1}{\phi}\right) \Bigg{)} \approx 2.3977565034445.
\end{split}
\end{equation}
Explicitly, after series of simplifications,
\begin{equation}
\begin{split}
    L_{11f} & = \frac{5}{6}+\frac{\sqrt{5}}{6}-\frac{\pi}{96}  \left(1+\sqrt{5}\right)^3+\frac{1}{2} \left(2+\sqrt{5}\right) \arccot \phi -\frac{1}{6}\left(2+\sqrt{5}\right) \arccot\left(\phi ^2\right)\\
    &+\frac{1}{24} \left(39+17 \sqrt{5}\right) \arccoth \phi +\frac{1}{48} \left(1-5\sqrt{5}\right) \ln 3-\frac{1}{96} \left(17+11 \sqrt{5}\right) \ln 5.
\end{split}
\end{equation}

\phantomsection
\subsubsection{\texorpdfstring{L\textsubscript{11t}}{L11t}}
%\addcontentsline{toc}{subsubsection}{L11t}
Again, let $A'$ be the same edge as in $L_{11d}$ and $B'=\overline{[1,0,-\phi][\phi,1,0]}$, then $L_{11t} = L_{A'B'}$. By shifting, $L_{A'B'}=L_{OA}$, where $O=[0,0,0]$ and $A=A'-B'$ is a polygon with vertices $[0,0,2 \phi], [-\frac{1}{\phi },-1,\phi], [-\phi ^2,-1,\phi], [-2,0,2 \phi]$ having area $\volA = \sqrt{2 \left(5+\sqrt{5}\right)}$. Projecting $O$ onto $\mathcal{A}(A)$, we obtain $\operatorname{proj}_A O =[0,-1-\frac{1}{\sqrt{5}},\frac{2}{\sqrt{5}}]$ and separation $h=\sqrt{2+\frac{2}{\sqrt{5}}}$. Point-Polygon formula yields
\begin{equation}
L_{11t} = \frac{2 h^3}{\volA} \Bigg{(} I^{(1)}_{00}\left(\frac{1}{2},\frac{\pi }{5}\right)-I^{(1)}_{00}\left(\frac{1}{2},\frac{2 \pi }{5}\right)+I^{(1)}_{00}\left(\phi ,\frac{\pi }{5}\right) \Bigg{)} \approx 2.8940519649490.
\end{equation}
Explicitly, after series of simplifications,
\begin{equation}
\begin{split}
    L_{11t} & = \frac{4}{3} \sqrt{1+\frac{2}{\sqrt{5}}}-\frac{1}{6}-\frac{\sqrt{5}}{6}-\frac{8\pi}{75} \left(1+\sqrt{5}\right) +\frac{8}{15}\left(1+\sqrt{5}\right) \arccot \phi \\
    & +\frac{4}{15} \left(8+3 \sqrt{5}\right) \arccsch \phi-\frac{13}{120}
   \left(1+\sqrt{5}\right) \ln 5.
\end{split}
\end{equation}

\phantomsection
\subsubsection{\texorpdfstring{L\textsubscript{20e}}{L20e}}
%\addcontentsline{toc}{subsubsection}{L20e}
Let $A$ be the face of $K$ with vertices $[1,0,\phi]$, $[-1,0,\phi ]$, $[0,\phi ,1]$ (an equilateral triangle) and let $B$ be vertex $[\phi ,-1,0]$, then $L_{20e} = L_{AB}$. Projecting $B$ onto $\mathcal{A}(A)$, we obtain $\operatorname{proj}_A B =[\phi ,-\frac{1}{3},\frac{2 \phi ^2}{3}]$ and separation $h= 2\phi /\sqrt{3}$. By Point-Polygon formula,
\begin{equation}
\begin{split}
L_{20e} = \frac{2 h^3}{\nu} \Bigg{(} & I^{(1)}_{00}\left(\frac{1}{2 \phi ^2},\frac{\pi }{3}\right) - I^{(1)}_{00}\left(\frac{1}{2 \phi ^2},\arctan\left(\sqrt{15}+2\sqrt{3}\right)\right)\\
& +I^{(1)}_{00}\left(\frac{\phi ^2}{2},\arctan\left(\sqrt{15}-2\sqrt{3}\right)\right)\Bigg{)} \approx 2.688729552544.
\end{split}
\end{equation}
Explicitly, after series of simplifications,
\begin{equation}
    \begin{split}
    L_{20e} & = \frac{7}{9}+\frac{\sqrt{5}}{9}+\frac{8\pi}{27} \left(2+\sqrt{5}\right) - \frac{16}{9} \left(2+\sqrt{5}\right) \arccot\phi\\
    &+\frac{1}{27}\left(104+47 \sqrt{5}\right) \arccoth\phi-\frac{1}{108} \left(112+61 \sqrt{5}\right) \ln 5.
    \end{split}
\end{equation}

\phantomsection
\subsubsection{\texorpdfstring{L\textsubscript{20r}}{L20r}}
%\addcontentsline{toc}{subsubsection}{L20r}
Let $A$ be the same face of $K$ as in the section on $L_{20e}$ and let $B$ be vertex $[1,0,-\phi]$, then $L_{20r} = L_{AB}$. Projecting $B$ onto $\mathcal{A}(A)$, we obtain $\operatorname{proj}_A B =[1,\frac{2 \phi }{3},\frac{\sqrt{5} \phi }{3}]$ and separation $h=2\phi ^2/\sqrt{3}$. By Point-Polygon formula,
\begin{equation}
L_{20r} = \frac{2 h^3}{\nu} \Bigg{(} I^{(1)}_{00}\left(\frac{1}{\phi^2},\frac{\pi}{3}\right)-I^{(1)}_{00}\left(\frac{1}{2 \phi ^2},\frac{\pi }{3}\right)\Bigg{)} \approx 3.28394367574.
\end{equation}
Explicitly, after series of simplifications,
\begin{equation}
    \begin{split}
   L_{20r} & = \frac{4}{9} \sqrt{2 \left(5+\sqrt{5}\right)}-\frac{1}{9}-\frac{\sqrt{5}}{9}-\frac{16 \pi}{135} \left(9+4 \sqrt{5}\right)+\frac{16}{27}
   \left(9+4 \sqrt{5}\right) \arccot\phi\\
   &-\left(\frac{43}{27}+\frac{2 \sqrt{5}}{3}\right) \arccoth\phi+\frac{4}{27} \left(23+9\sqrt{5}\right) \arccsch\phi+\left(\frac{43}{108}+\frac{\sqrt{5}}{6}\right) \ln 5.
    \end{split}
\end{equation}

\phantomsection
\subsubsection{\texorpdfstring{L\textsubscript{20f}}{L20f}}
%\addcontentsline{toc}{subsubsection}{L20f}
Let $A$ be the same face of $K$ as in the section on $L_{20e}$ and let $B$ be vertex $[\phi,1,0]$, then $L_{20f} = L_{AB}$. Projecting $B$ onto $\mathcal{A}(A)$, we obtain $\operatorname{proj}_A B =[\phi ,\frac{\phi ^3}{3},\frac{2 \phi }{3}]$ and separation $h=2/\sqrt{3} $. By Point-Polygon formula,
\begin{equation}
\begin{split}
L_{20f} = \frac{2 h^3}{\nu} \Bigg{(} & I^{(1)}_{00}\left(\frac{\phi ^2}{2},\frac{\pi }{3}\right)-I^{(1)}_{00}\left(\frac{\sqrt{5}}{2},\arctan\sqrt{\frac{3}{5}}\right)\\
   &-I^{(1)}_{00}\left(\frac{\phi ^2}{2},\arctan\left(\sqrt{15}-2 \sqrt{3}\right)\right)\Bigg{)} \approx 2.2472771159735.
\end{split}
\end{equation}
Explicitly, after series of simplifications,
\begin{equation}
   L_{20f} = \frac{10}{9}+\frac{2 \sqrt{5}}{9}-\frac{8 \pi }{27}+\frac{32}{27} \arccot\phi+\frac{16}{27} \arccot(\phi^2)-\frac{17}{54}\sqrt{5} \ln 3+\left(\frac{1}{2}+\frac{5 \sqrt{5}}{27}\right) \ln 5.
\end{equation}

\phantomsection
\subsubsection{\texorpdfstring{L\textsubscript{21r}}{L21r}}
%\addcontentsline{toc}{subsubsection}{L21r}
%%
Let $A$ be a face of $K$ and $B$ be a boundary of the opposite face. In icosahedron $K$, two faces are separated by the distance $2\phi ^2/\sqrt{3}$. Since the overlap diagram of these faces is the same as the one associated to two opposite faces of an octahedron (see Figure \ref{fig:O_over_wire}), the coefficients of the expansion of irreducible $L_{21r}$ term into auxiliary integrals $I^{(1)}_{ij}$ \emph{match}. However, this is only valid provided the edge length is $\sqrt{2}$. Since our icosahedron $K$ has $l=2$, we first rescale our icosahedron by $1/\sqrt{2}$. In the final step, since the mean distance scales linearly, we have just rescale $L_{ab}$ back by multiplying it by $\sqrt{2}$. Hence, by using Equation \eqref{eq:octaL21r},
\begin{equation}
\begin{split}
L_{21r} & = L_{21r}\big{|}_{l=2} = \sqrt{2}L_{21r}\big{|}_{l=\sqrt{2}} =\sqrt{2}\frac{6 h^3}{\volA \volB}\Bigg{[}
\frac{4 \sqrt{2}}{3}I^{(1)}_{00}\left(\frac{\sqrt{6}}{3h},\frac{\pi}{3}\right)
-\frac{h}{\sqrt{3}}I^{(1)}_{10}\left(\frac{\sqrt{6}}{3h},\frac{\pi}{3}\right)
 \\
&-hI^{(1)}_{01}\left(\frac{\sqrt{6}}{3h},\frac{\pi}{3}\right) + \frac{2h}{\sqrt{3}}I^{(1)}_{10}\left(\frac{\sqrt{6}}{6h},\frac{\pi}{3}\right)
- \frac{\sqrt{2}}{3}I^{(1)}_{00}\left(\frac{\sqrt{6}}{6h},\frac{\pi}{3}\right)\Bigg{]} \approx 3.1819213671057,
\end{split}
\end{equation}
where $h = \sqrt{2}\phi^2/\sqrt{3}$, $ \volA = \sqrt{3}/2$ and $\volB = 3\sqrt{2}$ are the rescaled icosahedron opposite faces separation, rescaled face area and face perimeter, respectively. Contrary to the octahedron case, we now have $\frac{\sqrt{6}}{3h} = 1/\phi^2$ and $\frac{\sqrt{6}}{6h} = 1/(2\phi^2)$. Via recursions, we get after some simplifications,
\begin{equation}
\begin{split}
    L_{21r} & = \frac{227}{108}+\frac{107 \sqrt{5}}{108}-\frac{25}{27} \sqrt{10+\frac{22}{\sqrt{5}}}-\frac{8 \pi}{135} \left(9+4\sqrt{5}\right)+\frac{16}{81} \left(9+4 \sqrt{5}\right) \arccot \phi\\
   &+\left(\frac{1043}{324}+\frac{13 \sqrt{5}}{9}\right) \arccoth \phi +\left(\frac{179}{81}+\frac{7 \sqrt{5}}{9}\right) \arccsch\phi -\frac{1043+468 \sqrt{5}}{1296} \ln 5.
\end{split}
\end{equation}

\phantomsection
\subsubsection{\texorpdfstring{L\textsubscript{22r}}{L22r}}
%\addcontentsline{toc}{subsubsection}{L22}
Again, Overlap diagram of $L_{21r}$ configuration matches that of an octahedron. Immediately from Equation \eqref{eq:octaL22r}, by rescaling and replacing $\sqrt{2}/2$ by $1/\phi^2$ and $ \sqrt{2}/4$ by $1/(2\phi^2)$ in the first argument of $I^{(p)}_{ij}$ integrals, we get
\begin{equation}
\begin{split}
L_{22r} & = \sqrt{2}\frac{6 h^3}{\volA \volB}\Bigg{[} \frac{4}{3 \sqrt{3}}I^{(1)}_{00}\left(\frac{1}{\phi^2},\frac{\pi}{3}\right) -\frac{\sqrt{2} h}{3}I^{(1)}_{10}\left(\frac{1}{\phi^2},\frac{\pi}{3}\right) -\sqrt{\frac{2}{3}} h I^{(1)}_{01}\left(\frac{1}{\phi^2},\frac{\pi}{3}\right)\\
&-\frac{h^2}{\sqrt{3}}I^{(1)}_{20}\left(\frac{1}{\phi^2},\frac{\pi}{3}\right) +h^2I^{(1)}_{11}\left(\frac{1}{\phi^2},\frac{\pi}{3}\right)-\frac{1}{3 \sqrt{3}}I^{(1)}_{00}\left(\frac{1}{2\phi^2},\frac{\pi}{3}\right) +\frac{2 \sqrt{2} h}{3}I^{(1)}_{10}\left(\frac{1}{2\phi^2},\frac{\pi}{3}\right) \\
&-\frac{2 h^2}{\sqrt{3}}I^{(1)}_{20}\left(\frac{1}{2\phi^2},\frac{\pi}{3}\right) \Bigg{]} \approx 3.12998447304770,
\end{split}
\end{equation}
where $h=\sqrt{2}\phi^2/\sqrt{3}$ and $\volA = \volB = \sqrt{3}/2$. Explicitly, after some simplifications,
\begin{equation}
\begin{split}
L_{22r} & = \frac{4}{9} \left(3+2 \sqrt{5}\right) \sqrt{10+\frac{22}{\sqrt{5}}}-\frac{271}{45}-\frac{119}{9 \sqrt{5}}+\frac{16 \pi}{675}\left(78+35 \sqrt{5}\right)-\frac{32}{135} \left(67+30 \sqrt{5}\right) \arccot\phi\\
   &+\left(\frac{611}{45}+\frac{164 \sqrt{5}}{27}\right) \arccoth \phi-\frac{28}{135} \left(9+5 \sqrt{5}\right)
   \arccsch \phi-\left(\frac{611}{180}+\frac{41 \sqrt{5}}{27}\right) \ln 5.
\end{split}
\end{equation}

\phantomsection
\subsubsection{\texorpdfstring{L\textsubscript{33}}{L33}}
%\addcontentsline{toc}{subsubsection}{L33}
Putting everything together by using \eqref{Eq:Icosa}, we finally arrive at
\begin{equation}
\begin{split}
    L_{33} & = \frac{197}{525}+\frac{239}{525 \sqrt{5}}-\frac{44}{525} \sqrt{2+\frac{2}{\sqrt{5}}}-\frac{\left(17226+6269\sqrt{5}\right) \pi }{157500}-\frac{\left(2186+1413 \sqrt{5}\right) \arccot \phi}{15750}\\
   &+\frac{\left(82-75\sqrt{5}\right) \arccot\left(\phi ^2\right)}{5250}+\frac{\left(15969+7151 \sqrt{5}\right) \arccoth \phi}{12600}+\frac{4 \left(2139+881 \sqrt{5}\right) \arccsch \phi}{7875}\\
   &+\frac{\left(4449-1685\sqrt{5}\right) \ln 3}{42000}-\frac{\left(75783+37789 \sqrt{5}\right) \ln 5}{252000} \approx 1.66353152568500.
\end{split}
\end{equation}
Rescaling, we get our mean distance in a regular icosahedron having unit volume
\begin{equation}
    L_{33}\big{|}_{\volK = 1} = \frac{L_{33}}{\sqrt[3]{\frac{10}{3} \left(3+\sqrt{5}\right)}} \approx 0.64131248551.
\end{equation}

\subsection{Regular dodecahedron}
Finaly, we will calculte the mean distance in the regular dodecahedron. Let us choose the vertices as $[\pm\phi,\pm\phi,\pm\phi], [0,\pm 1, \pm \phi^2]$ and all their cyclic permutations ($\phi = (1+\sqrt{5})/2$ as usual). Under this choice, each edge has length $l=2$ and each face has area $\nu = \sqrt{25+10 \sqrt{5}}$.

\begin{figure}[h]
    \centering
     \includegraphics[width=1.0\textwidth]{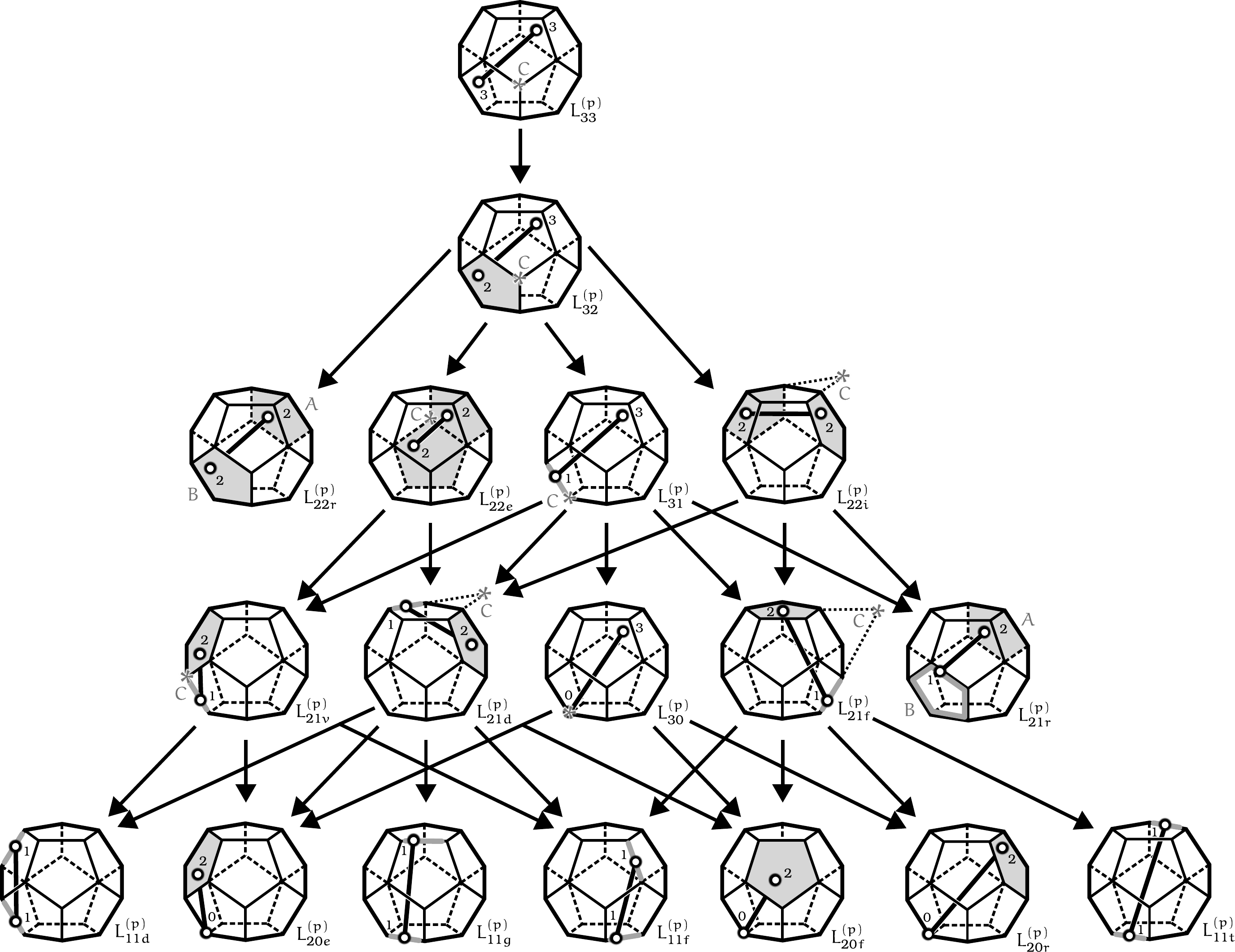}
    \caption{All different $L_{ab}^{(p)}$ configurations encountered for $K$ being a regular icosahedron}
    \label{fig:LinePickingDodecahedron}
\end{figure}

Performing CRT, we get the configurations shown in Figure \ref{fig:LinePickingDodecahedron}. Even though there are less configurations than for the icosahedron, the dodecahedron has more complicated overlap diagram (see Figure \ref{fig:OverlapDode}, there is ten-fold symmetry with respect to rotation and reflection). Distance moments are again connected through CRT via the following set of reduction equations
\begin{align*}
p P_{33}& =6 (P_{32}-P_{33}),\\
p P_{32} & = 3(P_{22} - P_{32})+2(P_{31}-P_{32}),\\
p P_{31} & = 3 (P_{21}-P_{31}) + 1(P_{30}-P_{31}),\\
p P_{22e} & = 4 (P_{21}'-P_{22e}),\\
p P_{22i} & = 4 (P_{21}''-P_{22i}),\\
p P_{30} & = 3 (P_{20}-P_{30}),\\
p P_{21v} & = 1 (P_{20e}-P_{21v})+2(P_{11}-P_{21v}),\\
p P_{21f} & =2 (P_{11}'-P_{21f})+1 (P_{20}'-P_{21f}),\\
p P_{21d}& =2(P_{11}''-P_{21d})+1 (P_{20}''-P_{21d})\\
\end{align*}
with
\begin{align*}
    P_{22} & = \frac{1}{6\phi}\left(\sqrt{5}P_{22e}+\phi P_{22r}+\phi^2\sqrt{5}P_{22i}\right),\\
    P_{21} & = \frac{P_{21r}}{3}+\frac{P_{21d}}{3}+\frac{\phi P_{21f}}{6}+\frac{P_{21v}}{6 \phi ^2},\\
    P_{21}' &= \frac{1}{\phi  \sqrt{5}} \left(P_{21v} + \phi^2 P_{21d} \right),\\
    P_{21}'' & =\frac{1}{\sqrt{5}}\left(\phi P_{21f} + \phi P_{21r} - P_{21d}\right),\\
    P_{20} & = \frac{1}{2 \phi^2}\left(P_{20e} + \phi P_{20f} + \phi^2 P_{20r}\right),\\
    P_{20}' & =\phi^2 P_{20r}-\phi P_{20f},\\
    P_{20}'' & =\phi^2 P_{20f}-\phi P_{20e},\\
    P_{11} & =\frac{1}{\phi  \sqrt{5}} \left(2 P_{11d}+ \phi P_{11f} \right),\\
    P_{11}' & =\frac{1}{\sqrt{5}}\left(2\phi P_{11t} - P_{11f}\right),\\
    P_{11}'' & =\frac{1}{\sqrt{5}} \left(\phi P_{11g} + \phi P_{11f} - P_{11d}\right).
\end{align*}
Solving the system, we get, after simplifications,
\begin{equation}
\begin{split}
P_{33} = \, & \frac{12\left(2 \sqrt{5} P_{11d}+5\phi P_{20e}+2 \phi^3 P_{11g}-2\phi^4\sqrt{5} P_{11f}-5 \phi^5 P_{20f} +4\sqrt{5}\phi^6 P_{20r} +2\phi^9 P_{11t}\right)}{\phi^4 \sqrt{5}\, (3+p) (4+p) (5+p) (6+p)}\\
   &+\frac{60 \phi P_{21r} }{\sqrt{5}(4+p) (5+p) (6+p)} + \frac{3 P_{22r}}{(5+p) (6+p)}.
\end{split}
\end{equation}
When $p=1$, we get for the mean distance
\begin{equation}\label{Eq:Dodeca}
L_{33} = \frac{L_{11d}}{35 \phi ^4}+\frac{L_{20e}}{14 \sqrt{5} \phi ^3}+\frac{L_{11g}}{35 \sqrt{5} \phi
   }-\frac{L_{11f}}{35}+\frac{L_{22r}}{14}-\frac{\phi  L_{20f}}{14 \sqrt{5}}+\frac{2 \phi  L_{21r}}{7\sqrt{5}}+\frac{2 \phi^2
   L_{20r}}{35}+\frac{\phi ^5 L_{11t}}{35 \sqrt{5}}.
\end{equation}

\phantomsection
\subsubsection{\texorpdfstring{L\textsubscript{11d}}{L11d}}
%\addcontentsline{toc}{subsubsection}{L11d}
Let $A'= \overline{[0, \phi^2, 1][0, \phi^2, -1]}$ and $B'=\overline{[\phi,\phi,-\phi][1,0,-\phi^2]}$ be edges of $K$, then $L_{11d} = L_{A'B'}$. By shifting, $L_{A'B'}=L_{OA}$, where $O=[0,0,0]$ is the origin and $A=A'-B'$ is a polygon with vertices $[-\phi ,1,\phi ^2], [-1,\phi ^2,\sqrt{5} \phi], [-1,\phi ^2,\phi], [-\phi ,1,1/\phi]$ having area $\volA = 2\sqrt{3}$. Projecting $O$ onto $\mathcal{A}(A)$, we obtain $\operatorname{proj}_A O =[-1-\frac{\sqrt{5}}{3},\frac{2}{3},0]$ and separation $h=(1+\sqrt{5})/\sqrt{3}$. Point-Polygon formula yields
\begin{equation}
\begin{split}
L_{11d} = \frac{2 h^3}{\volA} \Bigg{(} & I^{(1)}_{00}\left(\frac{1}{2 \phi ^2},\frac{\pi }{3}\right)+I^{(1)}_{00}\left(\frac{\sqrt{5}}{2},\frac{\pi }{3}\right)-I^{(1)}_{00}\left(\frac{1}{2 \phi ^2},\arctan\left(\sqrt{3}\left(2+\sqrt{5}\right)\right)\right)\\
   &-I^{(1)}_{00}\left(\frac{\sqrt{5}}{2},\arctan\sqrt{\frac{3}{5}}\right)\Bigg{)} \approx 3.1367199950978.
\end{split}
\end{equation}
Explicitly, after series of simplifications,
\begin{equation}
\begin{split}
    L_{11d} & = \frac{10 \sqrt{2}}{9}-\frac{\sqrt{5}}{3}+\frac{5 \sqrt{10}}{9}-\frac{5}{9}-\frac{2 \pi}{27} \left(2+\sqrt{5}\right)-\frac{17}{108}\left(5+2 \sqrt{5}\right) \ln 3-\frac{1}{108} \left(4+7 \sqrt{5}\right) \ln 5\\
   &+\frac{4}{27} \left(2+\sqrt{5}\right) \left(2 \arccot 2+2 \arccot\sqrt{2}-\arccos\frac{2}{3}\right)+\frac{17}{108} \left(5+2 \sqrt{5}\right) \arccosh\frac{13}{3}.
\end{split}
\end{equation}

\phantomsection
\subsubsection{\texorpdfstring{L\textsubscript{11g}}{L11g}}
%\addcontentsline{toc}{subsubsection}{L11g}
Let $A'$ be the same edge as in $L_{11d}$ and $B'=\overline{[\phi,-\phi,\phi][\phi^2,-1,0]}$, then $L_{11g} = L_{A'B'}$. By shifting, $L_{A'B'}=L_{OA}$, where $O=[0,0,0]$ and $A=A'-B'$ is a polygon with vertices $[-\phi ,\phi ^3,-1/\phi], [-\phi ^2,\sqrt{5} \phi ,1], [-\phi ^2,\sqrt{5} \phi ,-1], [-\phi ,\phi ^3,-\phi ^2]$ having area $\volA = \sqrt{10-2 \sqrt{5}}$. Projecting $O$ onto $\mathcal{A}(A)$, we obtain $\operatorname{proj}_A O =[-1-\frac{3}{\sqrt{5}},2+\frac{4}{\sqrt{5}},0]$ and separation $h=\sqrt{10+\frac{22}{\sqrt{5}}}$. Point-Polygon formula yields
\begin{equation}
\begin{split}
L_{11g} = \frac{2 h^3}{\volA} \Bigg{(} & 2 I^{(1)}_{00}\left(\frac{1}{2 \phi ^4},\frac{2 \pi }{5}\right)-I^{(1)}_{00}\left(\frac{1}{2 \phi ^2},\frac{\pi}{5}\right)+I^{(1)}_{00}\left(\frac{1}{2 \phi ^2},\frac{2 \pi }{5}\right) \Bigg{)} \approx 4.60478605392525.
\end{split}
\end{equation}
Explicitly, after series of simplifications,
\begin{equation}
\begin{split}
    L_{11g} & = \frac{5}{6}-\frac{1}{3 \sqrt{2}}+\frac{11}{6 \sqrt{5}}+\frac{1}{\sqrt{10}}-\left(\frac{47}{5}+\frac{21}{\sqrt{5}}\right) \pi +\frac{1}{120}\left(219+97 \sqrt{5}\right) \ln 3\\
    &+\frac{2}{15} \left(47+21 \sqrt{5}\right) \left(\arccos\frac{2}{3}+2 \arccos\frac{1}{\sqrt{41}}+2 \arccos\frac{3}{\sqrt{41}}-2 \arccot\sqrt{2}\right)\\
    &+\frac{1}{120} \left(219+97 \sqrt{5}\right) \left(\arccosh\frac{7}{3}-\arccosh 3 \right)\\
    &+\frac{1}{60} \left(91+33 \sqrt{5}\right) \left(\arccosh\frac{9}{\sqrt{41}}-\arccosh\frac{7}{\sqrt{41}}\right).
\end{split}
\end{equation}

\phantomsection
\subsubsection{\texorpdfstring{L\textsubscript{11f}}{L11f}}
%\addcontentsline{toc}{subsubsection}{L11f}
Let $A'$ be the same edge as in $L_{11d}$ and $B'=\overline{[1,0,-\phi^2][-1,0,-\phi^2]}$, then $L_{11f} = L_{A'B'}$. By shifting, $L_{A'B'}=L_{OA}$, where $O=[0,0,0]$ and $A=A'-B'$ is a polygon with vertices $[-1,\phi ^2,\sqrt{5} \phi], [1,\phi ^2,\sqrt{5} \phi], [1,\phi ^2,\phi], [-1,\phi ^2,\phi]$ having area $\volA = 4$. Projecting $O$ onto $\mathcal{A}(A)$, we obtain $\operatorname{proj}_A O =[0,\phi^2,0]$ and separation $h=\phi^2$. Point-Polygon formula yields
\begin{equation}
\begin{split}
L_{11f} = \frac{2 h^3}{\volA} \Bigg{(} & I^{(1)}_{00}\left(\frac{1}{\phi ^2},\arctan\sqrt{5} \phi\right)+I^{(1)}_{00}\left(\frac{\sqrt{5}}{\phi },\arctan\frac{1}{\sqrt{5} \phi }\right)-I^{(1)}_{00}\left(\frac{1}{\phi ^2},\arctan \phi\right)
\\& -I^{(1)}_{00}\left(\frac{1}{\phi },\arctan\frac{1}{\phi}\right) \Bigg{)} \approx 3.770095521642.
\end{split}
\end{equation}
Explicitly, after series of simplifications,
\begin{equation}
\begin{split}
    L_{11f} & = \frac{5}{3 \sqrt{2}}-\frac{1}{2}+\sqrt{\frac{5}{2}}-\frac{\sqrt{5}}{6}+\left(\frac{3}{8}+\frac{\sqrt{5}}{6}\right) \pi +\frac{1}{48} \left(125+53 \sqrt{5}\right) \arccos\frac{9}{\sqrt{41}}\\
    &+\left(\frac{3}{8}+\frac{\sqrt{5}}{6}\right) \left(2 \arccos\frac{13}{3 \sqrt{41}}+2 \arccot 2-\arccos\frac{1}{9}-2 \arccos\frac{3}{5}-2 \arccos\frac{1}{\sqrt{41}}\right)\\
    &+\frac{1}{48} \left(23+9 \sqrt{5}\right) \arccosh\frac{13}{3}-\frac{1}{48} \left(125+53 \sqrt{5}\right) \arccosh\frac{7}{\sqrt{41}}\\
    &-\frac{1}{48} \left(37+17 \sqrt{5}\right) \arcsinh 2-\frac{1}{48} \left(23+9 \sqrt{5}\right) \log (3)+\frac{1}{96} \left(37+17 \sqrt{5}\right) \ln 5.
\end{split}
\end{equation}

\phantomsection
\subsubsection{\texorpdfstring{L\textsubscript{11t}}{L11t}}
%\addcontentsline{toc}{subsubsection}{L11t}
Again, let $A'$ be the same edge as in $L_{11d}$ and $B'=\overline{[\phi,-\phi,-\phi][0,-\phi^2,-1]}$, then $L_{11t} = L_{A'B'}$. By shifting, $L_{A'B'}=L_{OA}$, where $O=[0,0,0]$ and $A=A'-B'$ is a polygon with vertices $[-\phi ,\phi ^3,\phi ^2], [0,2 \phi ^2,2], [0,2 \phi ^2,0], [-\phi ,\phi ^3,1/\phi]$ having area $\volA = \sqrt{2 \left(5+\sqrt{5}\right)}$. Projecting $O$ onto $\mathcal{A}(A)$, we obtain $\operatorname{proj}_A O =[-1-\frac{3}{\sqrt{5}},2+\frac{4}{\sqrt{5}},0]$ and separation $h=\sqrt{10+\frac{22}{\sqrt{5}}}$. Point-Polygon formula yields
\begin{equation}
L_{11t} = \frac{2 h^3}{\volA} \Bigg{(} I^{(1)}_{00}\left(\frac{1}{\phi },\frac{\pi }{5}\right)+I^{(1)}_{00}\left(\frac{1}{2\phi ^2},\frac{\pi }{5}\right)-I^{(1)}_{00}\left(\frac{1}{2 \phi^2},\frac{2 \pi}{5}\right) \Bigg{)} \approx 5.04162416571318.
\end{equation}
Explicitly, after series of simplifications,
\begin{equation}
\begin{split}
    L_{11t} & = \frac{\sqrt{\frac{2}{5}}}{3}-\frac{1}{2}+2 \sqrt{\frac{3}{5}}+\frac{2}{\sqrt{3}}-\frac{7}{6 \sqrt{5}}-\frac{\pi}{45}  \left(29+13\sqrt{5}\right)\\
    & +\frac{2}{15} \left(29+13 \sqrt{5}\right) \left(2 \arccot\sqrt{2}-\arccos\frac{2}{3}\right)+\left(\frac{61}{15}+\frac{9}{\sqrt{5}}\right) \left(\arccosh 4-\arccosh 2\right)\\
    &+\frac{1}{120} \left(133+61 \sqrt{5}\right) \left(\arccosh 3-\arccosh\frac{7}{3}-\ln 3 \right).
\end{split}
\end{equation}

\phantomsection
\subsubsection{\texorpdfstring{L\textsubscript{20e}}{L20e}}
%\addcontentsline{toc}{subsubsection}{L20e}
Let $A$ be the face of $K$ with vertices $[1,0,-\phi ^2]$, $[\phi ,\phi ,-\phi]$, $[0,\phi ^2,-1]$, $[-\phi ,\phi,-\phi]$, $[-1,0,-\phi ^2]$ (a regular pentagon) and let $B$ be vertex $[0,\phi^2,1]$, then $L_{20e} = L_{AB}$. Projecting $B$ onto $\mathcal{A}(A)$, we obtain $\operatorname{proj}_A B =[0,\frac{3}{\phi^2\sqrt{5}},-\frac{1}{\sqrt{5}}]$ and separation $h=\sqrt{2+\tfrac{2}{\sqrt{5}}}$. By Point-Polygon formula,
\begin{equation}
\begin{split}
L_{20e} = \frac{2 h^3}{\nu} \Bigg{(} & I^{(1)}_{00}\left(\frac{1}{2},\frac{\pi }{5}\right)-I^{(1)}_{00}\left(\frac{1}{2},\frac{2 \pi}{5}\right)-I^{(1)}_{00}\left(\frac{\phi^2}{2},\frac{\pi }{5}\right) +I^{(1)}_{00}\left(\frac{3\phi}{4},\arctan\frac{\sqrt{5-2 \sqrt{5}}}{3}\right) \\
& + I^{(1)}_{00}\left(\frac{\phi^2}{2},\arctan\sqrt{5 \left(5-2 \sqrt{5}\right)}\right)\Bigg{)} \approx 3.346942678627.
\end{split}
\end{equation}
Explicitly, after series of simplifications,
\begin{equation}
    \begin{split}
    & L_{20e} = \frac{7 \sqrt{\frac{2}{5}}}{3}+\frac{13 \sqrt{2}}{15}-\frac{4}{15}-\frac{4}{3 \sqrt{5}}+\frac{4 \pi }{15 \sqrt{5}}-\left(\frac{1}{3}+\frac{9}{10
   \sqrt{5}}\right) \ln 3-\frac{13 \ln 5}{30 \sqrt{5}}+\frac{8 \arccos\frac{2}{3}}{15\sqrt{5}}\\
   & +\left(\frac{1}{3}+\frac{9}{10 \sqrt{5}}\right) \arccosh\frac{13}{3}-\frac{3}{50} \left(25+8 \sqrt{5}\right) \arccosh\frac{7}{\sqrt{41}}+\frac{3}{50} \left(25+8 \sqrt{5}\right) \arccosh\frac{9}{\sqrt{41}}\\
   & -\frac{16 \arccot 2}{15 \sqrt{5}}+\frac{8 \arctan\frac{9 \sqrt{2}}{17}}{15 \sqrt{5}}-\frac{8 \arctan\frac{5
   \sqrt{2}}{7}}{15 \sqrt{5}}-\frac{8 \arctan\frac{3 \sqrt{10}}{19}}{15 \sqrt{5}}-\frac{8 \arctan\sqrt{10}}{15 \sqrt{5}}.
    \end{split}
\end{equation}

\phantomsection
\subsubsection{\texorpdfstring{L\textsubscript{20r}}{L20r}}
%\addcontentsline{toc}{subsubsection}{L20r}
Let $A$ be the same face of $K$ as in the section on $L_{20e}$ and let $B$ be vertex $[0,-\phi^2,1]$, then $L_{20r} = L_{AB}$. Projecting $B$ onto $\mathcal{A}(A)$, we obtain $\operatorname{proj}_A B =[0,\frac{1}{10} \left(\sqrt{5}-5\right),-1-\frac{4}{\sqrt{5}}]$ and separation $h=\sqrt{10+\frac{22}{\sqrt{5}}}$. By Point-Polygon formula,
\begin{equation}
\begin{split}
L_{20r} & = \frac{2 h^3}{\nu} \Bigg{(} I^{(1)}_{00}\left(\frac{1}{\phi },\frac{\pi }{5}\right)+I^{(1)}_{00}\left(\frac{1}{2 \phi ^2},\frac{2 \pi }{5}\right)-I^{(1)}_{00}\left(\frac{1}{2 \phi^2},\frac{\pi }{5}\right)-I^{(1)}_{00}\left(\frac{1}{2 \phi ^4},\frac{2 \pi }{5}\right)\Bigg{)} \\
& \approx 4.87605984948.
\end{split}
\end{equation}
Explicitly, after series of simplifications,
\begin{equation}
    \begin{split}
   L_{20r} & = \frac{4}{15}-\frac{2 \sqrt{2}}{15}+\frac{4}{5 \sqrt{3}}+\frac{2}{3 \sqrt{5}}+\frac{4}{\sqrt{15}}+\frac{16 \pi }{9}+\frac{4 \pi
   }{\sqrt{5}}-\frac{16}{75} \left(20+9 \sqrt{5}\right) \arccot\sqrt{2}
   \\
   &+\left(\frac{3}{5}+\frac{43}{30 \sqrt{5}}\right) \ln 3 +\frac{8}{75} \left(20+9 \sqrt{5}\right) \left(\arccos\frac{2}{3}-\arccos\frac{1}{\sqrt{41}}-\arccos\frac{3}{\sqrt{41}}\right)\\
   & -\frac{2}{75} \left(85+37 \sqrt{5}\right) \arccosh 2+\left(\frac{3}{5}+\frac{43}{30
   \sqrt{5}}\right) \left(\arccosh\frac{7}{3}-\arccosh 3\right)\\
   & +\frac{2}{75} \left(85+37 \sqrt{5}\right) \arccosh 4+\left(\frac{1}{15}+\frac{9}{10 \sqrt{5}}\right) \left(\arccosh \frac{7}{\sqrt{41}} -\arccosh\frac{9}{\sqrt{41}}\right)
    \end{split}
\end{equation}

\phantomsection
\subsubsection{\texorpdfstring{L\textsubscript{20f}}{L20f}}
%\addcontentsline{toc}{subsubsection}{L20f}
Let $A$ be the same face of $K$ as in the section on $L_{20e}$ and let $B$ be vertex $[0,-\phi^2,-1]$, then $L_{20f} = L_{AB}$. Projecting $B$ onto $\mathcal{A}(A)$, we obtain $\operatorname{proj}_A B =[0,-\frac{5+3 \sqrt{5}}{10},-2-\frac{3}{\sqrt{5}}]$ and separation $h=2 \sqrt{1+\frac{2}{\sqrt{5}}}$. By Point-Polygon formula,
\begin{equation}
\begin{split}
L_{20f} = \frac{2 h^3}{\nu} \Bigg{(} & I^{(1)}_{00}\left(\frac{\phi ^2}{2},\frac{\pi }{5}\right)-I^{(1)}_{00}\left(\frac{1}{2 \phi^2},\frac{2 \pi }{5}\right)+I^{(1)}_{00}\left(\frac{1}{2 \phi ^2},\arctan\sqrt{5 \left(5+2\sqrt{5}\right)}\right)\\
&-I^{(1)}_{00}\left(\frac{1}{2},\frac{\pi }{5}\right)-I^{(1)}_{00}\left(\frac{\phi^2}{2},\arctan\sqrt{85-38 \sqrt{5}}\right)\Bigg{)} \approx 4.000363965317.
\end{split}
\end{equation}
Explicitly, after series of simplifications,
\begin{equation}
    \begin{split}
   L_{20f} & = 1+\frac{\sqrt{\frac{2}{5}}}{3}-\frac{\sqrt{2}}{5}+\frac{7}{3 \sqrt{5}}+\frac{4 \pi }{15}+\frac{8 \pi }{15 \sqrt{5}}+\frac{2}{3}
   \left(2+\sqrt{5}\right) \ln 3+\frac{13}{300} \left(5+2 \sqrt{5}\right) \ln 5\\
   &-\frac{16}{75} \left(5+2 \sqrt{5}\right) \arccos\frac{2}{3}+\frac{8}{75} \left(5+2 \sqrt{5}\right) \left(\arccot 2+\arctan\frac{5 \sqrt{2}}{7}-\arctan\left(7 \sqrt{2}\right)\right)\\
   &+\left(\frac{47}{30}+\frac{52}{15 \sqrt{5}}\right) \arccosh\frac{7}{3}-\left(\frac{47}{30}+\frac{52}{15 \sqrt{5}}\right) \arccosh 3+\frac{1}{150} \left(35+4 \sqrt{5}\right) \arccosh\frac{13}{3}\\
   & -\frac{13}{150} \left(5+2 \sqrt{5}\right) \arcsinh 2.
    \end{split}
\end{equation}

\phantomsection
\subsubsection{\texorpdfstring{L\textsubscript{22r}}{L22r}}
%\addcontentsline{toc}{subsubsection}{L22r}
Finally, let us take a closer look on parallel configurations $L_{21r}$ and $L_{22r}$. We start with the latter. Let $A$ and $B$ be opposite faces of dodecahedron $K$ with separation $h=\sqrt{10+\frac{22}{\sqrt{5}}}$ then $L_{22}=L_{AB}$ with overlap diagram as seen in Figure \ref{fig:OverlapDode}. Note that, due to symmetry, only one tenth of the diagram is sufficient to be considered. The subdomains where $\vol{A \cap \operatorname{proj}B + k}$ can be written as a single polynomial are shown in the diagram. Again, they are labeled by number of sides of polygon of intersection $A \cap (\operatorname{proj}B + k)$, sliding $\operatorname{proj} B+k$ across $\operatorname{proj} A$ by letting $k$ to vary (vector $k$ is shown by a black dot). Let us denote $D$ as the union of the labeled subdomains. Then, by Overlap formula,
\begin{equation}
L_{22r} = \frac{10}{\volA \volB} \int_D \sqrt{h^2 + k^2}\vol{A \cap \operatorname{proj}B + k} \ddd k,
\end{equation}

\begin{figure}[h]
    \centering
     \includegraphics[width=0.6\textwidth]{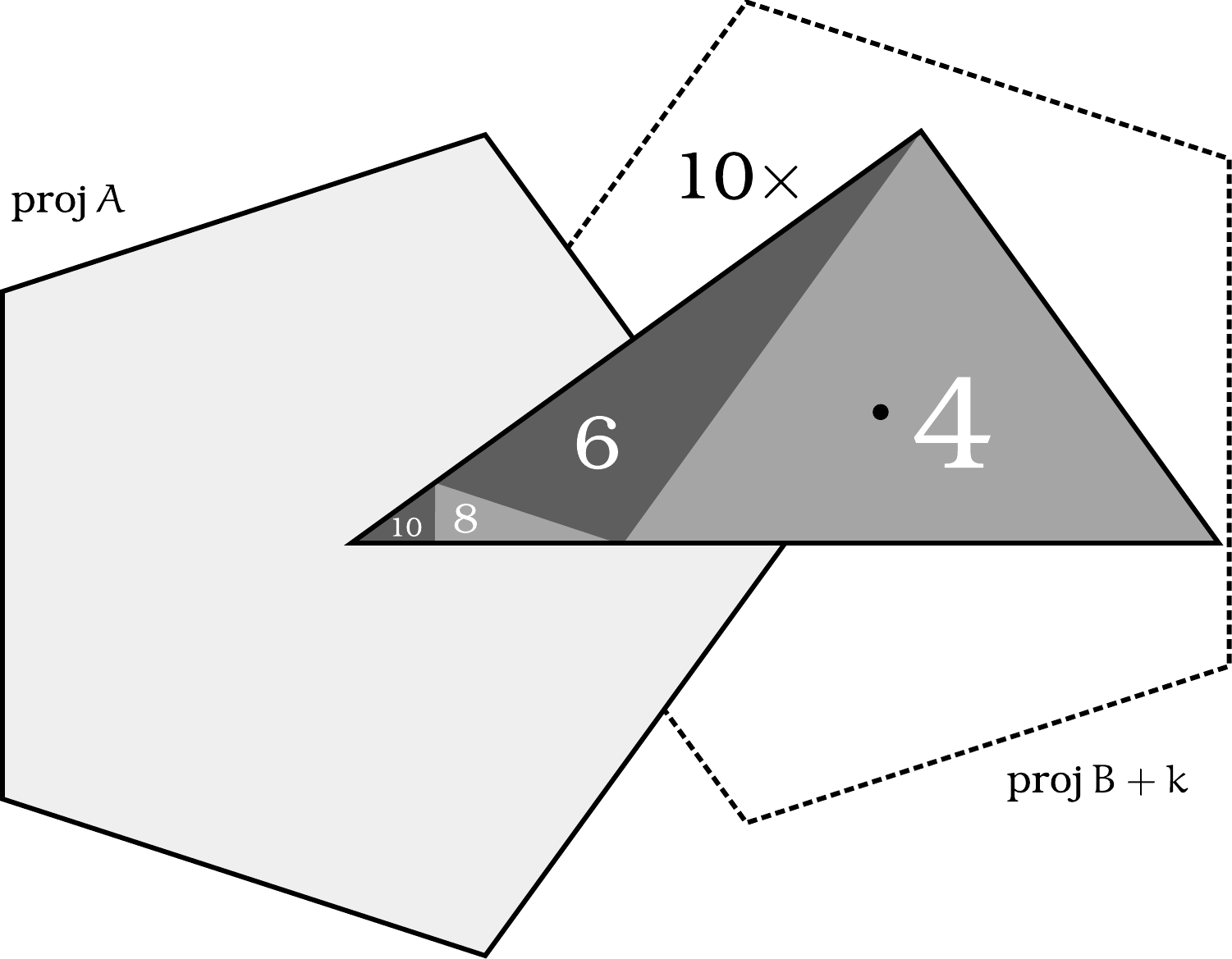}
    \caption{Overlap diagram for opposite-faces configuration in dodecahedron}
    \label{fig:OverlapDode}
\end{figure}

Let us express $\vol{A \cap \operatorname{proj}B + k}$ in the aforementioned subdomains. We denote $v_j = \vol{A \cap \operatorname{proj}B + k}$ for all $k \in D_j$. Let us restrict ourselves to the plane $\mathcal{A}(A)$, in which we put $k=(x,y)$ and in which $\operatorname{proj}A$ is a regular pentagon with vertices $\left[\sqrt{2+\frac{2}{\sqrt{5}}}\cos\frac{2\pi i}{5},\sqrt{2+\frac{2}{\sqrt{5}}}\sin\frac{2\pi i}{5}\right]$, $i\in\{0,1,2,3,4\}$ and area $\volA = \sqrt{5 \left(5+2 \sqrt{5}\right)}$. Similarly, $\operatorname{proj}B$ is another pentagon with vertices $\left[\sqrt{2+\frac{2}{\sqrt{5}}}\cos\frac{2\pi (i+1/2)}{5},\sqrt{2+\frac{2}{\sqrt{5}}}\sin\frac{2\pi (i+1/2)}{5}\right]$ and area $\volB = \volA = \sqrt{5 \left(5+2 \sqrt{5}\right)}$. Under this projection, the labeled subdomains $D_j$ are triangles with vertices
\begin{itemize}
    \item (subdomain $D_4$) $\left[\sqrt{\frac{5}{2}+\frac{11}{2 \sqrt{5}}},\frac{1}{2} \left(1+\sqrt{5}\right)\right], \left[\sqrt{2-\frac{2}{\sqrt{5}}},0\right], \left[\sqrt{8+\frac{8}{\sqrt{5}}},0\right]$, in which
    \begin{equation}
        v_4 =  \frac{1}{10} \left(8 \sqrt{50+22 \sqrt{5}}-4 \left(5+3 \sqrt{5}\right) x+\sqrt{5 \left(5+2 \sqrt{5}\right)} x^2-5 \sqrt{5-2 \sqrt{5}} y^2\right),
    \end{equation}
    \item (subdomain $D_6$) $\left[\sqrt{\frac{5}{2}+\frac{11}{2 \sqrt{5}}},\frac{1}{2} \left(1+\sqrt{5}\right)\right], \left[\sqrt{1-\frac{2}{\sqrt{5}}},\sqrt{5}-2\right], \left[\sqrt{2-\frac{2}{\sqrt{5}}},0\right]$, in which
    \begin{equation}
        \begin{split}
        v_6 & = \frac{1}{20} \bigg{(}8 \sqrt{145+62 \sqrt{5}}-4 \left(5+3 \sqrt{5}\right) x-\sqrt{10 \left(5+\sqrt{5}\right)} x^2\\
        & -4 \sqrt{10\left(5+\sqrt{5}\right)} y+10 \left(1+\sqrt{5}\right) x y-5 \sqrt{2 \left(5+\sqrt{5}\right)} y^2\bigg{)},        
        \end{split}
    \end{equation}
    \item (subdomain $D_8$) $\left[\sqrt{1-\frac{2}{\sqrt{5}}},\sqrt{5}-2\right], \left[\sqrt{1-\frac{2}{\sqrt{5}}},0\right], \left[\sqrt{2-\frac{2}{\sqrt{5}}},0\right]$, in which
    \begin{equation}
        v_8 = \frac{1}{10} \left(4 \sqrt{130+58 \sqrt{5}}-8 \sqrt{5} x-5 \sqrt{1+\frac{2}{\sqrt{5}}} x^2-5 \sqrt{5 \left(5+2 \sqrt{5}\right)} y^2\right),
    \end{equation}
    \item (subdomain $D_{10}$) $\left[\sqrt{1-\frac{2}{\sqrt{5}}},\sqrt{5}-2\right], \left[0,0\right], \left[\sqrt{1-\frac{2}{\sqrt{5}}},0\right]$, in which
    \begin{equation}
        v_{10} = \frac{1}{2} \sqrt{5+2 \sqrt{5}} \left(4-\sqrt{5} x^2-\sqrt{5} y^2\right).
    \end{equation}
\end{itemize}

In order to use the Overlap formula effectively, that is, to integrate $v_j = \vol{A \cap \operatorname{proj}B + k}$, $k \in D_j$ over all subdomains $D_j$, it is convenient to first perform appropriate rotation transformations and inclusion/exclusions. First, by inclusion/exclusion,
\begin{equation}
\begin{split}
L_{22r} & = \frac{10}{\volA \volB} \bigg{(} \int_{D_{10}} u_{10} \sqrt{h^2 + x^2+y^2} \ddd x \dd y+\int_{D_{10}\cup D_{8}} u_8 \sqrt{h^2 + x^2+y^2} \ddd x \dd y\\
& +\int_{D_{10} \cup D_{8} \cup D_{6}} u_6 \sqrt{h^2 + x^2+y^2} \ddd x \dd y+\int_{D_{10} \cup D_{8} \cup D_{6} \cup D_{4}} u_4 \sqrt{h^2 + x^2+y^2} \ddd x \dd y\bigg{)},
\end{split}
\end{equation}
where $u_4 = v_4$, $u_6=v_6-v_4$, $u_8 = v_8 - v_6$ and $u_{10} = v_{10}-v_8$. Explicitly
\begin{align}
    u_4 = & \frac{4 \sqrt{2}}{5} \sqrt{25+11 \sqrt{5}}-\frac{2x}{5} \left(5+3 \sqrt{5}\right) +\frac{x^2}{10} \sqrt{5 \left(5+2
   \sqrt{5}\right)} -\frac{y^2}{2} \sqrt{5-2 \sqrt{5}},\\
    \begin{split}
    u_6 = & -\frac{2}{5} \sqrt{5+2 \sqrt{5}}+x\left(1+\frac{3}{\sqrt{5}}\right)-y \sqrt{2+\frac{2}{\sqrt{5}}}+\frac{x y}{2} \left(1+\sqrt{5}\right)\\
    & -\frac{x^2 \sqrt{2}}{4} \sqrt{5+\frac{11}{\sqrt{5}}}-\frac{y^2 \sqrt{2}}{4} \sqrt{5-\sqrt{5}},     
    \end{split}\\
    u_8 = & -\frac{2}{5} \sqrt{5-2 \sqrt{5}}+x \left(1-\frac{1}{\sqrt{5}}\right)+y \sqrt{2+\frac{2}{\sqrt{5}}}-\frac{x y}{2} \left(1+\sqrt{5}\right)\\
    &-\frac{x^2 \sqrt{2}}{4} \sqrt{1-\frac{1}{\sqrt{5}}}-\frac{y^2 \sqrt{2}}{4} \sqrt{25+11 \sqrt{5}},\\
    \begin{split}
    u_{10} = & -\frac{2}{5}\sqrt{5-2 \sqrt{5}}+\frac{4 x}{\sqrt{5}}-2 x^2 \sqrt{1+\frac{2}{\sqrt{5}}}.
    \end{split}
\end{align}
Note that domain $D_{10}$ is already in the form of the fundamental triangle domain $D\left(\zeta,\frac{\pi}{5}\right)$ with $\zeta = \sqrt{1-\frac{2}{\sqrt{5}}}$. Since $\zeta/h = 1/(2\phi^4)$, we immediately get in terms of auxiliary integrals,
\begin{equation}
\begin{split}
    \int_{D_{10}} u_{10} \sqrt{h^2 + x^2+y^2} \ddd x \dd y = h^3 \Bigg{(} & -\frac{2}{5}\sqrt{5-2 \sqrt{5}}I^{(1)}_{00}\left(\frac{1}{2 \phi ^4},\frac{\pi}{5}\right)+\frac{4 h}{\sqrt{5}}I^{(1)}_{10}\left(\frac{1}{2 \phi ^4},\frac{\pi}{5}\right)\\
    &-2 h^2 \sqrt{1+\frac{2}{\sqrt{5}}}I^{(1)}_{20}\left(\frac{1}{2 \phi ^4},\frac{\pi}{5}\right)\Bigg{)}.  
\end{split}
\end{equation}
Domain $D = D_{10} \cup D_{8} \cup D_{6} \cup D_{4}$ in $(x,y)$ is transformed to the fundamental domain $D(\zeta,\frac{\pi}{5})$ with $\zeta = 2 \sqrt{1+\frac{2}{\sqrt{5}}}$ in $(x',y')$ via polar angle substitution $\varphi = \pi/5 - \varphi'$, that is $x = r \cos(\frac{\pi}{5}-\varphi')$ and $y = r \sin(\frac{\pi}{5}-\varphi')$. Expanding out the trigonometric functions and writing $x' = r \cos \varphi'$ and $y' = r \cos\varphi'$, we get the following transformation relations
\begin{equation}
    x = \frac{1}{4} \left(1+\sqrt{5}\right) x'+\sqrt{\frac{5}{8}-\frac{\sqrt{5}}{8}} y', \qquad y = \sqrt{\frac{5}{8}-\frac{\sqrt{5}}{8}} x'-\frac{1}{4} \left(1+\sqrt{5}\right) y'
\end{equation}
and so
\begin{equation}
   u_4 = \frac{4}{5} \sqrt{50+22 \sqrt{5}}-x' \left(2+\frac{4}{\sqrt{5}}\right)-2 y' \sqrt{1+\frac{2}{\sqrt{5}}}+x' y'+x'^2
   \sqrt{1-\frac{2}{\sqrt{5}}}.
\end{equation}
Since $\zeta/h = 1/\phi$, we immediately get
\begin{equation}
\begin{split}
    & \int_{D} u_4 \sqrt{h^2 + x^2+y^2} \ddd x \dd y = h^3 \Bigg{(} \frac{4}{5} \sqrt{50+22 \sqrt{5}}I^{(1)}_{00}\left(\frac{1}{\phi},\frac{\pi}{5}\right)-h \left(2+\frac{4}{\sqrt{5}}\right)I^{(1)}_{10}\left(\frac{1}{\phi},\frac{\pi}{5}\right)\\
    &-2 h \sqrt{1+\frac{2}{\sqrt{5}}}I^{(1)}_{01}\left(\frac{1}{\phi},\frac{\pi}{5}\right)+h^2I^{(1)}_{11}\left(\frac{1}{\phi},\frac{\pi}{5}\right)+h^2\sqrt{1-\frac{2}{\sqrt{5}}}I^{(1)}_{20}\left(\frac{1}{\phi},\frac{\pi}{5}\right)\Bigg{)}.
\end{split}
\end{equation}
In order to express the remaining integrals, we write $D_{10}\cup D_8\cup D_6 = E_4 \setminus E_6$ and $D_{10}\cup D_8 = E_8 \setminus E_{10}$, where
\begin{itemize}
    \item $E_4$ is a triangle with vertices $\left[\sqrt{\frac{5}{2}+\frac{11}{2 \sqrt{5}}},\frac{1}{2} \left(1+\sqrt{5}\right)\right]$, $[0,0]$, $\left[\frac{1}{2} \sqrt{1+\frac{2}{\sqrt{5}}},-\frac{1}{2}\right]$,
    \item $E_6$ is a triangle with vertices $\left[\sqrt{2-\frac{2}{\sqrt{5}}},0\right]$, $[0,0]$, $\left[\frac{1}{2} \sqrt{1+\frac{2}{\sqrt{5}}},-\frac{1}{2}\right]$,
    \item $E_8$ is a triangle with vertices $\left[\frac{1}{2} \sqrt{\frac{5}{2}-\frac{11}{2 \sqrt{5}}},\frac{1}{4} \left(\sqrt{5}-1\right)\right]$, $[0,0]$, $\left[\frac{1}{2} \sqrt{\frac{5}{2}-\frac{11}{2 \sqrt{5}}},\frac{1}{4} \left(\sqrt{5}-1\right)\right]$,
    \item $E_{10}$ is a triangle with vertices $\left[\frac{1}{2} \sqrt{\frac{5}{2}-\frac{11}{2 \sqrt{5}}},\frac{1}{4} \left(\sqrt{5}-1\right)\right]$, $[0,0]$, $\left[\sqrt{1-\frac{2}{\sqrt{5}}},\sqrt{5}-2\right]$.
\end{itemize}
Note that $E_6 \subset E_4$ and $E_{10} \subset E_8$ and thus
\begin{equation}
\begin{split}
    \int_{D_{10}\cup D_8 \cup D_6} u_6 \sqrt{h^2 + x^2+y^2} \ddd x \dd y & = \int_{E_4} u_6 \sqrt{h^2 + x^2+y^2} \ddd x \dd y - \int_{E_6} u_6 \sqrt{h^2 + x^2+y^2} \ddd x \dd y,\\
    \int_{D_{10} \cup D_8} u_8 \sqrt{h^2 + x^2+y^2} \ddd x \dd y & = \int_{E_8} u_8 \sqrt{h^2 + x^2+y^2} \ddd x \dd y - \int_{E_{10}} u_8 \sqrt{h^2 + x^2+y^2} \ddd x \dd y.\\
\end{split}
\end{equation}
Domains $E_4, E_6, E_8$ and $E_{10}$ can be rotated to fundamental triangle domains after appropriate rotations. First, let $\varphi = \varphi'-\pi/5$, so
\begin{equation}
    x = \frac{1}{4} \left(1+\sqrt{5}\right) x'+\sqrt{\frac{5}{8}-\frac{\sqrt{5}}{8}} y', \qquad y = -\sqrt{\frac{5}{8}-\frac{\sqrt{5}}{8}} x'+\frac{1}{4} \left(1+\sqrt{5}\right) y'
\end{equation}
and thus, after simplifications,
\begin{equation}
    u_6 = -\frac{2}{5} \sqrt{5+2 \sqrt{5}}+2 x' \left(1+\frac{1}{\sqrt{5}}\right)-x'^2 \sqrt{2+\frac{2}{\sqrt{5}}}.
\end{equation}
Suddenly in $(x',y')$, we have $E_4 = D(\zeta,\frac{2\pi}{5})$ and $E_6 = D(\zeta,\frac{\pi}{5})$ with $\zeta = \sqrt{\frac{1}{10} \left(5+\sqrt{5}\right)}$, hence $\zeta/h = 1/(2\phi^2)$ and immediately in terms of auxiliary integrals,
\begin{equation}
\begin{split}
    & \int_{D_{10} \cup D_8 \cup D_6} u_6 \sqrt{h^2 + x^2+y^2} \ddd x \dd y = h^3 \Bigg{(} -\frac{2}{5} \sqrt{5+2 \sqrt{5}}\left(I^{(1)}_{00}\left(\frac{1}{2\phi^2},\frac{2\pi}{5}\right)-I^{(1)}_{00}\left(\frac{1}{2\phi^2},\frac{\pi}{5}\right)\right)\\
    & \!+\!2 h \left(\!1\!+\!\frac{1}{\sqrt{5}}\!\right)\left(\!I^{(1)}_{10}\left(\frac{1}{2\phi^2},\frac{2\pi}{5}\right)\!-\! I^{(1)}_{10}\left(\frac{1}{2\phi^2},\frac{\pi}{5}\right)\right)\!-\! h^2 \sqrt{2\!+\!\frac{2}{\sqrt{5}}}\left(\! I^{(1)}_{20}\left(\frac{1}{2\phi^2},\frac{2\pi}{5}\right)\!-\! I^{(1)}_{20}\left(\frac{1}{2\phi^2},\frac{\pi}{5}\right)\right)\!\Bigg{)}.
\end{split}
\end{equation}
Next, let $\varphi = \frac{2\pi}{5} - \varphi'$, from which we obtain transformation relations
\begin{equation}
    x = \frac{1}{4} \left(\sqrt{5}-1\right) x'+\sqrt{\frac{5}{8}+\frac{\sqrt{5}}{8}} y', \qquad y = \sqrt{\frac{5}{8}+\frac{\sqrt{5}}{8}} x'-\frac{1}{4} \left(\sqrt{5}-1\right) y',
\end{equation}
so
\begin{equation}
u_8 = -\frac{2}{5} \sqrt{5-2 \sqrt{5}}+\frac{4 x'}{\sqrt{5}}-2 x'^2 \sqrt{1+\frac{2}{\sqrt{5}}}.
\end{equation}
In $(x',y')$, we have $E_8 = D(\zeta,\frac{2\pi}{5})$ and $E_{10} = D(\zeta,\frac{\pi}{5})$ with $\zeta = \sqrt{1-\frac{2}{\sqrt{5}}}$, hence $\zeta/h = 1/(2\phi^4)$ and immediately in terms of auxiliary integrals,
\begin{equation}
\begin{split}
    & \int_{D_{10} \cup D_8} u_8 \sqrt{h^2 + x^2+y^2} \ddd x \dd y = h^3 \Bigg{(} -\frac{2}{5} \sqrt{5-2 \sqrt{5}}\left(I^{(1)}_{00}\left(\frac{1}{2\phi^4},\frac{2\pi}{5}\right)-I^{(1)}_{00}\left(\frac{1}{2\phi^4},\frac{\pi}{5}\right)\right)\\
    & +\frac{4 h}{\sqrt{5}}\left(I^{(1)}_{10}\left(\frac{1}{2\phi^4},\frac{2\pi}{5}\right)- I^{(1)}_{10}\left(\frac{1}{2\phi^4},\frac{\pi}{5}\right)\right)- 2h^2 \sqrt{1+\frac{2}{\sqrt{5}}}\left( I^{(1)}_{20}\left(\frac{1}{2\phi^4},\frac{2\pi}{5}\right)- I^{(1)}_{20}\left(\frac{1}{2\phi^2},\frac{\pi}{5}\right)\right)\Bigg{)}.
\end{split}
\end{equation}
Therefore, in total,
\begin{align*}
    & L_{22r} = \frac{10 h^3}{\volA \volB} \Bigg{(} \frac{2}{5} \sqrt{5+2 \sqrt{5}} I^{(1)}_{00}\left(\frac{1}{2 \phi ^2},\frac{\pi }{5}\right)-\frac{2}{5} \sqrt{5+2 \sqrt{5}}I^{(1)}_{00}\left(\frac{1}{2 \phi ^2},\frac{2 \pi }{5}\right)\\
    & -\frac{2}{5} \sqrt{5-2 \sqrt{5}} I^{(1)}_{00}\left(\frac{1}{2 \phi ^4},\frac{2 \pi}{5}\right)+\frac{4}{5} \sqrt{50+22 \sqrt{5}} I^{(1)}_{00}\left(\frac{1}{\phi },\frac{\pi }{5}\right) -2 \sqrt{\frac{1}{5} \left(5+2 \sqrt{5}\right)} h I^{(1)}_{01}\left(\frac{1}{\phi },\frac{\pi }{5}\right)\\
    &+\frac{4 h I^{(1)}_{10}\left(\frac{1}{2\phi ^4},\frac{2 \pi }{5}\right)}{\sqrt{5}}-\frac{2 h I^{(1)}_{10}\left(\frac{1}{2 \phi ^2},\frac{\pi }{5}\right)}{\sqrt{5}}-2 h I^{(1)}_{10}\left(\frac{1}{2 \phi ^2},\frac{\pi }{5}\right)+\frac{2 h I^{(1)}_{10}\left(\frac{1}{2 \phi ^2},\frac{2 \pi }{5}\right)}{\sqrt{5}}+2h I^{(1)}_{10}\left(\frac{1}{2 \phi ^2},\frac{2 \pi }{5}\right)\\
    & -\frac{4 h I^{(1)}_{10}\left(\frac{1}{\phi },\frac{\pi }{5}\right)}{\sqrt{5}}-2 h I^{(1)}_{10}\left(\frac{1}{\phi },\frac{\pi }{5}\right) + h^2 I^{(1)}_{11}\left(\frac{1}{\phi },\frac{\pi }{5}\right)-2 \sqrt{1+\frac{2}{\sqrt{5}}} h^2 I^{(1)}_{20}\left(\frac{1}{2 \phi ^4},\frac{2 \pi}{5}\right)\\
    &+\sqrt{2+\frac{2}{\sqrt{5}}} h^2 I^{(1)}_{20}\left(\frac{1}{2 \phi ^2},\frac{\pi }{5}\right)-\sqrt{2+\frac{2}{\sqrt{5}}} h^2 I^{(1)}_{20}\left(\frac{1}{2 \phi ^2},\frac{2 \pi }{5}\right)+\sqrt{1-\frac{2}{\sqrt{5}}} h^2 I^{(1)}_{20}\left(\frac{1}{\phi },\frac{\pi}{5}\right)\Bigg{)}\\
    & \approx 4.69357209587.
\end{align*}
Or explicitly, after a lot of simplifications,
\begin{equation}
\begin{split}
    & L_{22r} = \frac{2 \sqrt{\frac{2}{5}}}{15}-\frac{38}{75}-\frac{4 \sqrt{2}}{75}+\frac{44}{25 \sqrt{3}}-\frac{88}{75 \sqrt{5}}+\frac{116}{25\sqrt{15}}-\frac{8 \left(1839+820 \sqrt{5}\right) \pi }{1125}\\
    &+\frac{16}{125} \left(67+30 \sqrt{5}\right) \arccos\frac{2}{3}+\frac{16}{375} \left(388+173 \sqrt{5}\right) \left(\arccos\frac{1}{\sqrt{41}}+\arccos\frac{3}{\sqrt{41}}\right)\\
    &+\frac{2}{375} \left(817+371 \sqrt{5}\right) \left(\arccosh 2-\arccosh 4\right)+\frac{1}{250} \left(1833+820 \sqrt{5}\right) \left(\arccosh\frac{7}{3}-\arccosh 3\right)\\
    &+\frac{1}{750}\left(3538+1523 \sqrt{5}\right) \left(\arccosh\frac{9}{\sqrt{41}}-\arccosh\frac{7}{\sqrt{41}}\right)-\frac{32}{125} \left(67+30 \sqrt{5}\right) \arccot\sqrt{2}\\
   &+\frac{1}{250}\left(1833+820 \sqrt{5}\right) \ln 3.
\end{split}
\end{equation}

\phantomsection
\subsubsection{\texorpdfstring{L\textsubscript{21r}}{L21r}}
%\addcontentsline{toc}{subsubsection}{L21r}
By definition, $L_{21r} = L_{AB}$, where $A$ is a face of $K$ and $B$ is the perimeter of its corresponding opposite face. Again, we use the Overlap formula to deduce the value of $L_{21r}$, that is, by symmetry,
\begin{equation}
L_{21r} = \frac{10}{\volA \volB} \int_D \sqrt{h^2 + k^2}\vol{A \cap \operatorname{proj}B + k} \ddd k,
\end{equation}
where $\volA=\sqrt{5 \left(5+2 \sqrt{5}\right)}$ is the area of $A$ and $\volB=10$ is the length of $B$. The overlap diagram is the same as in the case of $L_{22r}$, although the value $\vol{A \cap \operatorname{proj}B + k}$ now corresponds to the total length of polyline $A \cap (\operatorname{proj}B + k)$ of intersection. In order to keep the naming of the subdomains $D_j$ and functions $v_j = \vol{A \cap \operatorname{proj}B + k}$, $k \in D_j$ the same as in the case of $L_{22r}$, we let $j$, exceptionally, to denote \emph{twice} the number line segments of $A \cap (\operatorname{proj}B + k)$ in this section. That way, we get $D = D_{10} \cup D_8 \cup D_6 \cup D_4$ and
\begin{align*}
    v_4 & =  4+\frac{4}{\sqrt{5}}-x \sqrt{2+\frac{2}{\sqrt{5}}}, && \,\, v_6 =  4+\frac{2}{\sqrt{5}}-\frac{x}{2} \sqrt{2+\frac{2}{\sqrt{5}}}+\frac{y}{2} \left(1-\sqrt{5}\right),\\
    v_8 & = 2+\frac{6}{\sqrt{5}}-2 x \sqrt{1-\frac{2}{\sqrt{5}}}, &&
    v_{10} = 2 \sqrt{5}.
\end{align*}
Let $u_4 = v_4$, $u_6 = v_6-v_4$, $u_8 = v_8 - v_6$, $u_{10} = v_{10} - v_8$, that is
\begin{align*}
    u_4 & =  4+\frac{4}{\sqrt{5}}-x \sqrt{2+\frac{2}{\sqrt{5}}}, && \,\, u_6 =  -\frac{2}{\sqrt{5}}+\frac{x}{2} \sqrt{2+\frac{2}{\sqrt{5}}}+\frac{y}{2} \left(1-\sqrt{5}\right),\\
    u_8 & = -2+\frac{4}{\sqrt{5}}+x \sqrt{\frac{5}{2}-\frac{11}{2 \sqrt{5}}}-\frac{y}{2} \left(1-\sqrt{5}\right), &&
    u_{10} = -2+\frac{4}{\sqrt{5}}+2 \sqrt{1-\frac{2}{\sqrt{5}}} x.
\end{align*}
Overall, by inclusion/exclusion,
\begin{equation}
\begin{split}
L_{21r} = \frac{10}{\volA \volB} \bigg{(} & \int_{D_{10}} u_{10} \sqrt{h^2 + x^2+y^2} \ddd x \dd y+\int_{D} u_4 \sqrt{h^2 + x^2+y^2} \ddd x \dd y\\
 +&\int_{E_4} u_6 \sqrt{h^2 + x^2+y^2} \ddd x \dd y-\int_{E_6} u_6 \sqrt{h^2 + x^2+y^2} \ddd x \dd y\\
+&\int_{E_8} u_8 \sqrt{h^2 + x^2+y^2} \ddd x \dd y - \int_{E_{10}} u_8 \sqrt{h^2 + x^2+y^2} \ddd x \dd y\bigg{)},\end{split}
\end{equation}
The first integral can be immediately expressed in terms of auxiliary integrals
\begin{equation}
    \int_{D_{10}} u_{10} \sqrt{h^2+x^2+y^2} \ddd x \dd y = h^3 \left(\left(-2+\frac{4}{\sqrt{5}}\right) I^{(1)}_{00}\left(\frac{1}{2 \phi ^4},\frac{\pi }{5}\right)+2 h \sqrt{1-\frac{2}{\sqrt{5}}} I^{(1)}_{10}\left(\frac{1}{2 \phi ^4},\frac{\pi }{5}\right)\right).
\end{equation}
Performing the same set of transformations as in the previous case of $L_{22r}$, that is
\begin{itemize}
    \item $\varphi = \pi/5 - \varphi'$, we get $u_{4} = 4+\frac{4}{\sqrt{5}}-x' \sqrt{1+\frac{2}{\sqrt{5}}}-y'$,
    \item $\varphi = \varphi'-\pi/5$, we get $u_{6} = -\frac{2}{\sqrt{5}}+x' \sqrt{2-\frac{2}{\sqrt{5}}}$,
    \item $\varphi = \frac{2\pi}{5} - \varphi'$, we get $u_{8} = -2+\frac{4}{\sqrt{5}}+2 x' \sqrt{1-\frac{2}{\sqrt{5}}}$
\end{itemize}
and as a result, since all the subdomains are now expressed as fundamental triangle domains, we get
\begin{align*}
   & \int_{D} u_4 \sqrt{h^2\! +\! x^2\!+\!y^2} \ddd x \dd y = h^3 \Bigg{(}\! \left(4\!+\!\frac{4}{\sqrt{5}}\right)I^{(1)}_{00}\left(\frac{1}{\phi},\frac{\pi}{5}\right)\!-\! h \sqrt{1\!+\!\frac{2}{\sqrt{5}}} I^{(1)}_{10}\left(\frac{1}{\phi},\frac{\pi}{5}\right)\! -\! h I^{(1)}_{01}\left(\frac{1}{\phi},\frac{\pi}{5}\right)\!\Bigg{)},\\
\begin{split}
    &  \int_{D_{10} \cup D_8 \cup D_6} u_6 \sqrt{h^2 \!+\! x^2\!+\!y^2} \ddd x \dd y = h^3 \Bigg{(} -\frac{2}{\sqrt{5}}\left(I^{(1)}_{00}\left(\frac{1}{2\phi^2},\frac{2\pi}{5}\right)-I^{(1)}_{00}\left(\frac{1}{2\phi^2},\frac{\pi}{5}\right)\right)\\
    & \hspace{14em} +h \sqrt{2-\frac{2}{\sqrt{5}}} \left(1+\frac{1}{\sqrt{5}}\right)\left(I^{(1)}_{10}\left(\frac{1}{2\phi^2},\frac{2\pi}{5}\right)- I^{(1)}_{10}\left(\frac{1}{2\phi^2},\frac{\pi}{5}\right)\right)\Bigg{)},
\end{split}\\
\begin{split}
    & \int_{D_{10} \cup D_8} u_8 \sqrt{h^2 \!+\! x^2\!+\!y^2} \ddd x \dd y = h^3 \Bigg{(} \left(\frac{4}{\sqrt{5}}-2\right)\left(I^{(1)}_{00}\left(\frac{1}{2\phi^4},\frac{2\pi}{5}\right)-I^{(1)}_{00}\left(\frac{1}{2\phi^4},\frac{\pi}{5}\right)\right)\\
    & \hspace{12.6em} +2 h \sqrt{1-\frac{2}{\sqrt{5}}}\left(I^{(1)}_{10}\left(\frac{1}{2\phi^4},\frac{2\pi}{5}\right)- I^{(1)}_{10}\left(\frac{1}{2\phi^4},\frac{\pi}{5}\right)\right)\Bigg{)}.
\end{split}
\end{align*}
Therefore, in total
Therefore, in total,
\begin{align*}
    L_{21r} & = \frac{10 h^3}{\volA \volB} \Bigg{(} \left(\frac{4}{\sqrt{5}}-2\right) I^{(1)}_{00}\left(\frac{1}{2 \phi ^4},\frac{2 \pi }{5}\right)+\frac{2 I^{(1)}_{00}\left(\frac{1}{2 \phi^2},\frac{\pi }{5}\right)}{\sqrt{5}}-\frac{2 I^{(1)}_{00}\left(\frac{1}{2 \phi ^2},\frac{2 \pi }{5}\right)}{\sqrt{5}}\\
   &+\frac{4}{5}\left(5+\sqrt{5}\right) I^{(1)}_{00}\left(\frac{1}{\phi },\frac{\pi }{5}\right)-h I^{(1)}_{01}\left(\frac{1}{\phi },\frac{\pi }{5}\right)+2 \sqrt{1-\frac{2}{\sqrt{5}}} h I^{(1)}_{10}\left(\frac{1}{2 \phi ^4},\frac{2 \pi}{5}\right)\\
   &-\sqrt{2-\frac{2}{\sqrt{5}}} h I^{(1)}_{10}\left(\frac{1}{2 \phi ^2},\frac{\pi }{5}\right)+\sqrt{2-\frac{2}{\sqrt{5}}} hI^{(1)}_{10}\left(\frac{1}{2 \phi ^2},\frac{2 \pi }{5}\right)-\sqrt{1+\frac{2}{\sqrt{5}}} h I^{(1)}_{10}\left(\frac{1}{\phi },\frac{\pi
   }{5}\right)\Bigg{)}\\
    & \approx 4.808558828667.
\end{align*}
Or explicitly,
\begin{equation}
\begin{split}
    & L_{21r} = \frac{149}{30}-\frac{29 \sqrt{\frac{3}{5}}}{5}-\frac{\sqrt{2}}{15}-\frac{41}{5 \sqrt{3}}+\frac{166}{15 \sqrt{5}}+\frac{1}{3 \sqrt{10}}-\frac{4\pi}{225} \left(19+8 \sqrt{5}\right) -\frac{8}{75} \left(9+4 \sqrt{5}\right) \arccos\frac{2}{3}\\
   & -\frac{8}{75} \left(2+\sqrt{5}\right) \left(\arccos\frac{1}{\sqrt{41}} +\arccos\frac{3}{\sqrt{41}}\right)+\frac{1043+468 \sqrt{5}}{600} \left(\arccosh\frac{7}{3}-\arccosh 3\right)\\
   &+\frac{271+117 \sqrt{5}}{150} \left(\arccosh 4-\arccosh 2 \right)+\frac{746+283\sqrt{5}}{600} \left(\arccosh\frac{9}{\sqrt{41}}-\arccosh\frac{7}{\sqrt{41}}\right)\\
   &+\frac{16}{75} \left(9+4\sqrt{5}\right) \arccot\sqrt{2}+\frac{1}{600} \left(1043+468 \sqrt{5}\right) \ln 3.
\end{split}
\end{equation}

\phantomsection
\subsubsection{\texorpdfstring{L\textsubscript{33}}{L33}}
%\addcontentsline{toc}{subsubsection}{L33}
Putting everything together by using \eqref{Eq:Dodeca}, we finally arrive, after another series of simplifications and inverse trigonometric and hyperbolic identities, at
\begin{align*}
    L_{33} & = \frac{1516}{1575}+\frac{2 \sqrt{\frac{2}{5}}}{45}-\frac{124 \sqrt{\frac{3}{5}}}{175}-\frac{71 \sqrt{2}}{1575}-\frac{12 \sqrt{3}}{35}+\frac{342}{175 \sqrt{5}}+\frac{493 \pi }{23625}+\frac{67 \pi }{945 \sqrt{5}}\\
    &+\frac{\left(397-244 \sqrt{5}\right) \arccot 2}{18900}+\frac{\left(24023+11788 \sqrt{5}\right) \left(\arccos\frac{2}{3}-\arccos\frac{1}{3}\right)}{94500}\\
    &-\frac{\left(461+212 \sqrt{5}\right) \left(\arccos\frac{23}{41}+\arccos\frac{39}{41}\right)}{1000}-\frac{\left(1031+521 \sqrt{5}\right) \arccosh\frac{13}{3}}{75600}\\
    &+\frac{\left(367+163 \sqrt{5}\right) \arccosh 9 }{16800}+\frac{\left(22197+8149 \sqrt{5}\right) \left(\arccosh\frac{121}{41}-\arccosh\frac{57}{41}\right)}{84000}\\
    &+\frac{\left(15763+7063 \sqrt{5}\right) \left(\arccosh\frac{7}{3}-\arccosh 3 \right)}{21000}+\frac{2}{875} \left(423+187 \sqrt{5}\right) \left(\arccosh 4 -\arccosh 2 \right)\\
    &+\frac{\left(288889+129739 \sqrt{5}\right) \ln 3}{378000}+\frac{\left(109-3143 \sqrt{5}\right) \ln 5}{151200} \approx 2.533488631644.
\end{align*}
Rescaling, we get our mean distance in a regular dodecahedron having unit volume
\begin{equation}
    L_{33}\big{|}_{\volK = 1} = \frac{L_{33}}{\sqrt[3]{30+14\sqrt{5}}} \approx 0.65853073.
\end{equation}

\begin{comment}
XXX section "Mean distance in regular polygons" is found in separate file polygons.tex
\end{comment}

\begin{comment}
\section{Asymptotic behaviour in extreme polyhedra}
In \cite{bonnet2021sharp}, the authors introduce parametric families $K_\delta$ and $K'_{\delta}$ of polyhedra whose $\Gamma_{KK}$ approach in the limit the sharp lower and the upper bound of $\Gamma_{KK}$ valid for any convex compact $K$. The family $K_\delta$ is consisted of polyhedra created as a convex hull of vertices $[0,0,1],[0,0,-1],[\delta,0,0],[0,\delta,0]$ and $K'_\delta$ is a box with edge lengths $1,\delta,\delta$, respectively. XXX
\end{comment}

\section{Further remarks}
\subsection{Weights}
We believe that the equation for weights \eqref{EqWeights} possesses a closed form solution in terms of geometrical properties of convex non-parallel polyhedra. However, we were unable to deduce that.

\subsection{General convex polyhedra}
Let $K \in \mathcal{\mathbb{R}}^d$, then for any fixed $p>-d$, $L^{(p)}_{KK}$ is \emph{continuous} with respect to continuous transformations of $K$. Hence, in principle, we could obtain the formula for convex parallel polyhedra by a continuous limit from some convex non-parallel polyhedron. However, were not able to perform this limit.

\subsection{Bounds on moments}
Also, we believe, since the value $p=1$ is not special, there could be a bound on $L_{KK}^{(p)}$ similar that of Bonnet, Gusakova, Th\"{a}le and Zaporozhets \cite{bonnet2021sharp}.

% BIBLIOGRAPHY - BIBTEX
%\bibliography{bibliography}

% BIBLIOGRAPHY - BIBLATEX
\printbibliography[heading=bibintoc]

\clearpage

\appendix
\section{Exact mean distances in regular polyhedra}

\subsection*{Mean distances in solids of unit volume}
The table below summarises all new results of exact mean distance in various polyhedra. For completeness, the previously known cases of a ball and a cube have been added as well. Each solid $K$ has $\volK = 1$. As usual, $\phi = (1+\sqrt{5})/2$ is the Golden ratio.

%\vfill

\bgroup
% Additional commands control the vertical padding of the following tables:
\renewcommand{\ru}[1]{\rule{0pt}{#1 em}}%changing height in a cell
\newcommand{\tup}{\ru{3}} % change number to increase upper space in table 
\renewcommand{\tupsingle}{\ru{1.8}} % the same but with a different value
\newcommand{\tdown}{3} % change number to increase down space in table
\newcommand{\tdownsingle}{0.5} % change number to increase down space in table (in single row)
\renewcommand{\arraystretch}{1.0}
   \begin{table}[h]
\setlength{\tabcolsep}{2pt} % reduces collumn width
\setlength\jot{-6pt} % math expressions smaller space between rows
\begin{minipage}{\textwidth} % must be there otherwise footnotes not showing
\setlength{\extrarowheight}{2pt} % Adjusting vertical padding in tables
\begin{tabular}{|c|c|}
% \hline $K$ & $L_{KK}$ \\
 \hline
\tupsingle \begin{tabular}{c} \textit{ball}\footnote{trivial} \\ $0.63807479$ \end{tabular} & $\displaystyle\frac{18}{35}\sqrt[3]{\frac{6}{\pi }}$ \\[\tdownsingle em]
\hline
\ru{4.5} \begin{tabular}{c} icosahedron \\ $0.64131249$ \end{tabular} & $\begin{aligned}
  & \frac{1}{2}\sqrt[3]{\frac{9}{5}-\frac{3}{\sqrt{5}}} \bigg{(}\frac{197}{525}+\frac{239}{525 \sqrt{5}}-\frac{44}{525} \sqrt{2+\frac{2}{\sqrt{5}}}-\frac{\left(17226+6269 \sqrt{5}\right) \pi }{157500}\\[1ex]
  & -\frac{\left(2186+1413 \sqrt{5}\right) \arccot \phi}{15750}+\frac{\left(82-75 \sqrt{5}\right)\arccot\left(\phi ^2\right)}{5250}+\frac{4 \left(2139+881 \sqrt{5}\right) \arccsch\phi}{7875}\\[1ex]
    & +\frac{\left(15969+7151 \sqrt{5}\right) \arccoth \phi}{12600}+\frac{\left(4449-1685 \sqrt{5}\right) \ln 3}{42000}-\frac{\left(75783+37789 \sqrt{5}\right) \ln 5}{252000}\bigg{)}
    \end{aligned}$ \\[\tdown em]
    \hline
\ru{8.1} \begin{tabular}{c} dodecahedron \\ $0.64252068$ \end{tabular} & $\begin{aligned}
    & \frac{1}{\sqrt[3]{30+14 \sqrt{5}}}\bigg{(}\frac{1516}{1575}+\frac{2 \sqrt{\frac{2}{5}}}{45}-\frac{124 \sqrt{\frac{3}{5}}}{175}-\frac{71 \sqrt{2}}{1575}-\frac{12 \sqrt{3}}{35}+\frac{342}{175 \sqrt{5}}+\frac{493 \pi
   }{23625}\\[1ex]
    & +\frac{67 \pi }{945 \sqrt{5}}+\frac{\left(397-244 \sqrt{5}\right) \arccot 2}{18900}+\frac{\left(24023+11788 \sqrt{5}\right) \left(\arccos\frac{2}{3}-\arccos\frac{1}{3}\right)}{94500}  \\[1ex]
    & -\frac{\left(461+212 \sqrt{5}\right) \left(\arccos\frac{23}{41}+\arccos\frac{39}{41}\right)}{1000}-\frac{\left(1031+521 \sqrt{5}\right) \arccosh\frac{13}{3}}{75600} \\[1ex]
    & +\frac{\left(367+163 \sqrt{5}\right) \arccosh 9}{16800} +\frac{\left(22197+8149 \sqrt{5}\right) \left(\arccosh\frac{121}{41}-\arccosh\frac{57}{41}\right)}{84000} \\[1ex]
    & +\frac{\left(15763+7063
   \sqrt{5}\right) \left(\arccosh\frac{7}{3}-\arccosh 3\right)}{21000} +\frac{\left(288889+129739 \sqrt{5}\right) \ln 3}{378000} \\[1ex]
   & +\frac{2\left(423+187 \sqrt{5}\right) \left(\arccosh 4-\arccosh 2\right)}{875}+\frac{\left(109-3143 \sqrt{5}\right) \ln 5}{151200}\bigg{)}
\end{aligned}$ \\[\tdown em]
 \hline
\tupsingle \begin{tabular}{c} octahedron \\ $0.65853073$ \end{tabular} & $\displaystyle\sqrt[3]{\frac{3}{4}}\left(\frac{4}{105}\!+\!\frac{13 \sqrt{2}}{105}\!-\!\frac{4 \pi }{45}\!+\!\frac{109 \ln 3}{630 \sqrt{2}}\!+\!\frac{16\arccot\sqrt{2}}{315}\!+\!\frac{158\arccoth\sqrt{2}}{315}\sqrt{2} \right)$ \\[\tdownsingle em]
     \hline
\tupsingle  \begin{tabular}{c} \textit{cube}\footfullcite{robbinsCuLi} \\ $0.66170718$ \end{tabular} & $\displaystyle\frac{4}{105}+\frac{17 \sqrt{2}}{105}-\frac{2 \sqrt{3}}{35}-\frac{\pi }{15}+\frac{1}{5} \arccoth\sqrt{2}+\frac{4}{5} \arccoth\sqrt{3}$ \\[\tdownsingle em]
         \hline
\tupsingle  \begin{tabular}{c} tetrahedron \\ $0.72946242$ \end{tabular} & $\displaystyle\sqrt[3]{3} \left(\frac{\sqrt{2}}{7}-\frac{37 \pi }{315}+\frac{4}{15} \arctan \sqrt{2}+\frac{113 \ln 3}{210 \sqrt{2}}\right)$ \\[\tdownsingle em]
    \hline
    \end{tabular} 
\end{minipage}
\caption{Mean distance in various solids of unit volume, $\phi = (1+\sqrt{5})/2$ is the Golden ratio.}
\label{tab:meanlenall}
    \end{table}
\egroup

\subsection*{Normalised mean distance}
We could select normalisation in which $V_1(K) = 1$ rather than $\volK =1$. In order to express the normalised mean distance $\Gamma_{KK}$, we just rescale our values in Table $\ref{tab:meanlenall}$ by $\sqrt[3]{\volK}/V_1(K)$. Both $\vol K$ and $V_1(K)$ can be expressed easily. The following Table \ref{tab:VoluPlaton} shows the volume of the regular polyhedra with edge length equal to $l$.
\begin{table}[h]
    \centering
    \begin{tabular}{|c|c|c|c|c|c|}
    \hline
        $K$ & tetrahedron & cube & octahedron & dodecahedron & icosahedron\\
    \hline
       \ru{2} $\displaystyle\frac{\volK}{l^3}$ & $\displaystyle\frac{\sqrt{2}}{12}$ & $1$ & $\displaystyle\frac{\sqrt{2}}{3}$ & $\displaystyle\frac{15+7\sqrt{5}}{4}$ & $\displaystyle\frac{5\left(3+\sqrt{5}\right)}{12}$\\[2ex]
    \hline
    \end{tabular}
    \caption{First intrinsic volume of Platonic solids with unit edge length}
    \label{tab:VoluPlaton}
\end{table}
To express $V_1(K)$, we use the formula $V_1(K) = \frac{1}{2\pi} \sum_i l_i (\pi-\delta_i)$, where the sum is carried over all edges $E_i$ of $K$ having length $l_i$ and dihedral angle $\delta_i$. The following table shows the value of $V_1(K)$ for the five regular polyhedra (Platonic solids) with common edge length $l_i = l$ for all $i$.

\begin{table}[h]
    \centering
    \begin{tabular}{|c|c|c|c|c|c|}
    \hline
        $K$ & tetrahedron & cube & octahedron & dodecahedron & icosahedron\\
    \hline
       \ru{2} $\displaystyle\frac{V_1(K)}{l}$ & $\displaystyle\frac{3\arccos\left(-\frac13\right)}{\pi}$ & $6$ & $\displaystyle\frac{6\arccos\frac13}{\pi}$ & $\displaystyle\frac{15\arctan 2}{\pi}$ & $\displaystyle\frac{15\arcsin\frac23}{\pi}$\\[2ex]
    \hline
    \end{tabular}
    \caption{First intrinsic volume of Platonic solids with unit edge length}
    \label{tab:IntrinPlaton}
\end{table}

When $K$ is a ball, $\sqrt[3]{\volK}/V_1(K) = \frac{1}{6}\sqrt[3]{\frac{\pi}{6}}$ trivially. Finally, performing the scaling, in Table \ref{tab:NormalisedDist} we show numerical values of $\Gamma_{KK}$ for the same solids $K$ as in Table \ref{tab:meanlenall}. The lower and the upper bound of $\Gamma_{KK}$ for $K$ convex compact (based on \cite{bonnet2021sharp}) are set to $5/28$ and $1/3$, respectively.

\begin{table}[h]
    \centering
    \begin{tabular}{|c|c|c|c|c|}
    \hline
        $K$ & \begin{tabular}{c}
             lower  \\
             bound 
        \end{tabular} & tetrahedron & octahedron & cube\\
    \hline
       \ru{2} $\Gamma_{KK}$ & $0.17857143$ & $0.19601928$ & $0.21800285$ & $0.22056906$\\[2ex]
    \hline
    \multicolumn{5}{c}{}\\
    \hline
        $K$ & icosahedron & dodecahedron & ball & \begin{tabular}{c}
             upper  \\
             bound 
        \end{tabular}\\
    \hline
       \ru{2} $\Gamma_{KK}$ & $0.23872552$ & $0.23963024$ & $0.25714286$ & $0.33333333$\\[2ex]
    \hline
    \end{tabular}
    \caption{Normalised mean distance in Platonic solids with unit first intrinsic volume}
    \label{tab:NormalisedDist}
\end{table}

%\vfill

\section{Auxiliary integrals}
\subsection*{Reccurrence relations for auxiliary integrals}
Recall that $D(\zeta,\gamma)$ is the \emph{fundamental triangle domain} with vertices $[0,0]$, $[\zeta,0]$, $[\zeta,\zeta\tan\gamma]$ ($\zeta > 0$, $0<\gamma<\pi/2$). To express the integrals
\begin{equation}
I^{(p)}_{ij}(q,\gamma) =\int_{D(q,\gamma)} x^i y^j \left(1 + x^2 + y^2\right)^{p/2} \ddd x \dd y,
\end{equation}
we mainly employ recursive relations. However, $I^{(p)}_{11}(q,\gamma)$ can be expressed directly without recursions. We parametrise the domain $D(q,\gamma)$ as $y \in (0,x \tan \gamma)$, $x \in (0,q)$, by integrating out $y$ and then $x$, we get
\begin{equation}
I^{(p)}_{11}(q,\gamma) = \frac{\sin^2\gamma +\cos^2\gamma \left(1+q^2 \sec^2\gamma \right)^{2+\frac{p}{2}}-\left(1+q^2\right)^{2+\frac{p}{2}}}{(2+p) (4+p)}.
\end{equation}

\subsubsection*{K's}
In case of $I^{(p)}_{10}(q,\gamma)$ and  $I^{(p)}_{10}(q,\gamma)$, we cannot integrate twice. To overcome this, we first define our first auxiliary integral
\begin{equation}
K^{(p)}(r)= \int_0^r \left(1+t^2\right)^{1+p/2} \ddd t
\end{equation}
satisfying symmetry
\begin{equation}
\qquad K^{(p)}(-r) = -K^{(p)}(r)
\end{equation}
and, via integration by parts, the recurrence relation
\begin{equation}
K^{(p)}(r) = \frac{2+p}{3+p} K^{(p-2)}(r)+ \frac{r}{3+p}(1+r^2)^{1+p/2}
\end{equation}
with boundary conditions
\begin{equation}
K^{(-2)}(r) = r,\quad
K^{(-3)}(r) = \arcsinh r.
\end{equation}
We can then express our $I^{(p)}_{10}(q,\gamma)$ and $I^{(p)}_{10}(q,\gamma)$ in terms of $K's$ as
\begin{align}
    I_{10}^{(p)}(q,\gamma)& =\frac{1}{2+p}\left[\left(1+q^2\right)^{\frac{3+p}{2}} K^{(p)}\left(\frac{q \tan (\gamma
   )}{\sqrt{1+q^2}}\right)-\sin\gamma \, K^{(p)}(q \sec\gamma)\right], \\
   I_{01}^{(p)}(q,\gamma) & =\frac{1}{2+p}\left[ \cos \gamma K^{(p)}(q \sec\gamma)-K^{(p)}(q)\right].
\end{align}

\subsubsection*{J's}
We denote
\begin{equation}
J^{(p)}(q,\gamma)=-\gamma + \int_0^\gamma (1+q^2 \sec^2\varphi)^{1+p/2}\ddd \varphi,
\end{equation}
satisfying symmetry
\begin{equation}
    J^{(p)}(q,-\gamma) = -J^{(p)}(q,\gamma)
\end{equation}
and, via integration by parts, the recurrence relations
\begin{equation}
J^{(p)}(q,\gamma) = J^{(p-2)}(q,\gamma) + q\, (1+q)^{\frac{1+p}{2}}K^{(p-2)}\left(\frac{q \tan \gamma}{\sqrt{1+q^2}}\right),
\end{equation}
with boundary conditions
\begin{equation}
J^{(-2)}(q,\gamma) = 0, \quad
J^{(-3)}(q,\gamma) = -\gamma + \arcsin\frac{\sin\gamma}{\sqrt{1+q^2}}.
\end{equation}
\begin{remark}
Note that we can write $J^{(p)}(q,\gamma)=\int_0^\gamma \left( (1+q^2 \sec^2\varphi)^{1+p/2} - 1\right)\ddd \varphi$.
\end{remark}
We transform $I_{00}^{(p)}(q,\gamma)$ by substitution into polar coordinates $x = r \cos\varphi$, $y = r \sin \varphi$, our domain $D(\zeta,\gamma)$ becomes parametrised as $r \in (0,q \sec\varphi)$, $\varphi\in (0,\gamma)$ and thus
\begin{equation}
I^{(p)}_{00}(q,\gamma) = \int_0^\gamma \int_0^{q \sec\varphi} r \left(1 + r^2\right)^{p/2} \ddd r \dd \varphi,
\end{equation}
Integrating out $r$, we get
\begin{equation}
I_{00}^{(p)}(q,\gamma) = \frac{1}{2+p} \int_{0}^\gamma (1+q^2\sec^2\varphi)^{1+p/2} 
 - 1\ddd\varphi = \frac{1}{2+p}J^{(p)}(q,\gamma).
\end{equation}
Note that, by this integral formula, we can extend the definition of $I^{(p)}_{00}(q,\gamma)$ for negative $\gamma$ as well.

\subsubsection*{M's}
The last set of auxiliary integrals we define is
\begin{equation}
M^{(p)}(q,\gamma)=\int_0^\gamma \cos^2\varphi\,\left[ (1+q^2 \sec^2\varphi)^{1+p/2} - 1\right]\ddd \varphi,
\end{equation}
satisfying the recurrence relation
\begin{equation}
M^{(p)}(q,\gamma) = M^{(p-2)}(q,\gamma) + q^2\left(\gamma + J^{(p-2)}(q,\gamma)\right).
\end{equation}
Using standard techniques of calculus it is not hard to derive their specific values for $p=-2$ and $p=-3$,
\begin{equation}
M^{(-2)}(q,\gamma) = 0,\quad M^{(-3)}(q,\gamma) = \frac{1-q^2}{2} \arcsin\frac{\sin\gamma}{\sqrt{1+q^2}}-\frac{\gamma}{2}+\frac{\sin \gamma}{2}\left( \sqrt{\cos^2\gamma+q^2} - \cos\gamma\right).
\end{equation}

Finally, we can express $I^{(p)}_{20}(q,\gamma)$ and $I^{(p)}_{02}(q,\gamma)$. Note that we only need to express the former as $I_{02}^{(p)}(q,\gamma)$ can be extracted from other integrals since
\begin{equation}
    I_{00}^{(p)}(q,\gamma) + I_{20}^{(p)}(q,\gamma) + I_{02}^{(p)}(q,\gamma) = \int_{D(q,\gamma)}\left(1 + x^2 + y^2\right)^{1+p/2} \ddd x \dd y = I_{00}^{(p+2)}(q,\gamma).
\end{equation}
Again, by using the polar coordinates substitution, we transform the integral into
\begin{equation}
I^{(p)}_{20}(q,\gamma) = \int_0^\gamma \int_0^{q \sec\varphi} r^3 \cos^2\varphi \left(1 + r^2\right)^{p/2} \ddd r \dd \varphi,
\end{equation}
Integrating out $r$, we get
\begin{equation}
\begin{split}
I_{20}^{(p)}(q,\gamma) & = \int_{0}^\gamma \frac{\cos^2\varphi \left[\left(1+q^2 \sec^2\varphi\right)^{2+\frac{p}{2}}-1\right]}{4+p} -\frac{\cos^2\varphi \left[\left(1+q^2 \sec^2\varphi\right)^{1+\frac{p}{2}}-1\right]}{2+p}\ddd\varphi\\
& = \frac{1}{4+p} M^{(p+2)}(q,\gamma) - \frac{1}{2+p} M^{(p)}(q,\gamma).
\end{split}
\end{equation}
Selected values of the auxiliary integrals $I^{(p)}_{ij}(q,\gamma)$ can be found below in the next section.

\subsection*{Special values of auxiliary integrals}
The following Table \ref{tab:auxil} lists some of the values of $I^{(1)}_{ij}(q,\gamma)$ used throughout our paper.
\bgroup
\renewcommand{\arraystretch}{2.5}
\begin{table}[htb!]
    \centering
    \begin{tabular}{c|c}
        $I^{(1)}_{00}\left(1,\frac{\pi}{4}\right)$ & $\displaystyle \frac{1}{2 \sqrt{3}}-\frac{\pi }{36}+\frac{2}{3} \arccoth\sqrt{3}$\\
        $I^{(1)}_{00}\left(\frac{\sqrt{2}}{2},\frac{\pi}{3}\right)$ & $\displaystyle \frac{1}{4}-\frac{\pi}{36}+\frac{7 \arccoth\sqrt{2}}{12 \sqrt{2}}$\\
        $I^{(1)}_{00}\left(\frac{\sqrt{2}}{4},\frac{\pi}{3}\right)$ & $\displaystyle\frac{1}{16 \sqrt{2}}+\frac{\pi }{18}+\frac{25 \ln 3}{192 \sqrt{2}}-\frac{1}{3} \arccot\sqrt{2}$ \\
        $I^{(1)}_{00}\left(\frac{\sqrt{2}}{2},\frac{\pi}{4}\right)$ & $\displaystyle\frac{1}{6 \sqrt{2}}-\frac{\pi }{12}+\frac{7 \ln 3}{24 \sqrt{2}}+\frac{1}{3} \arccot\sqrt{2}$ \\
        $I^{(1)}_{00}\left(\frac{1}{2},\frac{\pi}{5}\right)$ & $\displaystyle \frac{\sqrt{5}}{16} -\frac{5}{48} +\frac{\pi }{60}-\frac{13 \ln 5}{192}-\frac{1}{6} \arccot 2+\frac{13}{96}\arcsinh 2$ \\        
        $I^{(1)}_{00}\left(\frac{1}{2},\frac{2\pi}{5}\right)$ & $\displaystyle \frac{\sqrt{5}}{16} -\frac{5}{48} +\frac{\pi }{20}-\frac{13 \ln 5}{192}-\frac{1}{6} \arccot 2+\frac{13}{96}\arcsinh 2$ \\        
        $I^{(1)}_{10}\left(1,\frac{\pi}{4}\right)$ & $\displaystyle \frac{\sqrt{3}}{8}-\frac{\arcsinh\sqrt{2}}{8 \sqrt{2}}+\frac{1}{2} \arccoth\sqrt{3}$\\
        $I^{(1)}_{10}\left(\frac{\sqrt{2}}{2},\frac{\pi}{3}\right)$ & $\displaystyle\frac{3}{16 \sqrt{2}}-\frac{\sqrt{3}}{16} \arcsinh\sqrt{2}+\frac{9}{32} \arccoth\sqrt{2}$ \\
        $I^{(1)}_{10}\left(\frac{\sqrt{2}}{4},\frac{\pi}{3}\right)$ & $\displaystyle \frac{3}{256}+\frac{81 \ln 3}{1024}-\frac{\sqrt{3}}{16} \arccoth\sqrt{3}$ \\
        $I^{(1)}_{01}\left(1,\frac{\pi}{4}\right)$ & $\displaystyle -\frac{7}{12 \sqrt{2}}+\frac{3 \sqrt{3}}{8}+\frac{\arcsinh\sqrt{2}}{8 \sqrt{2}}-\frac{1}{8} \arccoth\sqrt{2}$\\
        $I^{(1)}_{01}\left(\frac{\sqrt{2}}{2},\frac{\pi}{3}\right)$ & $\displaystyle\frac{3}{8}\sqrt{\frac{3}{2}}-\frac{\sqrt{3}}{8}-\frac{1}{8} \arccoth\sqrt{3}+\frac{1}{16} \arcsinh\sqrt{2}$ \\
        $I^{(1)}_{11}\left(1,\frac{\pi}{4}\right)$ & $\displaystyle \frac{1}{30}-\frac{4 \sqrt{2}}{15}+\frac{3 \sqrt{3}}{10}$\\
        $I^{(1)}_{11}\left(\frac{\sqrt{2}}{2},\frac{\pi}{3}\right)$ & $\displaystyle \frac{1}{20}-\frac{3}{20}\sqrt{\frac{3}{2}}+\frac{3 \sqrt{3}}{20}$\\
        $I^{(1)}_{20}\left(\frac{\sqrt{2}}{2},\frac{\pi}{3}\right)$ & $\displaystyle \frac{1}{40}+\frac{1}{20 \sqrt{3}}+\frac{\pi }{180}+\frac{13 \arccoth\sqrt{2}}{120 \sqrt{2}}$\\
        $I^{(1)}_{20}\left(\frac{\sqrt{2}}{4},\frac{\pi}{3}\right)$ & $\displaystyle \frac{1}{20 \sqrt{3}}-\frac{29}{640 \sqrt{2}}-\frac{\pi }{90}+\frac{43 \ln 3}{7680 \sqrt{2}}+\frac{1}{15} \arccot\sqrt{2}$
    \end{tabular}
    \caption{Selected values of $I_{ij}^{(1)}(q,\gamma)$ for various arguments}
    \label{tab:auxil}
\end{table}
\egroup

\subsection*{List of equivalent values}
Note that the values in Table \ref{tab:auxil} are get by not only recursions alone, but also with addition to the following rules (equivalent replacement rules). These rules are only aesthetic and have no effect on the correctness of our results.
\begin{align*}
\arcsin\frac{1}{\sqrt{3}} & \to \frac{\pi }{2}-\arctan\sqrt{2}\\
\arcsin \sqrt{\frac{2}{3}} & \to \frac{\pi}{2}-\arccot\!\sqrt{2} \\
\arcsinh 1 & \to \arccoth \sqrt{2} \\
\arcsinh\frac{1}{\sqrt{3}} & \to\frac{\ln 3}{2}\\
\arcsinh\frac{1}{\sqrt{2}} & \to \arccoth\sqrt{3}
\end{align*}

\end{document}